\documentclass[11pt,a4paper,reqno]{amsart}
\usepackage[utf8]{inputenc}
\usepackage[T1]{fontenc}
\usepackage{textcomp}  
\usepackage{lmodern}  
\usepackage[british]{babel}
\usepackage{csquotes}
\usepackage{stmaryrd}
\usepackage{psfrag}
\usepackage{perpage}
\usepackage{url}
\usepackage{color}
\usepackage{mathrsfs}

\usepackage{mathtools}
\numberwithin{equation}{section}

\usepackage[normalem]{ulem}

\usepackage{microtype}
\mathtoolsset{centercolon}
\usepackage{enumitem}
\usepackage{amssymb}
\usepackage{bm}

\usepackage[dvipsnames]{xcolor}
\usepackage{hyperref}
\hypersetup{
	colorlinks,
	linkcolor={red!50!black},
	citecolor={blue!50!black},
	urlcolor={blue!80!black}
}
\usepackage{amsthm}
\usepackage{geometry}
\usepackage{caption}
\usepackage{cleveref}

\usepackage{pgfplots}
\pgfplotsset{compat=1.14}

\setlist{itemsep=1pt,topsep=1pt,parsep=1pt}

\usepackage{pgf,tikz}
\usetikzlibrary{
  matrix, 
  arrows,
  arrows.meta,
  calc,
  intersections,
  external,
  patterns
}

\DeclareFontFamily{U}{stix2bb}{\skewchar\font127 }
\DeclareFontShape{U}{stix2bb}{m}{n} {<-> stix2-mathbb}{}
\DeclareMathAlphabet{\mathblackboard}{U}{stix2bb}{m}{n}

\newcommand*{\ind}{\bm{1}}

\newcommand*{\Var}{\operatorname{\mathbb{V}ar}}

\newcommand*\dd{\mathop{}\!\mathrm{d}}
\newcommand*{\EE}{\mathbb{E}}
\newcommand*{\bbE}{\mathbb{E}}
\newcommand*{\bE}{\mathbf{E}}
\newcommand*{\bP}{\mathbf{P}}
\newcommand*{\PP}{\mathbb{P}}

\DeclarePairedDelimiterXPP\V[1]{\mathbb{V}}{[}{]}{}{#1}
\newcommand*{\numberset}{\mathbb}
\newcommand*{\N}{\numberset{N}}
\newcommand*{\Z}{\numberset{Z}}
\newcommand*{\Zeven}{\numberset{Z}_{\mathrm{even}}}
\newcommand*{\R}{\numberset{R}}

\newcommand*{\tf}{\mathtt{F}}

\newcommand{\sumtwo}[2]{\sum_{\substack{#1 \\ #2}}} 

\newcommand{\gep}{\varepsilon}

\newcommand{\cV}{\mathcal{V}}
\newcommand{\cM}{\mathcal{M}}

\newcommand{\tPP}{\widetilde{\PP}}
\newcommand{\tEE}{\widetilde{\EE}}
\newcommand{\tVV}{\widetilde{\mathbb{V}}\mathrm{ar}}

\newcommand{\tCC}{\widetilde{\mathbb{C}}\mathrm{ov}}

\newcommand{\e}{\mathrm{e}}

\renewcommand{\V}{\mathbf{V}}

\newcommand*{\llb}{\llbracket}
\newcommand*{\rrb}{\rrbracket}

\renewcommand{\epsilon}{\varepsilon}
\renewcommand{\phi}{\varphi}

\usepackage{scalerel,stackengine}

\theoremstyle{plain}
\newtheorem{theorem}{Theorem}[section]

\newtheorem{lemma}[theorem]{Lemma}
\newtheorem{claim}[theorem]{Claim}
\newtheorem{proposition}[theorem]{Proposition}

\theoremstyle{definition}	
\newtheorem{assumption}{Assumption}

\theoremstyle{remark}
\newtheorem{remark}[theorem]{Remark}
\newtheorem{theoremalpha}{Theorem}

\newcommand{\blue}{\color{blue}}

\title[2D directed polymers with critical spatial correlations]{Free energy and phase transition for\\ 2D directed polymers with critical spatial correlations}

\author[Q. Berger]{Quentin Berger}
\address{Université Sorbonne Paris Nord, Laboratoire d'Analyse, Géométrie et Applications, CNRS UMR 7539, 99 Av. J-B Clément, 93436 Villetaneuse, France and Institut Universitaire de France}
\email{quentin.berger@math.univ-paris13.fr}

\author[F. Cottini]{Francesca Cottini}
\address{Sorbonne Université, Laboratoire de Probabilités, Statistique et Modélisation, CNRS UMR 8001, 4 place Jussieu, 75005 Paris, France}
\email{francesca.cottini@sorbonne-universite.fr}

\begin{document}

\begin{abstract}
We study the two-dimensional directed polymer model in a Gaussian environment which is independent in time and spatially correlated, with covariances $h(x)$ either summable or with a critical decay, satisfying $h(x) \sim (\log |x|)^a/|x|^2$ as $|x|\to\infty$ for some \(a>-1\). 
We determine the precise high-temperature asymptotics of the free energy, confirming a conjecture of Lacoin~\cite{Lac11} later refined in~\cite{CCD25}. 
We also establish a phase transition for the diffusively rescaled partition functions: below some critical point they converge to the Lebesgue measure, while above it they converge to zero. 
A key feature of our approach is that both results are obtained using only second-moment estimates and are based on a simplified change-of-measure argument in the supercritical regime, that may prove useful for other disordered models.
\end{abstract}


\maketitle


\section{Introduction and main results}

The directed polymer model is a widely studied model from statistical mechanics, introduced in the 1980's by Huse and Henley \cite{HH85}.
It has attracted the attention of physicists and mathematicians, both as a model for a polymer placed in a solvent with impurities, and because of its close connection to the stochastic heat equation and the KPZ equation.
We refer to \cite{Com17} for an extensive overview of the model and to \cite{Zyg24} for a review of recent results.

The model is defined as follows. 
Consider a simple random walk \((S_n)_{n\geq 0}\) on \(\Z^d\), \(d\geq 1\), whose directed trajectory \((n,S_n)_{n\geq 0}\) represents the polymer.
We denote by \(\bP_x\) is law when started from \(S_0=x\) and we write more simply \(\bP\coloneqq\bP_0\).
Consider also a field \(\omega \coloneqq (\omega(n,x))_{n\in \N, x\in \Z^d}\) of random variables, which represent the random environment; its law is denoted by \(\PP\).

Then, for a fixed realization of \(\omega\) (quenched disorder), we defined the following Gibbs measures on random walk paths: for \(N\in \N\) and \(\beta \geq 0\) (the inverse temperature), let
\begin{equation*}
  \frac{\dd \bP_{N}^{\beta}}{\dd \bP}(S) \coloneqq \frac{1}{Z_{N}^{\beta}} \exp\Big( \beta\sum_{n=1}^N \omega(n,S_n)  \Big) \,.
\end{equation*} 
In the above, \(Z_{N}^{\beta} = Z_{N}^{\beta}(\omega)\) is the partition function of the model, which normalizes \(\bP_{N}^{\beta}\) to a probability measure and can be written as \(Z_N^{\beta}= \bE[\exp( \beta\sum_{n=1}^N \omega(n,S_n))]\).

\subsection{Free energy and localization phase transition}

The present article considers an environment which possesses (critical) long-range spatial correlations, but let us first present an overview in the i.i.d.\ setting, which is the most commonly studied in the literature.

\subsubsection{With an i.i.d.\ disorder}
\label{sec:intro-iid}

Assume that the \((\omega(n,x))_{n\in \N, x\in \Z^d}\) are i.i.d.\ and that 
\begin{equation}
  \label{eq:integrability-omega}
  \EE[\omega(1,0)] =0, \quad \Var(\omega(1,0)) =1 \,,
  \qquad \lambda(\beta) \coloneqq \log \EE\big[\e^{\beta \omega(1,0)}\big] < \infty \quad \text{for all }\beta \in \R\,. 
\end{equation}
Then, we may define the normalized partition function
\begin{equation*}
  \label{def:W-0}
  W_N^{\beta} \coloneqq \frac{Z_N^{\beta}}{\EE[Z_N^{\beta}]}= \bE\bigg[\exp\Big( \sum_{n=1}^N  (\beta \omega(n,S_n)  -\lambda(\beta)) \Big) \bigg] \,.
\end{equation*}

An important quantity which encodes physical information on the model is the \textit{free energy}, defined as the \(\PP\)-a.s.\ and \(L^q\) limit 
\begin{equation}
  \label{def:free-energy}
  \tf(\beta) \coloneqq \lim_{N\to\infty} \frac{1}{N} \log W_N^{\beta} = \lim_{N\to\infty} \frac{1}{N} \EE \log W_N^{\beta} \,.
\end{equation}
The fact that the limit exists and is \(\PP\)-a.s.\ constant follows from a sub-additivity argument; we refer to \cite[Thm.~2.1]{Com17} and its proof for details.
One can show that \(\beta \mapsto \tf(\beta) +\lambda(\beta)\) is convex and that \(\beta\mapsto \tf(\beta)\) is non-positive and non-increasing. 
In particular, there is a critical point \(\beta_c\), which may be defined as
\[
  \beta_c \coloneqq \sup\{\beta \geq 0, \tf(\beta) =0\} = \inf\{\beta \geq 0, \tf(\beta) <0\} \,. 
\]

We stress that the free energy encodes some localization properties of the directed polymer model, see \cite{CH02,CH06,CSY03}.
Let us state for instance a result from \cite{CH02}, which shows that for a Gaussian environment \((\omega_{1,0}\sim \mathcal{N}(0,1))\), 
\begin{align*}
  \PP\text{-a.s.} \qquad
\tf'(\beta) = -
\beta \lim_{N\to\infty} \EE\bigg[(\bP_N^{\beta,\omega})^{\otimes 2}\bigg(\frac1N\sum_{k=1}^N \ind_{\{S_k^{(1)}=S_k^{(2)}\}}\bigg)\bigg] \,.
\end{align*}
In the above, \((\bP_N^{\beta,\omega})^{\otimes 2}\) denotes the law of two independent copies \((S_n^{(1)})_{n\geq 0}\), \((S_n^{(2)})_{n\geq 0}\) of trajectories under the polymer measure \(\bP_N^{\beta,\omega}\), with a fixed, common environment.
This is obtained by passing to the limit in \cite[Lem.~7.1]{CH02}, and is valid whenever \(\tf\) is differentiable at~\(\beta\), \textit{i.e.}\ at all but at most countably many~\(\beta\), by convexity.

It turns out that, in the i.i.d.\ setting, one can show that \(\beta_c =0\) in dimensions \(d=1\) and \(d=2\) (see \cite{CV06} and \cite{Lac10a}) and \(\beta_c>0\) in \(d\geq 3\) (see \cite{Bol89}).
In view of its physical meaning, it is also natural to investigate the behavior of the free energy close to criticality, that is when~\(\beta\downarrow \beta_c\): this as been done in dimension \(d=1\) (see e.g.\ \cite{AY15,Lac10a,Nak19}), in dimension \(d=2\) (see e.g.\ \cite{Lac10a,Nak14,BL17,BN26}), and in dimension \(d\geq 3\) with more recent breakthroughs (see~\cite{Lac25,JL26}).
Let us briefly summarize the existing results as follows:

\begin{theoremalpha} 
  \label{thm:iid}
  In the case of an i.i.d.\ environment satisfying~\eqref{eq:integrability-omega}, we have the following:
\begin{enumerate}[label=(\roman*)]
  \item 
  In dimension \(d=1\), we have \(\beta_c=0\) and \cite{Nak19} shows that
  \[
  \tf(\beta) \sim - \frac16 \, \beta^4 \quad \text{ as } \beta\downarrow 0 \,. 
  \]

  \item 
  In dimension \(d=2\), we have \(\beta_c=0\) and \cite{BCT25,BN26} show that there is a constant \(c>0\) such that we have (we assume here that \(\EE[\omega(1,0)^3]=0\) to simplify the statement)
  \[
    -\, \frac{1}{c}\, \exp\Big( -\frac{\pi}{\beta^2} \Big) \leq \tf(\beta) \leq   -\, c \,\exp\Big( -\frac{\pi}{\beta^2} \Big) \quad \text{ as } \beta\downarrow 0 \,.
  \]

  \item 
  In dimension \(d\geq 3\), we have \(\beta_c>0\) and \cite{JL26} shows that 
  \[
    -\, \exp\Big(- u^{- \frac{2+d}{8 d} +o(1)} \Big) \leq \tf(\beta_c +u) \leq -\, \exp\Big(- u^{-1+o(1)} \Big) \quad \text{ as } u\downarrow 0 \,,
  \]
  where the upper bound holds for Gaussian disorder.
\end{enumerate}
\end{theoremalpha}

\subsubsection{With spatially correlated disorder}

In the present article, we consider the case where the environment is not i.i.d.\ anymore, but presents long-range correlations in space, while remaining independent in time.
More precisely, we make the following assumption, which is the setting considered in \cite{CCD25}.

\begin{assumption}
  \label{main-assumption}
  The environment \(\omega =(\omega(n,x))_{n\in \N, x\in \Z^d}\) is a centered Gaussian field, with covariances given by
  \[
  \EE[\omega(n,x) \omega(m,y)] = h(x-y) \ind_{\{n=m\}} \,,
  \]
  with \(h:\Z^2\to\R_+\) a bounded function verifying \(h(0)=1\).
  We additionally assume that \(h=h_0\ast h_0 \) for some \(h_0 :\Z^d\to \R_+\), so that \(\omega(n,x)\) can be written as
  \begin{equation}
    \label{omega-hat-omega}
    \omega(n,x) = \sum_{y \in \Z^2} h_0(x-y) \,\hat{\omega}(n,y)\,, \qquad (n,x) \in \N \times \Z^2\,,
  \end{equation}
  where $\hat{\omega}=(\hat{\omega}(n,y))_{(n,x) \in \N \times \Z^2}$ are i.i.d.\ standard Gaussian random variables.
\end{assumption}

\noindent
We stress that~\eqref{eq:integrability-omega} still holds, with \(\lambda(\beta) = \frac12 \beta^2\).
In fact, all definitions of Section~\ref{sec:intro-iid} remain valid.
For instance, the free energy \(\tf(\beta)\) from~\eqref{def:free-energy} is still well-defined.

Directed polymers in Gaussian environment with spatial correlations have been considered by Lacoin in~\cite{Lac11}, in a space-time continuous setting, with power-law correlations \(h(x) \asymp \|x\|^{-\theta}\) for some \(\theta>0\). 
(Here \(h(x)\asymp g(x)\) means that there is a constant \(c>0\) such that \(\frac{1}{c} g(x) \leq h(x) \leq c g(x)\) for all \(x\).)
In a nutshell, \cite{Lac11} shows that long-range spatial correlations modify the behavior of the model when \(\theta<2\wedge d\), whereas it should behave similarly to the i.i.d.\ case when \(\theta > 2\wedge d\).

Let us state for instance~\cite[Thm.~1.3]{Lac11}, which shows how Theorem~\ref{thm:iid} is modified by the presence of spatial correlations with power-law decay \(\theta < 2\wedge d\).

\begin{theoremalpha}
  \label{thm:Lacoin}
  Suppose that Assumption~\ref{main-assumption} holds with \(h(x) \asymp \|x\|^{-\theta}\), for some \(\theta <2\wedge d\).
  Then, in any dimension \(d\geq 1\), we have that \(\beta_c=0\) and \(\tf(\beta) \asymp - \beta^{4/(2-d)}\) as \(\beta\downarrow 0\).
\end{theoremalpha}

Lacoin states in~\cite[Rem.~1.5]{Lac11} that in the case of \textit{critical} spatial correlations (\textit{i.e.}\ \(\theta=2\)), the free energy should behave as \(-\exp( - \frac{cst.}{\beta})\) in dimension \(d=2\) and as \(-\exp(-\frac{cst.}{\beta^2})\) in dimension \(d\geq 3\) as \(\beta\downarrow 0\).
These cases have been left open since then.
In a simultaneous and independent work \cite{CRW26}, Cao, Rang and Wu study the same Brownian directed polymer as Lacoin, with product-type and radial spatial correlations. 
In particular, in the case of critical spatial correlations $\theta = 2$, their results confirm the aforementioned predictions of Lacoin in dimension \(d=2\), up to multiplicative constants in the exponential rate.
In dimension \(d\geq 3\), Lacoin's prediction in that case is in fact wrong; we discuss this further in Section~\ref{sec:d-3} below.

In the present article, we focus on the case of the critical dimension \(d=2\), with \textit{critical} spatial correlations (\textit{i.e.}\ \(\theta=2\)), and we fully answer Lacoin's question in that case, additionally identifying the correct constant in the exponential rate.
In fact, we treat a more general setting than having pure power correlations.
We consider the following two cases:
\begin{itemize}
  \item Correlations are summable (this includes the i.i.d.\ case or the correlated case with \(h(x)\asymp |x|^{-\theta}\) with \(\theta>2\)),
  \begin{equation}
    \label{finite-sum}
    \Sigma_h \coloneqq \sum_{x\in \Zeven^2} h(x) <+\infty\,,
    \tag{Sum.}
  \end{equation}
  where we have set \(\Zeven^2 \coloneqq \{ z=(z_1,z_2) \in \Z^2 \,,\, z_1+z_2 \in 2 \Z\}\). 
  \item Correlations are not summable and they have the following asymptotic behavior:
  \begin{equation}
  \label{def:h}
  \text{ for some } a>-1, \qquad
  h(x) \sim \frac{(\log |x|)^a}{|x|^2} \quad \text{ as } |x|\to\infty \,.
  \tag{$\mathrm{Log}^a$}
  \end{equation}
\end{itemize}
Let us stress that the case of summable correlations is very similar to the i.i.d.\ case, and that the core of the article concerns the non-summable case~\eqref{def:h}. 
We have kept the summable case for completeness since our techniques apply there as well (in fact many technical statements greatly simplify); it also serves as a point of comparison with the non-summable case.

Assuming either~\eqref{finite-sum} or~\eqref{def:h}, we identify the correct exponent of \(\beta\) in the exponential decay of the free energy, together with the correct constant.
This fully answers the question of \cite{Lac11} and the first conjecture in~\cite[\S7]{CCD25}.

\begin{theorem}
  \label{thm:free-energy}
  Suppose that Assumption~\ref{main-assumption} holds. 
  Then, as \(\beta\downarrow 0\), we have that
  \[
   \tf(\beta) = 
   \begin{cases}
    - \exp\Big( - (1+o(1))\, \mathscr{C}_h \, \beta^{-2}  \Big)  &\quad\text{ if \eqref{finite-sum} holds,}\\
    - \exp\Big( - (1+o(1))\, \mathscr{C}_a\, \beta^{-\frac{2}{2+a}}  \Big)  & \quad\text{ if \eqref{def:h} holds.}
   \end{cases} 
  \]
  The constants \(\mathscr{C}_h\) and \(\mathscr{C}_a\) are explicit, given by $\mathscr{C}_h = \frac{\pi}{\Sigma_h}$ and \(\mathscr{C}_a = ((a+2) 2^{(a-1)/2} \,z_a)^{2/(a+2)}\), where \(z_a\) is the first zero of the following Bessel function of the first kind
  \begin{equation}
    \label{eq:Bessel-function}
    J_{\alpha} (z) \coloneqq \sum_{n=0}^{\infty} \frac{(-1)^n}{\Gamma(n+\alpha+1) n!}\Big(\frac{z}{2}\Big)^{2n+\alpha} \qquad \text{ with } \alpha \coloneqq \frac{a+1}{a+2} -1 >-1 \,.
  \end{equation}
\end{theorem}

\begin{remark}
  Note that the pure power case \(h(x)\sim c |x|^{-2}\) corresponds to \(a=0\) (or \(\alpha=-1/2\)), in which case \(J_{-1/2} (z) = \sqrt{\frac{2}{\pi}} \cos(z)\) and thus \(z_{0}=\pi/2\).
\end{remark}

Let us stress here that, aside from the result in itself, our proof's techniques have their own interest.
Indeed, our method only relies on second moment estimates and somehow reduces the proof of Theorem~\ref{thm:free-energy} to showing that the second moment exhibits a phase transition (see Theorem~\ref{thm:second-moment} below); the constants \(\mathscr{C}_h\), \(\mathscr{C}_a\) are closely connected to the critical point of this phase transition.
In other words, Theorem~\ref{thm:free-energy} may be seen as a proof of concept that \emph{second moments are enough to obtain sharp estimates on the free energy in the critical dimension}, and we expect our method to be broadly applicable to other disordered systems, such as long-range directed polymers or disordered pinning models.
We refer to Section~\ref{sec:comments-sec-moment} for further discussion.

\begin{remark}
  We focus here on the Gaussian setting for simplicity, but let us mention \cite{CG23,CR26,Rang20}, which consider the scaling limit of one-dimensional polymers with (subcritical) correlated environments defined as an auto-regressive moving average of the type~\eqref{omega-hat-omega}.
\end{remark}

\subsection{Intermediate disorder in the critical dimension \texorpdfstring{\(d=2\)}{}}

In cases where \(\beta_c=0\), one may consider the so-called \textit{intermediate disorder regime}, whose study has been initiated in the case of an i.i.d.\ environment in dimension \(d=1\) in \cite{AKQ14a}.
The idea is to consider an inverse temperature \(\beta_N \downarrow 0\) at some appropriate rate in order to obtain a non-trivial limit for the normalized partition functions \((W_N^{\beta_N}(x))_{N\geq 0}\).
The limit can then be interpreted as a notion of solution for the Stochastic Heat Equation: we refer to \cite{CSZ25} for an overview of recent results, in particular in dimension \(d=2\), which is the critical dimension of the model.

\subsubsection{Intermediate disorder scaling and phase transition for the second moment}

In dimension \(d=2\), let us introduce the correct intermediate disorder regime, under Assumption~\ref{main-assumption} with either~\eqref{finite-sum} (which includes the i.i.d.\ case) or~\eqref{def:h}.

\begin{assumption}[Intermediate disorder scaling]
  \label{hyp:scaling}
  Let \(\hat \beta>0\).
  \begin{enumerate}[label=(\roman*)]
    \item Under~\eqref{finite-sum}, we let \((\beta_N)_{N\ge1}\) be defined by
    \begin{equation}
      \label{def:beta-N-summable}
      \beta_N \coloneqq \hat \beta \, \frac{\mathfrak{C}_h}{\sqrt{\log N}}
      \qquad \text{ with } \quad \mathfrak{C}_h \coloneqq (\pi /\Sigma_h)^{1/2} \,.
    \end{equation}

    \item Under~\eqref{def:h}, we let \((\beta_N)_{N\ge1}\) be defined by
    \begin{equation}
      \label{def:beta-N}
      \beta_N \coloneqq  \hat \beta \frac{\mathfrak{C}_a}{(\log N)^{(a+2)/2}} \qquad \text{ with }\quad \mathfrak{C}_a \coloneqq (a+2) 2^{(a-1)/2} \,.
    \end{equation}
  \end{enumerate}
\end{assumption}

Then, with the scaling given by Assumption~\ref{hyp:scaling}, one observes a phase transition in \(\hat \beta\) already at the level of second moments.
We collect here results from \cite{CSZ17b} in the i.i.d.\ setting and from~\cite{CCD25} under \eqref{def:h}.
The general case of summable correlations~\eqref{finite-sum} does not differ much from the i.i.d.\ case but is not treated in the literature: we include a proof in Appendix~\ref{app:sec-mom-summable}.

\begin{theoremalpha}
  \label{thm:second-moment}
  Suppose that Assumption~\ref{main-assumption} holds and that \((\beta_N)_{N\geq 0}\) is given as in Assumption~\ref{hyp:scaling}.
  
  \begin{enumerate}[label=(\roman*)]
    \item Under~\eqref{finite-sum}, we have that
    \begin{equation}
      \label{eq:limit-second-moment}
      \lim_{N\to\infty} \EE\Big[ (W_N^{\beta_N})^2 \Big] =
      \begin{cases}
        (1-\hat\beta^2)^{-1} =: \exp(\varsigma^2(\hat \beta)) & \quad \text{ if } \hat \beta<1 \,,\\
        +\infty & \quad \text{ if } \hat \beta \geq 1 \,.\\
      \end{cases}
    \end{equation}

    \item Under~\eqref{def:h}, letting \(\tilde{J}_{\alpha}(z) \coloneqq \Gamma(\alpha+1) (z/2)^{-\alpha} J_{\alpha}(z)\) be a modified Bessel function (recall~\eqref{eq:Bessel-function}) and \(z_a\) is its first zero, we have that
    \begin{equation}
      \label{eq:limit-second-moment-2}
      \lim_{N\to\infty} \EE\Big[ (W_N^{\beta_N})^2 \Big] =
      \begin{cases}
        \tilde{J}_{\alpha}(\hat\beta)^{-1} =: \exp(\varsigma^2(\hat \beta)) & \quad \text{ if } \hat \beta<z_a \,,\\
        +\infty & \quad \text{ if } \hat \beta \geq z_a \,.\\
      \end{cases}
    \end{equation}
  \end{enumerate}
In particular, there is a phase transition for the second moment at the critical point
\begin{equation}
  \label{def:critical-beta}
  \hat \beta_c= 1 \quad \text{under \eqref{finite-sum}} 
  \qquad \text{ and } \qquad  
  \hat \beta_c=z_a \quad \text{under \eqref{def:h}} .
\end{equation}
\end{theoremalpha}

\begin{remark}
  It turns out that the phase transition for the second moment also corresponds to that of the point-to-plane partition function. 
  Indeed, either in the i.i.d.\ case (see~\cite{CSZ17b,CC22,CD25}) or in the correlated case~\eqref{def:h} (see \cite{CCD25}), we have 
  \begin{equation}
    \label{eq:log-normal}
    \log W_N^{\beta_N} \xrightarrow{(d)} \mathcal{N}(-\tfrac12 \varsigma^2(\hat \beta),\varsigma(\hat \beta)) \ \ \text{ if } \hat\beta <\hat\beta_c \,,
    \quad \text{ and } \quad 
    W_N^{\beta_N} \xrightarrow{(d)} 0 \ \ \text{ if } \hat\beta \geq \hat\beta_c \,.
  \end{equation}
\end{remark}

\subsubsection{Phase transition for diffusively averaged partition functions}

Let us now introduce the normalized partition function with starting point \(x\)
\begin{equation*}
  \label{def:W}
  W_N^{\beta}(x) \coloneqq \bE_x\bigg[\exp\Big( \sum_{n=1}^N  (\beta \omega(n,S_n)  -\lambda(\beta)) \Big) \bigg] \,.
\end{equation*}
We consider the field of diffusively rescaled partition functions \((W_N^{\beta_N}( \llb x \sqrt{N} \rrb))_{x\in \R^2}\), where \(\llb\cdot \rrb\) maps points of \(\R^2\) to its nearest neighbor in \(\Zeven^2 \coloneqq \{ z=(z_1,z_2) \in \Z^2 \,,\, z_1+z_2 \in 2 \Z\}\). 
More precisely, for any function \(\varphi \in L_1(\R^2)\), we let
\begin{equation}
  \label{def:W-phi-1}
  W_N^{\beta_N} (\varphi) \coloneqq \int_{\R^2} \varphi(z) W_N^{\beta_N}\big(\llb z\sqrt{N}\rrb\big) \dd z \,.
\end{equation}
Notice that by Fubini's theorem we have that \(\EE[W_N^{\beta_N} (\varphi)]=\int_{\R^2} \varphi(z) \dd z\).
Let us note that we may also write
\begin{equation}
  \label{def:W-phi-2}
  W_N^{\beta_N} (\varphi) = \sum_{x\in \Zeven^2} \varphi_N(x) W_N^{\beta_N}(x)  \,, \quad \text{ with } \varphi_N(x) = \frac{N}{2} \int_{|z|_1\leq \frac{1}{\sqrt{N}}} \varphi\Big(\frac{x}{\sqrt{N}} + z \Big) \dd z \,.
\end{equation}
In other words, we consider \((W_N^{\beta_N}(x))_{x\in \Zeven^2}\), diffusively rescaled, as a random measure on~\(\R^2\).
A reason for considering this random distribution is that it can be interpreted as a discretization of the stochastic heat equation; we refer to~\cite{CSZ24-rev}, in particular Section~1.3, for more details.

Though the log-normality of~\eqref{eq:log-normal} provides the precise behavior of single partition functions \(W_N^{\beta_N}(x)\), it does not give useful information on the diffusively rescaled partition functions.
However, our second main result shows that the phase transition identified in Theorem~\ref{thm:second-moment} for the second moment also holds at the level of the diffusively averaged partition functions, at the same critical point \(\hat\beta_c\) from~\eqref{def:critical-beta}.

\begin{theorem}
  \label{thm:fluctuations}
  Suppose that Assumption~\ref{main-assumption} holds with either~\eqref{finite-sum} or~\eqref{def:h} and that the scaling of \((\beta_N)_{N\geq 1}\) is as in Assumption~\ref{hyp:scaling}.
  Then, we have 
  \[
  \forall \varphi \in \mathcal{C}_c^{\infty}(\R^2), \qquad
  W_N^{\beta_N} (\varphi)  \xrightarrow[\;N\to\infty\;]{(d)}
  \begin{cases}
    \int_{\R^2} \varphi(z) \dd z  & \quad \text{if } \hat\beta<\hat\beta_c  \text{ (\textit{sub-critical})},\\
    0  & \quad \text{if } \hat\beta>\hat\beta_c    \text{ (\textit{super-critical})}.
  \end{cases}
  \]
\end{theorem}

\noindent
We leave the case \(\hat \beta = \hat\beta_c\) in~Theorem~\ref{thm:fluctuations} open, and we discuss it further in Section~\ref{sec:comments-beta-c}.

\subsection{Further comments}
\label{sec:comments}

Let us now discuss in more details some aspects of our results.

\subsubsection{About the scaling of Assumption~\ref{hyp:scaling}}

First of all, let us note that a simple replica-trick calculation, see \cite[Lem.~2.1]{CCD25}, shows that for \(x,y \in \Zeven^2\)
\begin{equation}
  \label{eq:covariances}
  \EE\big[ W_{N}^{\beta}(x) W_N^{\beta}(y) \big] = \bE_x \otimes \bE_y \Big[\e^{\beta_N^2 \mathcal{L}_N(S,S')} \Big]\,,
\end{equation}
where \(\mathcal{L}_N(S,S') \coloneqq \sum_{n=1}^N h(S_n-S_n')\), with \(S,S'\) two independent copies of the simple random walk.
Note that in the i.i.d.\ case, \(\mathcal{L}_N(S,S')\) is simply the overlap \(\sum_{n=1}^N \ind_{\{S_n=S_n'\}}\) between the two random walks \(S,S'\) up to time \(N\).

Then, a straightforward calculation (details are given in Appendix~\ref{sec:scaling-L}) shows that 
\begin{equation}
  \label{eq:scaling-L}
  \bE^{\otimes 2}[\mathcal{L}_N(S,S')] \sim 
  \begin{cases}
    C_h \log N & \quad \text{under~\eqref{finite-sum},}\\
    C_a (\log N)^{a+2}& \quad \text{under~\eqref{def:h}},
  \end{cases}
\end{equation}
with \(C_h \coloneqq\frac{\Sigma_h}{\pi}\) in the first case and \(C_a \coloneqq\frac{2^{-(a+1)}}{(a+1)(a+2)}\) in the second case. 

This shows that the ``first order'' calculation~\eqref{eq:scaling-L} provides the correct guess on how to scale~\(\beta_N\) in~\eqref{eq:covariances} to obtain a non-trivial limit, and indeed corresponds to the scaling chosen in Assumption~\ref{hyp:scaling}.
Let us stress here that the constant \(C_h\coloneqq\frac{\Sigma_h}{\pi} = (\mathfrak{C}_h)^{-2}\) somehow captures the correct critical point \(\hat\beta_c\) under \eqref{finite-sum}, whereas the constant \(C_a\coloneqq\frac{2^{-(a+1)}}{(a+1)(a+2)}\) \textit{does not capture} the correct critical point \(\hat\beta_c\) for non-summable correlations~\eqref{def:h}.
The correct critical points are however related to the convergence in distribution 
\[
\frac{\mathcal{L}_N(S,S')}{\bE^{\otimes 2}[\mathcal{L}_N(S,S')]} \xrightarrow[\;N\to\infty\;]{(d)} Y \,,
\]
where \(Y \sim \mathrm{Exp}(1)\) in the summable case~\eqref{finite-sum} (see Remark~\ref{rem:convergence}) and \(Y\) is  the first hitting time of \(1\) of a Bessel process of dimension \(2\alpha+ 2\) started at \(0\) in the non-summable case~\eqref{def:h} (see \cite[Thm.~1.4]{CCD25}).
In view of~\eqref{eq:covariances}, the critical value \(\hat\beta_c\) for the second moment is then captured by the critical value \(\lambda_c \coloneqq \sup\{\lambda, \bE[\e^{\lambda Y}] <+\infty\}\) for the Laplace transform of~\(Y\).

\subsubsection{On the proofs: second moments are enough}
\label{sec:comments-sec-moment}

We believe that a central interest of our paper is that our proofs rely \emph{only on second moment estimates}. 
Our results may be divided into two parts: (i)~lower bounds on the free energy and the subcritical case of Theorem~\ref{thm:fluctuations}; (ii)~upper bounds on the free energy and the super-critical case of Theorem~\ref{thm:fluctuations}.

\smallskip
(i)~The first set of results in fact relies on second moment upper bounds in the subcritical regime \(\hat\beta<\hat\beta_c\) from Theorem~\ref{thm:second-moment}.
The proofs follow some standard path and essentially only uses the fact that 
\begin{equation}
  \label{eq:sup-covariances}
  \sup_{x,y\in \Z^2} \EE\big[ W_{N}^{\beta_N}(x) W_N^{\beta_N}(y) \big] <+\infty \qquad \text{ if } \hat{\beta}<\hat{\beta}_c \,.
\end{equation}
This bound is easy in the summable case~\eqref{finite-sum}, see~\eqref{eq:covariances-summable}, and follows from \cite[Eq.~(15)]{CCD25} in the non-summable case~\eqref{def:h}.
Details on how the subcritical case of Theorem~\ref{thm:fluctuations} and lower bounds on the free energy derive from~\eqref{eq:sup-covariances} are provided in Sections~\ref{sec:subcriticalpart} and~\ref{sec:lower-bound} below.

\smallskip
(ii)~The second set of results constitutes the most difficult and technical part of the article.
In Section~\ref{sec:key-super-critical}, we explain how the upper bounds on the free energy reduce to a single ``super-critical'' statement which slightly generalizes Theorem~\ref{thm:fluctuations}, see Theorem~\ref{thm:key-theorem}.
The proof of this result then follows a well-established strategy based on a change of measure argument, but we stress that our change of measure uses a clever choice of a statistic which allows us to only use second moment estimates (details of the strategy are given in Section~\ref{sec:key-proof}).
We organize the proof to highlight that Theorem~\ref{thm:key-theorem} essentially derives from (\eqref{eq:sup-covariances} and) the following estimate: for any \(b>0\)
\begin{equation}
  \label{eq:divergence}
  \lim_{N\to\infty} (\log N)^{-b}\, \EE\big[ (W_{N}^{\beta_N})^2 \big] =+\infty \qquad \text{ if } \hat{\beta}>\hat{\beta}_c \,.
\end{equation}
The fact that the second moment diverges at \(\hat\beta =\hat\beta_c\) is standard (it follows from Theorem~\ref{thm:second-moment} together with the monotonicity of \(\beta \mapsto \EE[ (W_{N}^{\beta})^2]\)), but what~\eqref{eq:divergence} says is that the divergence is ``faster than any power of \(\log N\)''.
Let us mention that we actually need some variant of~\eqref{eq:divergence} with a ``truncated'' partition function: the precise statement that we use is given in Claim~\ref{claim:lower-bound}.

\subsubsection{About the critical case \texorpdfstring{\(\hat\beta =\hat\beta_c\)}{}}
\label{sec:comments-beta-c}

We have left open the critical case \(\hat\beta =\hat\beta_c\) in Theorem~\ref{thm:fluctuations}, but this regime is now understood for i.i.d.\ disorder.
In the i.i.d.\ setting, there is a \textit{critical window} for the parameter \(\beta_N\) around \(\frac{\sqrt{\pi}}{\sqrt{\log N}}\) (recall Assumption~\ref{hyp:scaling}, with \(\hat\beta_c=1\)).
Let us now state the main result of \cite{CSZ23}, in a slightly simpler form to highlight the analogy with Theorem~\ref{thm:fluctuations}. (The super-critical regime is proven in~\cite{BCT25}.)

\begin{theoremalpha}
  \label{thm:SHF}
  Assume that the disorder \(\omega\) is i.i.d.\ and define \((\beta_N)_{N\geq 0}\) by
  \begin{equation}
    \label{def:critical-window}
    \beta_N  = \frac{\sqrt{\pi}}{\sqrt{\log N}} \Big(1+ \frac{\vartheta_N}{\log N}\Big) \,,
  \end{equation}
  for some sequence \((\vartheta_N)_{N\geq 1}\) of real numbers.
  Then, we have that 
  \[
  \forall \varphi \in \mathcal{C}_c^{\infty}(\R^2), \qquad
  W_N^{\beta_N} (\varphi)  \xrightarrow[\;N\to\infty\;]{(d)}
  \begin{cases}
    \int_{\R^2} \varphi(z) \dd z  & \quad \text{if } \vartheta_N \to-\infty  \text{ (\textit{sub-critical})},\\
    \int_{\R^2} \varphi(z) \mathscr{Z}_1^{\vartheta}(\dd z) & \quad \text{if } \vartheta_N \to \vartheta \in \R  \text{ (\textit{critical})},\\
    0  & \quad \text{if } \vartheta_N \to +\infty   \text{ (\textit{super-critical})}.
  \end{cases}
  \]
  In the above, \((\mathscr{Z}_s^{\vartheta})_{s\geq 0}\) is measure-valued process, namely (a marginal of) the critical \(2d\) Stochastic Heat Flow (SHF) constructed by Caravenna, Sun and Zygouras~\cite{CSZ23}. 
  We also refer to \cite{Tsai24} for another construction of the SHF through a mollification of the stochastic heat equation.
\end{theoremalpha}

In particular, Theorem~\ref{thm:SHF} identifies the correct critical window~\eqref{def:critical-window} and shows the convergence of the diffusively rescaled partition functions to the SHF, which is some universal, non-Gaussian random measure on \(\R^2\).
Since its introduction in~\cite{CSZ23}, the SHF has been the object of an intense activity, and we refer to \cite{CSZ25} for some account of recent developments.

As far as the case of a spatially correlated noise is concerned, we expect that Theorem~\ref{thm:SHF} still holds with the same SHF limit even with non-summable critical correlations, for \((\beta_N)_{N\geq 1}\) in some well-chosen critical window.
Some indication in this direction comes from~\cite{BCC26}, where the Edwards--Wilkinson fluctuations of $W_N^{\beta_N}(\varphi)$ are studied in the subcritical regime \(\beta<\hat\beta_c\) in the non-summable case~\eqref{def:h}. 
In particular, the covariances at starting points at distance of order $\sqrt{N}$ (the relevant scale for $W_N^{\beta_N}(\varphi)$) have the same asymptotic behavior as in the i.i.d.\ setting.
This suggests that the effect of spatial correlations vanishes at the diffusive scale. 
We may therefore expect that, within a critical window around \(\hat\beta_c \mathfrak{C}_a (\log N)^{- (a+2)/2}\), the second and higher moments of the diffusively scaled partition functions have the same limiting kernels as in the i.i.d.\ case, suggesting that the convergence to the same SHF limit as in Theorem~\ref{thm:SHF} holds.

\subsubsection{The case of dimension \texorpdfstring{\(d\geq 3\)}{}}
\label{sec:d-3}

We have here focused on the case of the critical dimension \(d=2\) and we have left aside the case of dimension \(d\geq 3\).
Let us mention that the Stochastic Heat Equation (SHE) in \(d\geq 3\) with spatially correlated noise has attracted some attention in recent years, in particular in the case of critical correlations \(h(x) \sim |x|^{-2}\), see e.g.\ \cite{MT04,KT24,DHL25}.

In their seminal work, Mueller and Tribe~\cite{MT04} showed the existence of a measure-valued solution to the SHE in dimension \(d\geq 3\) with critical correlations \(h(x)=|x|^{-2}\), in some high-temperature regime.
If we translate their result in the directed polymer context, we get the following: in dimension \(d\geq 3\) with spatial correlations verifying \(h(x) \sim |x|^{-2}\) as \(|x|\to \infty\), there exists some \(\tilde{\beta}_c>0\) such that:
\begin{equation}
  \label{eq:Tribe-Mueller}
  \forall \varphi \in \mathcal{C}_c^{\infty}(\R^2), \qquad
  W_N^{\beta} (\varphi)  \xrightarrow[\;N\to\infty\;]{(d)}
  \int_{\R^2} \varphi(z) \mathscr{Z}_1^{\beta}(\dd z) \qquad \text{ if } \beta<\tilde{\beta}_c\,,
\end{equation}
where \((\mathscr{Z}^{\beta}_s)_{s\geq 0}\) is the measure-valued process constructed in~\cite{MT04}.
Note here that \textit{no normalization} is needed for the inverse-temperature \(\beta\): in particular, \eqref{eq:Tribe-Mueller} holds for any fixed (small enough) \(\beta\).
We also stress that~\eqref{eq:Tribe-Mueller} is not a consequence of~\cite{MT04} but follows from~\cite{CG26}.

Let us point out (again by~\cite{CG26})  that, even though one can show that \(W_N^{\beta} \to 0\) for any \(\beta>0\), the convergence~\eqref{eq:Tribe-Mueller} proves that the free energy \(\tf(\beta)\) is equal to~\(0\) for all \(\beta<\tilde{\beta}_c\), since one can show that \(W_N^{\beta} (\varphi)\to 0\) whenever \(\tf(\beta)<0\).
In particular,~\eqref{eq:Tribe-Mueller} disproves Lacoin's conjecture that the free energy behaves as \(-\exp(-\frac{cst.}{\beta^2})\) as \(\beta\downarrow 0\), since the free energy critical point is \(\beta_c\geq \tilde{\beta}_c>0\).
This is indeed what is proven in \cite[Thm.~1.2-(i)]{CRW26} in the continuous setting of Lacoin~\cite{Lac11}.
It would be interesting to determine under which condition on the decay of correlations the critical point for the free energy (or for~\eqref{eq:Tribe-Mueller}) would indeed be equal to~\(0\).

\smallskip
Finally, let us mention the work~\cite{DHL25}, which considers the (non-linear) SHE in dimension \(d\geq 3\) with critical correlations \(h(x)\sim |x|^{-2}\), in some intermediate disorder regime.
More precisely, if we translate the main result of~\cite{DHL25} to the directed polymer context, we get the following: setting 
\[
\beta_N \coloneqq \frac{\hat\beta}{\sqrt{\log N}} \,,
\]
there is some \(\hat{\beta}_0\) such that for \(\hat\beta<\hat{\beta}_0\) we have that \(W_N^{\beta_N} (x)\) converges in distribution to some explicit limit (expressed in terms of a suitable forward-backward stochastic differential equation).
This has the flavor of the log-normality~\eqref{eq:log-normal} but \cite[Rem.~2.1]{DHL25} actually stresses that \(\hat{\beta}_0\) \textit{should not be a critical threshold}.
In fact \cite[Thm.~1.5]{DHL25} shows that \(\sup_{N\geq 0} \EE[(W_N^{\beta_N})^2] <+\infty\) for all \(\hat\beta \in (0,\infty)\), and there is therefore \textit{no phase transition for the second moment} at the scale \(\beta_N \coloneqq \frac{\hat\beta}{\sqrt{\log N}}\), contrary to Theorem~\ref{thm:second-moment} in dimension \(d=2\).
This stresses that the case of dimension \(d\geq 3\) with critical correlations \(h(x)\sim |x|^{-2}\) is of different nature than that of the critical dimension \(d=2\).

\section{Some preliminaries: chaos expansion and covariances}

\subsection{Chaos expansion of the partition function}

We now introduce and comment on the chaos decomposition of the partition function, which has become a crucial tool in the study of the directed polymer model since~\cite{CSZ17a}.
Recall that 
\[
W_N^{\beta_N}(x) = \bE_x\Big[ \exp\Big(\sum_{n=1}^N (\beta_N \omega_{n,S_n} -\frac12 \beta_N^2) \Big) \Big] = \bE_x\Big[ \exp\Big(\sum_{n=1}^N  \sum_{x\in \Z^2}(\beta_N \omega_{n,x} -\frac12 \beta_N^2) \ind_{\{S_n=x\}} \Big) \Big] \,.
\]
Let us introduce 
\begin{equation}
  \label{def:xi}
  \xi_{n,x} = \xi_{n,x}^{(\beta_N)} \coloneqq \e^{\beta_N \omega_{n,x} - \frac12 \beta_N^2} -1\,,
\end{equation}
and note that \(\EE[\xi_{n,x}]=0\), \(\Var(\xi_{n,x}) = \e^{\beta_N^2}-1 \sim \beta_N^2\).
With this notation, we have
\begin{equation*}
\begin{split}
  W_N^{\beta_N}(x)
  & = \bE_x\Big[ \prod_{n=1}^N  \prod_{x\in \Z^2}\big(1+ \xi_{n,x} \ind_{\{S_n=x\}}\big)\Big] \\
  & = 1+ \sum_{k=1}^{\infty} \sum_{1\leq n_1 <\ldots< n_k\leq N} \sum_{x_1,\ldots, x_k \in \Z^2} \bP_x( S_{n_i}=x_i \ \forall \ 1\leq i \leq k) \prod_{i=1}^k \xi_{n_i,x_i} \,,
\end{split}
\end{equation*}
where we have simply expanded the product and used Fubini.

This expression is known as a chaos expansion of \(W_N^{\beta_N}(x)\) and may be reformulated in a more compact way by introducing some further notation.
For some finite set \(A\subseteq \N \times \Z^2\), we let 
\[
 \xi(A) = \xi^{(\beta_N)}(A) \coloneqq \prod_{(n,x)\in A} \xi_{n,x} \qquad \text{ and } \qquad q^{(x)}(A) \coloneqq \bP_x( S_m=y \,, \forall \ (m,y)\in A) \,.
\]
By convention, both terms are equal to \(1\) when \(A=\emptyset\).
In particular, we stress that \(q^{(x)}(A)\) can be non-zero only for \(A \subseteq\N \times \Z^2 \) of the form \(\{(n_1,x_1), \ldots, (n_k,x_k)\}\) with \(n_1<\ldots <n_k\), in which case 
\begin{equation}
  q^{(x)}(A) =  q_{n_1}(x_1-x) \prod_{i=2}^k q_{n_i-n_{i-1}}(x_i-x_{i-1})  \,,
\end{equation}
where \(q_{n}(z) \coloneqq \bP(S_n=z)\) is the random walk kernel.
With these notations, we may write 
\begin{equation}
  \label{def:chaos-expansion}
  W_N^{\beta_N}(x) = \sum_{A\subseteq \llbracket 1,N\rrbracket\times \Z^2 } q^{(x)}(A) \xi(A) \,.
\end{equation}

For \(f:\Z^2 \to \R_+\), let us also introduce the notation \(W_N^{\beta}(f) \coloneqq \sum_{x\in \Z^2} f(x) W_{N}^{\beta}(x)\).
We can again write the analogous chaos expansion of \(W_N^{\beta}(f)\) simply by writing
\[
W_N^{\beta_N}(f) = \sum_{A\subseteq \llbracket 1,N\rrbracket\times \Z^2 } q^{(f)}(A) \xi(A) \qquad \text{ with }\quad q^{(f)}(A) \coloneqq \sum_{x\in \Z^2} f(x)q^{(x)}(A) \,.
\]

\subsection{Covariances}

Let us stress here that in the case of an i.i.d.\ environment, the chaos expansion~\eqref{def:chaos-expansion} corresponds to some orthogonal \(L^2(\PP)\) decomposition, since we have \(\EE[\xi(A) \xi(B)] =0\) in that case.
However, in the spatially correlated case, \(\xi(A)\) and \(\xi(B)\) are orthogonal in \(L^2(\PP)\) if and only if the time indices of \(A\) and \(B\) do not coincide.
More precisely, if \(A=\{(n_1,x_1), \ldots, (n_k,x_k)\}\) with \(n_1<\ldots <n_k\) and \(B=\{(m_1,y_1), \ldots, (m_j,y_j)\}\) with \(m_1<\ldots <m_j\), a simple calculation gives that, if \(k=j\),
\begin{equation}
  \label{eq:cov-xi}
  \EE[ \xi(A) \xi(B)] = (\beta_N^2)^k \prod_{i=1}^k h_N(x_i-y_i) \ind_{\{n_i=m_i\}} \qquad \text{ with } h_N(z) \coloneqq \frac{\e^{\beta_N^2 h(z)}-1}{\beta_N^2} \,,
\end{equation}
and \(\EE[\xi(A) \xi(B)] =0\) otherwise (\textit{i.e.}\ if the time indices of \(A\) and \(B\) do not coincide).
We refer to Figure~\ref{fig:time-indices} for an illustration.
Note also that \(h(z)\leq h_N(z)\leq \e^{\beta_N^2} h(z)\) so we may use that \(h_N(z) = (1+o(1)) h(z)\), where the \(o(1)\) is uniform in \(z\in \Z^2\).

\begin{figure}
  \centering
  \begin{tikzpicture}[
      xscale=1.1, yscale=1.1,
      solid dot/.style={circle, fill=blue, inner sep=0pt, minimum size=6pt},
      empty dot/.style={circle, draw=black, very thick, fill=white, inner sep=0pt, minimum size=6pt},
      green line/.style={color=green!50!black, thick,<->},
      path line/.style={thick, dashed}
  ]
  \draw[thick, ->] (0,0) -- (11,0);
  \draw[thick] (1,0.1) -- (1,-0.1) node[below=5pt] {$n_1=m_1$};
  \draw[thick] (2.5,0.1) -- (2.5,-0.1) node[below=5pt] {$n_2=m_2$};
  \node at (5.5,-0.5) {$\dots$};
  \draw[thick] (9.5,0.1) -- (9.5,-0.1) node[below=5pt] {$n_k=m_k$};
  \draw[path line, blue] (1,2) -- (2.5,1.8) -- (4.5,2.6) -- (6.5,1.5) -- (8,1.3) -- (9.5,2.8);
  \draw[path line] (1,0.7) -- (2.5,1.2) -- (4.5,0.8) -- (6.5,2.4) -- (8,2.8) -- (9.5,1.9);
  \draw[green line] (1,1.85) -- (1,0.85);
  \draw[green line] (2.5,1.65) -- (2.5,1.35);
  \draw[green line] (4.5,2.45) -- (4.5,0.95);
  \draw[green line] (6.5,1.65) -- (6.5,2.25);
  \draw[green line] (8,1.45) -- (8,2.65);
  \draw[green line] (9.5,2.05) -- (9.5,2.65);
  \node[solid dot] at (1,2) {};
  \node[empty dot] at (1,0.7) {};
  \node[above right] at (2.5,1.9) {\blue $B$};
  \node[below right] at (2.5,1.1) {$A$};
  \node[solid dot] at (2.5,1.8) {};
  \node[empty dot] at (2.5,1.2) {};
  \node[solid dot] at (4.5,2.6) {};
  \node[empty dot] at (4.5,0.8) {};
  \node[solid dot] at (6.5,1.5) {};
  \node[empty dot] at (6.5,2.4) {};
  \node[solid dot] at (8,1.3) {};
  \node[empty dot] at (8,2.8) {};
  \node[solid dot] at (9.5,2.8) {};
  \node[empty dot] at (9.5,1.9) {};
  \end{tikzpicture}
  \caption{Illustration of the covariance computation~\eqref{eq:cov-xi}. The elements of two sets of indices \(A,B\) are represented by blue and white dots. 
  The covariance \(\EE[\xi(A)\xi(B)]\) is non-zero only if the time-indices of \(A,B\) coincide, but the spatial positions do not need to coincide: the covariance depends on the distances between the points of \(A,B\) at each time index, through the function \(h_N\).}
  \label{fig:time-indices}
\end{figure}
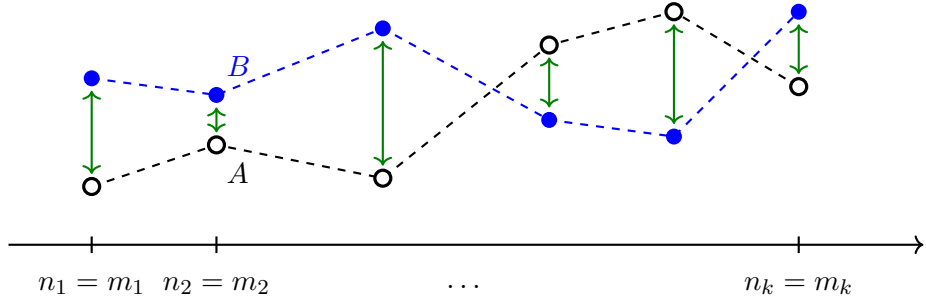

Thus, using the chaos expansion~\eqref{def:chaos-expansion}, we get that 
\begin{align}
  \label{chaos-covariances}
    & \EE \big[ W_N^{\beta_N}(x)W_N^{\beta_N}(y) \big] 
    = \sum_{A,B\subset \llb 1,N\rrb \times \Z^2} q^{(x)}(A)q^{(y)}(B) \EE[\xi(A)\xi(B)]  \\
    & = \sum_{k=0}^{\infty} (\beta_N)^k \sum_{1\leq n_1<\cdots<n_k \leq N} \sumtwo{x_1,\ldots, x_k\in \Z^2}{y_1,\ldots, y_k \in \Z^2} \prod_{i=1}^{k} h_N(x_i-y_i) q_{n_i-n_{i-1}}(x_i-x_{i-1})  q_{n_i-n_{i-1}}(y_i-y_{i-1}) \,, \notag
\end{align}
where by convention \(n_0=0\), \(x_0=x\), \(y_0=y\).
The second line simply comes from~\eqref{eq:cov-xi}, using that \(\EE[\xi(A)\xi(B)]=0\) if \(A,B\) do not share the same time-indices (but contrary to the i.i.d.\ case their points may differ in space, see Figure~\ref{fig:covariances}).
Finally, let us stress that because of parity constraints, the above sum is restricted to \(x_i-y_i \in \Zeven^2\).

\section{Proof of the main theorems}

In this section, we show how the super-critical case of Theorem~\ref{thm:fluctuations} and upper bounds on the free energy follow from a single ``super-critical'' statement.
We also prove the subcritical case of Theorem~\ref{thm:fluctuations} and lower bounds on the free energy, relying only on second moment upper bounds (namely~\eqref{eq:sup-covariances}).

\subsection{A key super-critical estimate and its consequences}
\label{sec:key-super-critical}

We now state a result which is slightly stronger than the super-critical regime in Theorem~\ref{thm:fluctuations}.
We let \(\mathcal{M}_1(R)\) be the set of probability measures with support on \(B_R=\{ x\in \Zeven^2, \|x\|_{\infty}\leq R  \}\); for simplicity, we also denote \(\mathcal{M}_1= \mathcal{M}_1(\infty)\) the set of probability measures on \(\Zeven^2\).
Recall also that we have defined \(W_N^{\beta}(f) =\sum_{x\in \Z^d} f(x) W_N^{\beta}(x)\).

\begin{theorem}
  \label{thm:key-theorem}
  Assumption~\ref{main-assumption} holds with either~\eqref{finite-sum} or~\eqref{def:h} and that the scaling of \((\beta_N)_{N\geq 1}\) is as in Assumption~\ref{hyp:scaling} with \(\hat\beta>\hat\beta_c\).
  Then we have that 
  \[
    \sup_{f\in \mathcal{M}_1(\sqrt{N})} \EE[W_N^{\beta_N}(f) \wedge 1]   \xrightarrow[\;N\to\infty\;]{} 0 \,.
  \]
\end{theorem}

We now show how the super-critical regime in Theorem~\ref{thm:fluctuations} and the upper bounds on the free energy in Theorem~\ref{thm:free-energy} can be deduced from Theorem~\ref{thm:key-theorem}.

\subsubsection{Super-critical regime in Theorem~\ref{thm:fluctuations}}

For any \(\varphi \in \mathcal{C}_{c}^{\infty}\), we may consider the function \(f_N \coloneqq \varphi_N/\|\varphi\|_{L^1(\R^2)}\), and notice that \(W_N^{\beta_N} (\varphi) = \|\varphi\|_{L^1(\R^2)} W_N^{\beta_N}(f_N)\), so we only need to show that \(W_N^{\beta_N}(f_N) \to 0\) in probability. 

The function \(f_N\) is a density on \(\Zeven^2\), whose support is contained in \(B_{R\sqrt{N}}\), for some \(R>0\).
In particular, when \(R\leq 1\), we have that \(f_N\in \mathcal{M}_1(\sqrt{N})\), so by a simple application of Markov's inequality Theorem~\ref{thm:key-theorem} implies that \(W_N^{\beta_N}(f_N) \to 0\) in probability.

When \(R> 1\), we consider \(K\) balls \(B^{(i)} = \{ x\in \Zeven^2, \|x-x_i\| \leq \sqrt{N}\}\) of radius \(\sqrt{N}\) with centers \(x_i\in \Z^2\) such that \(B_{R\sqrt{N}} \subseteq \bigcup_{i=1}^K B^{(i)}\), with \(K\leq C R^2\).
Then, writing \(\alpha_i\coloneqq \sum_{x\in B^{(i)}}f_N(x) \in [0,1]\) and \(f^{(i)} = \alpha_i^{-1} f_N \ind_{B^{(i)}}\) (with by convention \(f^{(i)}=0\) if \(\alpha_i=0\)), we have that \(f_N \leq  \sum_{i=1}^{K} \alpha_i f^{(i)}\).
By translation invariance, we get that \(\EE[W_N^{\beta_N}(f^{(i)}) \wedge 1 ]\leq \sup_{f\in \mathcal{M}_1(\sqrt{N})} \EE[W_N^{\beta_N}(f) \wedge 1 ]\), which goes to \(0\) for all \(1\leq i\leq K\), by Theorem~\ref{thm:key-theorem}.
This shows that \(W_N^{\beta_N}(f^{(i)}) \to 0\) in probability for all \(i\leq K\) and since \(0\leq \alpha_i\leq 1\) and \(i\leq K\), we conclude that \(W_N^{\beta_N}(f_N) \leq \sum_{i=1}^{K} \alpha_i W_N^{\beta_N}(f^{(i)}) \to 0\) in probability.
This concludes the proof of the super-critical regime in Theorem~\ref{thm:fluctuations}.
\qed

\subsubsection{Upper bound on the free energy}

For the upper bound on the free energy, we use the following ``finite-volume'' estimate. 
The proof follows a coarse-graining procedure and is classical: we can in fact directly use \cite[Prop.~2.4]{BCT25}, for which only the independence in time is used.
Let us state here \cite[Prop.~3.5]{BCT25} (the proof is in fact valid in any dimension \(d\geq 2\)).

\begin{proposition}
  \label{prop:coarse-graining}
  There exists a (small) constant \(\eta = \eta(d)\) such that, if there exists \(L \in \N\) and \(\beta>0\) such that 
  \begin{equation}
    \label{eq:finite-volume}
    \sup_{ f \in \mathcal{M}_1(\sqrt{L})} \EE\Big[ W_L^{\beta}(f)^{1/2} \Big] \leq \eta_d \,,
  \end{equation}
  then for all \(N\in \N\), we have
  \[
  \sup_{ f \in \mathcal{M}_1(\sqrt{L})} \EE\Big[ W_N^{\beta}(f)^{1/2} \Big] \leq 3 \, \e^{-N/L} \,.
  \]
\end{proposition}

\begin{proof}[Proof of the upper bound in Theorem~\ref{thm:free-energy}]
  Let \(\gep>0\).
  For any \(\beta\in (0,1)\), define \(L=L(\beta)\) as follows (we assume that \(L\in \N\) for simplicity of notation):
  \begin{equation}
    \label{def:Leps}
    \begin{split}
    L(\beta) &\coloneqq  \exp\bigg(\Big(  \frac{(1+\epsilon) \mathfrak{C}_h }{\beta}\Big)^{2} \bigg)  \quad  \text{ under \eqref{finite-sum}}\,,\\
    L(\beta) & \coloneqq  \exp\bigg(\Big(  \frac{(1+\epsilon) z_a \mathfrak{C}_a }{\beta}\Big)^{\frac{2}{a+2}} \bigg)  \quad \text{ under \eqref{def:h}}\,,
    \end{split}
  \end{equation}
  which is equivalent to setting \(\beta=\beta(L)\) as 
  \[
  \beta(L) = (1+\gep)  \frac{\mathfrak{C}_h}{ (\log L)^{1/2}} \quad \text{ under \eqref{finite-sum}} \,, \quad
  \beta(L) = (1+\gep)  \frac{z_a\mathfrak{C}_a}{ (\log L)^{(a+2)/2}} \quad \text{ under \eqref{def:h}}\,.
  \] 
  In other words, \(\beta(L)\) follows the scaling of Assumption~\ref{hyp:scaling} with \(\hat \beta = (1+\gep) \hat \beta_c\), \textit{i.e.}\ with super-critical intensity.

  Using that \(\EE[Z^{1/2}]\leq \sqrt{2} \EE[Z\wedge 1]^{1/2}\) for any \(Z\geq 0\) with \(\EE[Z]=1\), see \cite[Lem.~2.2]{BCT25}, applying Theorem~\ref{thm:key-theorem} shows that\(\sup_{ f \in \mathcal{M}_1(\sqrt{L})} \EE[ W_L^{\beta}(f)^{1/2} ] \to 0\) as \(\beta \downarrow 0\), or equivalently as \(L\to \infty\).
  In other words, there exists some \(\beta_{\gep}\) such that for all \(\beta\leq \beta_{\gep}\) the inequality \eqref{eq:finite-volume} holds.
  Thus, Proposition~\ref{prop:coarse-graining} gives that, for any \(N\geq 1\), for any \(\beta\leq \beta_{\eta}\),
  \[
    \EE\big[ (W_N^{\beta})^{1/2} \big] \leq \sup_{ f \in \mathcal{M}_1(\sqrt{L})} \EE\Big[ W_N^{\beta}(f)^{1/2} \Big] \leq 3 \, \e^{-N/L} \,.
  \]
  This bound on the fractional moment then yields the following upper bound on the free energy
  \[
  \tf(\beta) = \lim_{N\to\infty} \frac{2}{N} \bbE\big[ \log (W_N^{\beta})^{1/2} \big] \leq \liminf_{N\to\infty} \frac{2}{N}  \log \bbE\big[(W_N^{\beta})^{1/2} \big] \leq  - \frac{2}{L} \,.
  \]
  In conclusion, in view of the definition~\eqref{def:Leps}, we get that for any \(\gep \in (0,1)\), for all \(\beta\) sufficiently small,
  \[
  \begin{split}
    \tf(\beta) & \leq - 2 \exp\bigg(\Big( -  \frac{(1+\epsilon)  \mathfrak{C}_h }{\beta}\Big)^{2} \bigg) \quad\text{ if \eqref{finite-sum} holds}\,,\\
  \tf(\beta) & \leq - 2 \exp\bigg(\Big( -  \frac{(1+\epsilon) z_a \mathfrak{C}_a }{\beta}\Big)^{\frac{2}{a+2}} \bigg) \quad\text{ if \eqref{def:h} holds}\,.
  \end{split}
  \]
  This concludes the proof of the upper bound in Theorem~\ref{thm:free-energy}.
\end{proof}

\subsection{Sub-critical regime in Theorem~\ref{thm:fluctuations}} 
\label{sec:subcriticalpart}

Applying Chebyshev's inequality and recalling that \(\EE[W_N^{\beta_N}(\varphi)]= \int_{\R^2} \varphi(x) \dd x\), we get that for any $\varepsilon >0$
\begin{equation*}
	\PP \bigg( \, \Big| W_N^{\beta_N}(\varphi) - \int_{\R^2} \varphi(x) \dd x\Big| > \varepsilon \, \bigg) \le  \frac{\Var\big(W_N^{\beta_N}(\varphi)\big)}{\varepsilon^2} \,.
\end{equation*}
It therefore only remains to show that \(\Var(W_N^{\beta_N}(\varphi))\) goes to \(0\) in the subcritical regime \(\hat \beta <\hat \beta_c\), which is given by the following result.

\begin{proposition}\label{prop:boundvariance}
	Suppose that Assumption~\ref{main-assumption} holds and that either~\eqref{finite-sum} or~\eqref{def:h} is satisfied.
  Let \((\beta_N)_{N\geq 1}\) be as in Assumption~\ref{hyp:scaling} with \(\hat\beta<\hat\beta_c\).
	Then, there exist constants \(C= C(\hat\beta)<\infty\) and \(N_0\in\mathbb N\) such that, for every \(N\ge N_0\),
	\begin{equation} \label{eq:boundvariance}
	\Var\big(W_N^{\beta_N}(\varphi)\big)
	\le
	\frac{C}{\log N}\,.
\end{equation}
\end{proposition}

\begin{remark}
  Since this is not our main goal here, we do not derive the sharp asymptotic behavior of \(\Var\big(W_N^{\beta_N}(\varphi)\big)\).
  One should however be able to obtain the value of \(\lim_{N\to\infty} \log N \Var\big(W_N^{\beta_N}(\varphi)\big)\) in the subcritical case \(\hat \beta<\hat \beta_c\).
  We refer to \cite{BCC26} for some further results on the fluctuations of \(\log W_N^{\beta_N}(\varphi)\), where similar computations are carried out.
\end{remark}

\begin{proof}[Proof of Proposition \ref{prop:boundvariance}]
First of all, by definition~\eqref{def:W-phi-2} of \(W_N^{\beta}(\varphi)\), we have that
\begin{equation}
  \label{eq:variance-subcritical-1}
  \Var\big(W_N^{\beta_N}(\varphi)\big) 
  = \sum_{x,y \in \Zeven^2} \varphi_N(x) \varphi_N(y) \Big( \EE \big[ W_N^{\beta_N}(x)W_N^{\beta_N}(y) \big] -1\Big) \,.
\end{equation}
Now, using the chaos decomposition~\eqref{chaos-covariances} and decomposing over the first time-index \(n_1\), we get 
\[
  \EE \big[ W_N^{\beta_N}(x)W_N^{\beta_N}(y) \big] = 1+ \beta_N^2 \sum_{n=1}^{N}\sum_{x',y'\in \Z^2}  h_N(x'-y') q_n(x'-x)  q_n(y'-y)  \EE \big[ W_{N-n}^{\beta_N}(x')W_{N-n}^{\beta_N}(y') \big]\,.
\]
Now, using the monotonicity in \(N\) of the covariances (which is obvious from~\eqref{chaos-covariances}), we have that 
\[
\sup_{x',y'\in \Z^2}  \EE \big[ W_{N-n}^{\beta_N}(x')W_{N-n}^{\beta_N}(y') \big] \leq \sup_{x',y'\in \Z^2}  \EE \big[ W_{N}^{\beta_N}(x')W_{N}^{\beta_N}(y') \big]  <+\infty \,,
\]
where the last inequality uses~\eqref{eq:sup-covariances} (recall that \(\hat\beta<\hat\beta_c\)).

Going back to~\eqref{eq:variance-subcritical-1} and using that \(\varphi\) has compact support so that \(\varphi_N(x) \leq \frac{c}{N} \ind_{\{|x|\leq c\sqrt{N}\}}\) for some constant \(c>0\), we finally obtain (recall also that \(h_N(z)\leq C h(z)\))
\begin{equation*}
		\Var\big(W_N^{\beta_N}(\varphi)\big) \le \frac{C c^2 \, \beta_N^2}{N^2} \sum_{|x|,|y| \le c\sqrt{N}} \sum_{n=1}^N \sum_{x',y'\in\Z^2}h(x'-y') q_{n}(x'-x) q_n(y'-y)  \,,
\end{equation*}
and let us recall that because of parity issues we have that \(x'-y' \in \Zeven^2\).
In particular, if we change variables to write \(z=x'-y'\) the sums over \(z\) will be restricted to \(\Zeven^2\) (omitting this restriction only leads to upper bounds).

\smallskip
\textbullet\
In the summable case \eqref{finite-sum}, we simply bound \(\sum_{|x| \leq c\sqrt{N}} q_{n}(x'-x)\) by \(1\), then we sum \(h(x'-y')\) over \(x'\), then \(q_n(y'-y)\) over \(y'\), to get that
\begin{equation}\label{eq:boundsumm}
		\Var\big(W_N^{\beta_N}(\varphi)\big)  \leq \frac{C' \, \beta_N^2}{N^2} \sum_{n=1}^N  \sum_{|y|\leq c\sqrt{N}} \Big(\sum_{z\in \Z^2}  h(z) \Big)  \le \frac{C''}{\log N}\,,
\end{equation}
recalling that $\beta_N^2 \sim C (\log N)^{-1}$, see~\eqref{def:beta-N-summable}.

\smallskip
\textbullet\
In the non--summable case \eqref{def:h}, we split the sum according to whether $|x'-y'|\le \sqrt{N}$ or not.
In the case \(|x'-y'|\leq \sqrt{N}\) we bound \(\sum_{|x| \leq c\sqrt{N}} q_{n}(x'-x)\) by \(1\) then we sum \(h(x'-y')\) over \(x'\) (with the constraint \(|x'-y'|\leq \sqrt{N}\)) then we sum \(q_n(y'-y)\) over \(y'\).
Similarly as above, we get that the contribution of the sum with \(|x'-y'| \leq \sqrt{N}\) is bounded by a constant times
\[
\frac{\beta_N^2}{N^2} \sum_{n=1}^N  \sum_{|y|\leq c\sqrt{N}} \Big(\sum_{|z|\leq \sqrt{N}}  h(z) \Big)\leq \frac{C''}{\log N} \,,
\]
where we have used that $\beta_N^2 \sim C_{\hat\beta} (\log N)^{-(a+2)}$, see~\eqref{def:beta-N}, and that \(\sum_{|z| \le \sqrt{N}} h(z) \sim C' (\log N)^{a+1}\), see~\eqref{sum:h}.

For the remaining term, we use that \(h(x'-y') \leq C (\log N)^a N^{-1}\) for all \(|x'-y'| > \sqrt{N}\). 
Then we simply bound the sum \(\sum_{x',|x'-y'|> \sqrt{N}} q_{n}(x'-x)\) by \(1\) and then we sum \(q_n(y'-y)\) over~\(y'\).
Thus, the contribution of the sum with \(|x'-y'| > \sqrt{N}\) is bounded by a constant times
\[
\beta_N^2 \sum_{n=1}^N \sum_{|x|,|y| \le c\sqrt{N}} \frac{(\log N)^a}{N} \le \frac{C'}{ (\log N)^2} \,,
\]
where we have again used that $\beta_N^2 \sim C (\log N)^{-(a+2)}$, see~\eqref{def:beta-N}.
Combining the last two estimates conclude the proof in the non-summable case~\eqref{def:h}.
\end{proof}

\subsection{Lower bound on the free energy}
\label{sec:lower-bound}

To prove the lower bound on the free energy, we use the super-additivity of \((\EE\log W_N^{\beta})_{N\geq 0}\), together with a ``subcritical'' estimate of \(\EE\log W_N^{\beta}\), see Lemma~\ref{lem:LWexplog} below, which is in practice based on a second moment estimate. 

For any fixed $\beta >0$, the sub-additivity of  \((\EE\log W_N^{\beta})_{N\geq 0}\) gives
\begin{equation}
  \label{eq:lwF}
	\tf(\beta) = \lim_{N\to\infty} \frac{1}{N} \bbE\big[ \log W_N^{\beta} \big] = \sup_{N \in \N}\frac{1}{N} \bbE\big[ \log W_N^{\beta} \big] \,.
\end{equation}
We now choose $\bar N=\bar N(\beta)$ as large as possible while remaining in the subcritical regime: our aim is to keep $\EE[ \log W_{\bar N}^{\beta}]$ uniformly bounded from below while the prefactor $1/ \bar N$ is smallest possible. 
To this end, in analogy with~\eqref{def:Leps}, fix $\varepsilon >0$ and set \(\bar N=\bar N(\beta)\) as follows (we assume that~\(\bar N\) is an integer for simplicity):
\begin{equation}
  \label{def:barNeps}
  \begin{split}
  \bar N(\beta) &\coloneqq  \exp\bigg(\Big(  \frac{(1-\epsilon) \mathfrak{C}_h }{\beta}\Big)^{2} \bigg)  \quad  \text{ under \eqref{finite-sum}}\,,\\
  \bar N(\beta) & \coloneqq  \exp\bigg(\Big(  \frac{(1-\epsilon) z_a \mathfrak{C}_a }{\beta}\Big)^{\frac{2}{a+2}} \bigg)  \quad \text{ under \eqref{def:h}}\,.
  \end{split}
\end{equation}
Notice that we can again invert this relation and write \(\beta\) as a function of \(\bar N\) (so that choosing \(\beta\) small is equivalent to choosing \(\bar N\) large):
\begin{equation}
	\label{eq:betaNbar}
  \beta(\bar N)=(1-\varepsilon) \frac{ \mathfrak{C}_h }{\sqrt{\log N}} \quad \text{ under \eqref{finite-sum}}\,,
  \quad 
	\beta(\bar N)=(1-\varepsilon) \frac{ z_a \mathfrak{C}_a }{(\log N)^{\frac{a+2}{2}}} \quad \text{ under \eqref{def:h}}\,.
\end{equation}
In other words, \(\beta(\bar N)\) satisfies the scaling of Assumption~\ref{hyp:scaling} with subcritical intensity \(\hat \beta = (1-\gep) \hat \beta_c\).

We now claim the following.
\begin{lemma}\label{lem:LWexplog}
  Assume that \((\beta_N)_{N\geq 1}\) satisfies the scaling of Assumption~\ref{hyp:scaling} with \(\hat \beta < \hat \beta_c\).
  Then, there exists a constant $C_{\hat\beta}>0$ such that for all \(N\geq 1\)
	\begin{equation*}
    \label{eq:Ceps}
		\EE [ \log W_{N}^{\beta_N}] \ge -C_{\hat \beta}\,.
	\end{equation*}
\end{lemma}
Applying Lemma~\ref{lem:LWexplog} in \eqref{eq:lwF}, we obtain 
\begin{equation*}
	\tf(\beta) \geq \frac{1}{\bar N} \EE [ \log W_{\bar N}^{\beta}]  \ge -C_{(1-\gep)\hat\beta_c} 
  \begin{cases}
    \exp\Big(-\big(  \frac{(1-\epsilon) \mathfrak{C}_h }{\beta}\big)^{2} \Big) & \ \text{ under~\eqref{finite-sum}},\\
    \exp\Big(-\big(  \frac{(1-\epsilon) z_a \mathfrak{C}_a }{\beta}\big)^{\frac{2}{a+2}} \Big) & \ \text{ under~\eqref{def:h}}.
  \end{cases}
\end{equation*}
Since \(\gep >0\) is arbitrary, this gives the desired lower bound on the free energy.
\qed

\begin{proof}[Proof of Lemma \ref{lem:LWexplog}]
	We follow the lines of \cite[\S4]{BL17} and rely on the concentration inequality from \cite[Prop.~1.6]{Led05} in order to estimate the left tail of \(\log W_N^{\beta_N}\) only thanks to second moment estimates.
	Since only the Gaussian case is relevant for our purposes, we state the result in this setting only, with the formulation given in \cite[Prop.~3.4]{CTT17}.
  \begin{proposition}
    \label{prop:CSZ20}
    Assume $\hat \omega=(\hat \omega_i)_{i \in \N}$ is a sequence of i.i.d.\ standard Gaussian random variables. There exist constants $c_1,c_2\in (0,\infty)$ such that, for any $n \in \N$ and for any differentiable and convex function $f : \R^n \to \R$, the following bound holds for any $b \in \R$ and $t,c \in (0,\infty)$:
    \begin{equation}
      \label{eq:concineq}
      \PP \big(\, f(\hat \omega) \le b-t \,\big)\PP \big(\, f(\hat \omega) \ge b \,, \ |\nabla f(\hat \omega)| \le c\, \big) \le c_1 \exp \Big\lbrace - \frac{(t/c)^2}{c_2} \Big\rbrace\,,
    \end{equation}
    where $\hat \omega=(\hat \omega_1,\ldots,\hat \omega_n)$ and $|\nabla f(\hat \omega)|=\sqrt{\sum_{i=1}^n (\frac{\partial}{\partial \hat \omega_i} f)^2}$ is the norm of the gradient. 
  \end{proposition}

We shall apply Proposition \ref{prop:CSZ20} with $f(\hat \omega)=f_{N}(\hat \omega)=\log W_{N}^{\beta_N}$.
However, to apply Proposition~\ref{prop:CSZ20}, we first express $\log W_{N}^{\beta_N}$ as a function of i.i.d.\ Gaussian random variables \(\hat\omega(n,y)\), recalling the expression~\eqref{omega-hat-omega}.
A straightforward calculation (details are provided in Appendix~\ref{app-gradient}) gives the following:
\begin{equation}
  \label{eq:nabla-f}
		|\nabla f_N(\hat \omega)|^2 = \frac{\beta_N^2}{(W_N^{\beta_N})^2} \bE^{\otimes 2} \bigg[ \,  \mathcal{L}_{N}(S,S')\, \e^{\sum_{m=1}^N \big( \beta_N \{\omega(m,S_m)+\omega(m,S'_m) \}- \beta_N^2\big)}\,\bigg]\,,
\end{equation}
where \(\bE^{\otimes 2}\) denotes the expectation with respect to to independent copies $S,S'$ of the random walk and where $\mathcal{L}_N(S,S') \coloneqq\sum_{n=1}^N h(S_n-S_n')$ as in~\eqref{eq:covariances}.

Now, let us show that \(\PP( f_N(\hat \omega) \ge - \log 2 ; |\nabla f_N(\hat \omega)| \le c)\) is bounded away from~\(0\) for some well chosen constant \(c>0\).
First of all, notice that since \(\EE[W_N^{\beta_N}]=1\), we have by Paley--Zygmund
\begin{equation}
  \label{eq:PZineq}
    \PP \big( \, f_N(\hat \omega) \ge - \log 2\, \big) = \PP \Big( \, W_N^{\beta_N} \ge \frac{1}{2}\EE [W_N^{\beta_N}]\, \Big)\ge \frac{\EE [W_N^{\beta_N}]^2}{4 \EE[(W_N^{\beta_N})^2]} \ge \frac{1}{4C_{\hat \beta}}\,,
\end{equation}
where we have used that \(\sup_{N\geq 1}\EE[(W_N^{\beta_N})^2] <+\infty\) thanks to~\eqref{eq:sup-covariances} (recall \(\hat\beta<\hat\beta_c\)).

On the other hand, on the event that $f_N(\hat \omega)\ge -\log 2$, \textit{i.e.}\ $W_N^{\beta_N} \ge \frac{1}{2}$, using~\eqref{eq:nabla-f}, we have
\begin{equation*}
    \EE \big[ |\nabla f_N(\hat \omega)|^2 \ind_{\{f_N(\hat \omega)  \ge -\log 2 \}} \big] \le 4 \EE \bigg[ \bE^{\otimes 2} \Big[  \beta_N^2 \mathcal{L}_{N}(S,S')\, \e^{\sum_{m=1}^N \big( \beta_N \{\omega(m,S_m)+\omega(m,S'_m) \}- \beta_N^2\big)}  \Big]\bigg] \,.
\end{equation*}
Then, using the identity $\EE \big[   \e^{\beta \sum_{m=1}^N\{ \omega(m,S_m)+\omega(m,S'_m) \}- N \beta^2}  \big]=\e^{\beta^2 \mathcal{L}_N(S,S')}$ following from \cite[Lem.~2.1]{CCD25}, we get that for any fixed \(\delta>0\) the above is bounded by 
\begin{equation*}
    4\bE^{\otimes 2} \Big[ \beta_N^2  \mathcal{L}_{N}(S,S')\, \e^{\beta_N^2 \mathcal{L}_{N}(S,S') }\Big] \le \frac{4}{\delta(2+\delta)}\, \bE^{\otimes 2} \Big[   \e^{(1+\delta)^2 \beta_N^2 \mathcal{L}_{N}(S,S') }\Big] \leq \frac{2}{\delta} \EE\big[ (W_N^{(1+\delta) \beta_N})^2\big] \,,
\end{equation*}
using also the inequality $x \le \frac{1}{u} \e^{u x}$ valid for any \(x\geq 0\) and \(u>0\) and then~\eqref{eq:covariances}.
Then, if we have fixed \(\delta\) small enough so that \(\hat\beta' :=(1+\delta) \hat \beta <\hat \beta_c\) we get that \(\sup_{N\geq 1} \EE[ (W_N^{(1+\delta) \beta_N})^2] \leq C_{\hat\beta'}\) thanks to~\eqref{eq:sup-covariances}.
All together, using Markov's inequality we get
\begin{equation}
  \label{eq:bound-gradient}
  \PP\big(  f_N(\hat \omega) \ge - \log 2 \, ; \, |\nabla f_N(\hat \omega)| > c \big) \le \frac{1}{c^2} \EE \big[ | \nabla f_N(\hat \omega)|^2 \ind_{\{f_N(\hat \omega)  \ge -\log 2 \}} \big] \leq \frac{1}{c^2} \frac{2}{\delta} C_{\hat\beta'} \,.
\end{equation}

We now choose \(c= c_{\hat\beta}\coloneqq (16 C_{\hat\beta} C_{\hat\beta'} /\delta)^{1/2}\) so that combining~\eqref{eq:PZineq}-\eqref{eq:bound-gradient}, we end up with 
\begin{equation}
  \label{eq:claimtheta}
  \PP\big( f_N(\hat \omega) \ge - \log 2 \, ; \, |\nabla f_N(\hat \omega)| \le c_{\hat\beta} \big) \geq \frac{1}{4 C_{\hat\beta}}- \frac{1}{c_{\hat\beta}^2} \frac{2}{\delta} C_{\hat\beta'} = \frac{1}{8 C_{\hat\beta}} \,.
\end{equation}

We are now ready to conclude the proof. By Proposition~\ref{prop:CSZ20}, \eqref{eq:claimtheta} gives the following estimate on the left-tail of \(f_N(\hat \omega) = \log W_{N}^{\beta_N}\):
\[
\PP\big(  \log W_{N}^{\beta_N} \le - \log 2 -t  \big) \le C_{\hat\beta}' \, \e^{ - c_{\hat\beta}' t^2}  \,,
\]
for some constants $C_{\hat\beta}', c_{\hat\beta}'>0$ (that do not depend on \(N\)).
From this, we deduce that
\begin{equation*}
	\EE \big[ \log W_{N}^{\beta_N} \big] 
  \ge -\log 2 - \int_0^\infty \PP\big(  \log W_{N}^{\beta_N} \le - \log 2 -t  \big) \dd t 
  \ge -\log 2 - C_{\hat\beta}' \int_0^\infty  \e^{ - c_{\hat\beta}' t^2}\eqqcolon -C_{\hat\beta}''\,,
\end{equation*}
which completes the proof of Lemma~\ref{lem:LWexplog}.
\end{proof}

\section{Strategy and main steps of the proof of Theorem~\ref{thm:key-theorem}}
\label{sec:key-proof}

We now turn to the proof of our central result, Theorem~\ref{thm:key-theorem}.
In this section, we explain the main steps of the proof, and highlight the main technical estimates we need to obtain.
The proof follows a by-now standard change of measure argument, whose delicate point is always to find a well-suited event on the environment \(\omega\).
We present here a general, streamlined, method to build such an event, based on a functional constructed as a coarse-grained chaos decomposition inspired by~\cite{BCT25,BCZ26}.
The main interest of our approach is that \emph{it only relies on second moment estimates} and is flexible enough to treat the case of long-range spatial correlations (and most likely other models) where second moment calculations have been carried out in~\cite{CCD25}.
Note that~\cite{BCT25} also uses some (tedious) third moment computation, but it aims at showing a much sharper result (namely the super-critical regime in Theorem~\ref{thm:SHF}): we use here hypercontractivity to reduce to a second moment estimate.

\subsection{The change of measure argument}

As a preliminary step, we start by using a change of scale for the starting point: using \cite[Prop.~4.5]{BCT25}, we have that for any \(\rho\in (0,1)\),
\begin{equation*}
  \sup_{ f \in \mathcal{M}_1(\sqrt{N})} \EE\big[ W_N^{\beta_N}(f)\wedge 1 \big] \leq 8 \sup_{f\in \mathcal{M}_1(\sqrt{\rho N})} \EE\Big[ W_N^{\beta_N}(f)\wedge \frac{1}{\rho} \Big] \,.
\end{equation*}
Loosely speaking, the idea behind this change of scale is to make \(W_N^{\beta_N}(f)\) look like a ``coarse-grained point-to-plane partition function'', for which the change of measure will be more easily applied.
(We eventually take \(\rho = (\log N)^{-1}\), but this precise value is immaterial.)

Now, observe that for any density \(f\) on \(\Z^2\) we have that \(W_N^{\beta_N}(f)\) is a non-negative mean~\(1\) random variable. 
We can thus us introduce the so-called size-biased measure:
\begin{equation}
  \label{def:size-biased}
  \forall\, A\in \mathcal{F}_N\,, \qquad 
  \tPP_f(A) = \tilde{\PP}_{N,f}^{\beta_N}(A) \coloneqq  \EE\big[ W_N^{\beta_N}(f) \ind_A \big] \,.
\end{equation}

\begin{remark}
  There is some interpretation of the size-biased probability measure \(\tPP_f\), which will not be useful for us but is worth mentioning. 
  Using Fubini's theorem, we get that 
  \[
  \tPP_f(A) = \bE_f\Big[ \EE\Big[ \e^{\sum_{n=1}^N (\beta_N \omega_{n,S_N} -\frac12 \beta_N^2 ) } \ind_A \Big] \Big] = \bE_f\big[ \hat{\PP}^{(S)}(A)\Big]
  \]
  where \(\bP_{f}\coloneqq \sum_{x\in \Z^d} f(x) \bP_x\) and \(\dd \hat \PP_N^{(S)} (\omega) \coloneqq \e^{\sum_{n=1}^N (\beta_N \omega_{n,S_N} -\frac12 \beta_N^2 ) } \dd \PP \) for any fixed trajectory~\(S\).
  Note that, by Assumption~\ref{hyp:scaling}, \(\hat\PP_N^{(S)}\) is the law of a Gaussian field with the same covariances as under \(\PP\) but with mean \(\beta_N h(x-S_n)\).
  In other words, under \(\tPP_f\), the environment (up to time~\(N\)) is constructed by drawing a random walk trajectory \((S_n)_{1\leq n \leq N}\) under \(\bP_f\) and then adding \((\beta_N h(x-S_n))_{1\leq n \leq N, x\in \Z^2}\) to the original environment \(\omega\). 
\end{remark}

With this notation, for any event \(A_N \in \mathcal{F}_N\) (that might depend on \(f\)), we have that
\begin{equation}
  \label{eq:change-of-measure}
  \EE\Big[ W_N^{\beta_N}(f)\wedge \frac{1}{\rho} \Big] \leq \EE\Big[ \frac{1}{\rho} \ind_{A_N} \Big] +  \EE\Big[ W_N^{\beta_N}(f) \ind_{A_N^{c}} \Big] = \frac{1}{\rho}  \PP(A_N) + \tPP_f(A_N^c) \,.
\end{equation}
This is what is known as a change of measure argument: one is left with finding some event \(A_N\) which is atypical under the initial measure \(\PP\), but is typical under the size-biased measure \(\tPP_f\).

Let us mention that in~\eqref{eq:change-of-measure}, the inequality is an equality for the event \(A_N = \{W_N^{\beta_N}(f) > \frac1\rho\}\), but in that case estimating \(\PP(A_N)\) is equivalent to the initial problem.
The idea is thus to find some statistic \(X_N(f)\) to serve as a proxy for \(W_N^{\beta_N}(f)\), and to construct an event \(A_N=A_N(f)\) based on this \(X_N(f)\).
More precisely, \(X_N(f)\) will be some random variable with \(\EE[X_N(f)]=0\) and \(\tEE_f[X_N(f)] >0\), and we define 
\[
A_N(f) = \Big\{ X_N(f) > \frac12 \tEE[X_N(f)] \Big\} \,.
\]
Thus, Chebychev's inequality gives that 
\begin{equation*}
  \PP(A_N(f)) \leq 4 \frac{\Var(X_N(f))}{\tEE_f[X_N(f)]^2} \,,
  \qquad
  \tPP_f(A_N(f)^c) \leq 4 \frac{\tVV_f(X_N(f))}{\tEE_{f}[X_N(f)]^2} \,,
\end{equation*}
so that, combined with~\eqref{eq:change-of-measure}, we end up with
\begin{equation}
  \label{eq:change-of-measure-2}
  \EE\Big[ W_N^{\beta_N}(f)\wedge \frac{1}{\rho} \Big] \leq 32\, \frac{\rho^{-1} \Var(X_N(f)) }{\tEE_{f}[X_N(f)]^2}   +  32\, \frac{\tVV_f(X_N(f))}{\tEE_{f}[X_N(f)]^2}   \,.
\end{equation}
It thus remains to define \(X_N(f)\) and obtain a lower bound on \(\tEE_{f}[X_N(f)]\) and upper bounds on \(\Var(X_N)\), \(\tVV_f(X_N(f))\) so that we can bound~\eqref{eq:change-of-measure-2} uniformly for \(f\in \mathcal{M}_1(\sqrt{N})\).

\subsection{Choice of the proxy \texorpdfstring{\(X_N(f)\)}{} for \texorpdfstring{\(W_N^{\beta_N}(f)\)}{}}

We now define our proxy for \(W_N^{\beta_N}(f)\) by truncating the chaos expansion~\eqref{def:chaos-expansion} in several ways: (i)~we only keep subsets \(A\) with \(|A|\leq 1+k_N\) for \(k_N \coloneqq (\log \log N)^2\) (in fact, any choice \(\log \log N \ll k_N \ll \log N\) would do the trick); (ii)~we only keep subsets \(A\) of width smaller than some \(M = M_N\ll N\) (see~\eqref{def:MN} below for its definition).
This is inspired by~\cite{BCT25}, and we refer to its Section~4.3 for more comments as to why this is a natural proxy.
More precisely, we introduce
\begin{equation}
  \label{def:proxy}
  X_N(f) = \sum_{n =\rho N}^{\frac12 N} \sum_{x\in \Z^2}  X_{n,x}(f) \,,
\qquad
\text{with}
\quad
X_{n,x}(f) \coloneqq  \sum_{A \in \mathcal{I}_{n,x}} q^{(f)}(A) \xi(A) \,,
\end{equation}
with \(\mathcal{I}_{n,x} =\mathcal{I}_{n,x}^{(N)} \) the set of (non-empty) subsets \(A\) which start at \((n,x)\), with \(|A|\leq k_N\) and of width smaller than \(M_N\), namely
\[
\mathcal{I}_{n,x} \coloneqq \Big\{ A \subseteq \llb n,n+M_N \rrb \times \mathbb{Z}^2 ,\, \mathrm{start}(A) = (n,x) \,, 1\leq |A| \leq 1+k_N \Big\} \,,
\]
where we have denoted \(\mathrm{start}(A) = (n_1,x_1)\) if \(A= \{(n_1,x_1), \ldots, (n_k,x_k)\}\) with \(n_1<\ldots<n_k\).

Here, \(M=M_N\) is defined as follows: let \(\hat\beta' \coloneqq \frac{\hat \beta_c+\hat\beta}{2} \in (\hat\beta_c,\hat\beta)\) and set 
\begin{equation}
  \label{def:MN}
  M_N \coloneqq N^{\gamma} \qquad \text{ with }
  \gamma \coloneqq
  \begin{cases}
    (\hat\beta'/\hat\beta)^2 & \  \text{ under~\eqref{finite-sum}},\\
    (\hat\beta'/\hat\beta)^{2/(a+2)}& \  \text{ under \eqref{def:h}},
  \end{cases}
  \quad \gamma \in (0,1)\,.
\end{equation}
This way, recalling Assumption~\ref{hyp:scaling}, we have that 
\begin{equation}
  \label{eq:beta-MN}
  \beta_N = \hat\beta'\frac{\mathfrak{C}_h}{\sqrt{\log M_N}} \quad \text{ under~\eqref{finite-sum}}\,, 
  \qquad 
  \beta_N = \hat\beta' \frac{\mathfrak{C}_a}{ (\log M_N)^{(a+2)/2}} \quad \text{ under~\eqref{def:h}}\,.
\end{equation}
In other words, \((W_{M_N}^{\beta_N})_{N\geq 1}\) lies in the super-critical regime since \(\hat\beta'>\hat\beta_c\), but is ``strictly less super-critical'' than \((W_{N}^{\beta_N})_{N\geq 1}\), in the sense that \(\beta_N\) corresponds to some \(\hat\beta' <\hat\beta\) for \(M_N\).


\subsection{Variance, size-biased expectation and size-biased variance}
\label{sec:conclusion-proof}

We are now ready to give the main estimates on \(\tEE_f[X_N(f)]\) and \(\Var(X_N(f))\), \(\tVV_f( X_N(f))\), which hold uniformly in \(f\in \cM_1(\sqrt{\rho N})\).
For this, we introduce the following notation: for \(m\in \N\) and \(k\in \N\), define
\[
W_{m,\leq k}^{\beta_N} (x) = \sum_{A \subseteq \llb 1,m \rrb \times \Z^2,\, |A|\leq k} q^{(x)}(A) \xi(A)
\] 
a \(k\)-truncated version of \(W_m^{\beta_N} (x)\) (recall~\eqref{def:chaos-expansion} for the chaos expansion of \(W_m^{\beta_N}(x)\)).
We also let
\begin{equation}
  \label{def:V}
  \mathcal{V}_{m,\leq k}^{\beta_N} (x) \coloneqq \EE\big[ W_{m,\leq k}^{\beta_N}(x) W_{m,\leq k}^{\beta_N}(0)\big] =  \sum_{A,B \subseteq \llb 1,m \rrb \times \Z^2,\, |A|, |B|\leq k} q(A) q^{(x)}(B) \EE[\xi(A)\xi(B)]\,.
\end{equation}

We now express \(\tEE_f[X_N(f)]\) and \(\Var(X_N(f))\), \(\tVV_f( X_N(f))\) in terms of (sums of) the truncated covariances \(\mathcal{V}_{m,\leq k}^{\beta_N} (x)\), via a series of general results, see Lemmas~\ref{lem:tildeEX}, \ref{lem:diagonal-Var} and \ref{lem:off-diagonal-Var}.
Our goal is to clarify the role of \(\mathcal{V}_{m,\leq k}^{\beta_N} (x)\) in the different estimates.
Along the way, we highlight the core properties of the (truncated) covariances that are needed to conclude the proof of Theorem~\ref{thm:key-theorem}, see Claims~\ref{claim:sup-V}, \ref{claim:ratio} and~\ref{claim:lower-bound}.
The proofs of Lemmas~\ref{lem:tildeEX}, \ref{lem:diagonal-Var} and \ref{lem:off-diagonal-Var} are postponed to Section~\ref{sec:proof-variance} while the proofs of Claims~\ref{claim:sup-V}, \ref{claim:ratio} and~\ref{claim:lower-bound} are postponed to Section~\ref{sec:estim-V}.

\begin{remark}
  Let us stress that most of the following technical estimates become almost trivial in the case of an i.i.d.\ disorder, when \(h(x) = \ind_{\{x=0\}}\).
  In particular, only the term \(\mathcal{V}_{m,\leq k}^{\beta_N} (0)\) would appear in the following lemmas, and several claims or technical lemmas become unnecessary. 
  We will keep in mind this i.i.d.\ setting to see how our method greatly simplifies the approach in~\cite{BL17} --- this may be hidden by the technicalities of our proofs needed to deal with the spatially correlated case.
\end{remark}

We start with two estimates on the truncated covariances that are useful in the proof of the lemmas.
Claim~\ref{claim:sup-V} complements~\eqref{eq:sup-covariances} in the supercritical setting and in fact only relies on subcritical estimates (it is an easy consequence of~\eqref{eq:sup-covariances}) while Claim~\ref{claim:ratio} does not hide anything deep (in fact it is empty in the i.i.d.\ setting).

\begin{claim}
  \label{claim:sup-V}
  Assume that \((\beta_N)_{N\geq 0}\) follows the scaling of Assumption~\ref{hyp:scaling}, with \(\hat\beta\geq \hat{\beta}_c\).
  Then, there is a constant \(C_{\hat \beta}>1\) such that for any \(m\leq N\) and any \(k\geq 1\)
  \[
  1\leq \mathcal{V}_{m,\leq k}^{\beta_N} (x) \leq (C_{\hat\beta})^{k} \qquad \text{ for all } x\in \Z^2 \,.
  \]
\end{claim}

\begin{claim}
  \label{claim:ratio}
  Set \(T_N\coloneqq +\infty\) under~\eqref{finite-sum} and \(T_N\coloneqq \sqrt{N \log N}\) under~\eqref{def:h}.
  Then, for any sequence \((t_N)_{N\geq 1}\) with \(t_N\leq T_N\) verifying \(t_N= N^{\frac12+o(1)}\),
  \[
  \lim_{N\to\infty}  \frac{\sum_{|z|\leq t_N} h(z) \mathcal{V}_{M_N,\leq k_N}^{\beta_N} (z) }{\sum_{|z|\leq T_N} h(z) \mathcal{V}_{M_N,\leq k_N}^{\beta_N} (z) } =1 \,.
  \]
  The same holds true if both sums are restricted to \(z\in \Zeven^2\).
\end{claim}

\subsubsection{About the variance and size-biased expectation of \texorpdfstring{\(X_N(f)\)}{}}

We have the following lemma to deal with the first term in~\eqref{eq:change-of-measure-2}; here \(T_N\) is as defined in Claim~\ref{claim:ratio}.

\begin{lemma}[Size-biased expectation]
  \label{lem:tildeEX}
  For any \(f \in \mathcal{M}_1\), we have
  \[
  \tEE_f[X_N(f)] = \Var(X_N(f)) \,.
  \]
  Additionally, there are some constants \(c,C>0\) such that
    \begin{equation}
    \label{eq:upper-variance}
    \sup_{f \in \mathcal{M}_1(\sqrt{\rho N})} \Var(X_N(f)) \leq C \, \log\Big(\frac{2}{\rho}\Big) \;  \beta_N^2 \sum_{z\in \Zeven^2, |z| \leq T_N} h(z)\mathcal{V}_{M_N,\leq k_N}^{\beta_N}(z)\,,
  \end{equation}
  and
  \begin{equation}
    \label{eq:lower-tEE}
    \inf_{f \in \mathcal{M}_1(\sqrt{\rho N})} \tEE_f[X_N(f)] \geq c\,  \log\Big(\frac{2}{\rho}\Big) \, \beta_N^2 \sum_{z\in \Zeven^2,|z|\leq T_N} h(z)\mathcal{V}_{M_N,\leq k_N}^{\beta_N}(z)\,.
  \end{equation}
\end{lemma}

This shows that the first part in~\eqref{eq:change-of-measure-2} verifies
\begin{equation}
  \label{key-1}
  \sup_{f\in \mathcal{M}_1(\sqrt{\rho N})} \frac{\rho^{-1} \Var(X_N(f))}{\tEE_{f}[X_N(f)]^2}   \leq C' \bigg( \rho \log \Big(\frac2\rho\Big) \beta_N^2 \sum_{z\in \Zeven^2, |z| \leq T_N} h(z)\mathcal{V}_{M_N,\leq k_N}^{\beta_N}(z) \bigg)^{-1} \,.
\end{equation}
We now only need to show the following lower bound on \(\mathcal{V}_{M_N,\leq k_N}^{\beta_N}(z)\).
It relies the fact that \((W_{M_N}^{\beta_N})_{N\geq 1}\) lies in the super-critical regime (recall that \(\hat\beta'>\hat\beta_c\) in~\eqref{eq:beta-MN}).

\begin{claim}
  \label{claim:lower-bound} 
  For any \(\hat\beta'>\hat\beta_c\) in~\eqref{eq:beta-MN}, we have that for any \(b>0\)
  \[
  \lim_{N\to\infty} (\log N)^{-b}\, \mathcal{V}_{M_N,\leq k_N}^{\beta_N}(0) = +\infty  \,.
  \]
\end{claim}

\noindent
This shows in particular that \(\rho \log(\frac{2}{\rho}) \beta_N^2 \sum_{|z|\leq T_N} h(z)\mathcal{V}_{M_N,\leq k_N}^{\beta_N}(z) \geq \rho \beta_N^2 \mathcal{V}_{M_N,\leq k_N}^{\beta_N}(0) \to\infty\) as \(N\to\infty\), recalling the scaling of \(\beta_N\) from Assumption~\ref{hyp:scaling} and our choice \(\rho=(\log N)^{-1}\).
All together, this shows that the right-hand side of~\eqref{key-1} goes to \(0\).

\subsubsection{About the size-biased variance \texorpdfstring{\(X_N(f)\)}{}}

We now turn to estimating the second part in~\eqref{eq:change-of-measure-2}, namely the size-biased variance of \(X_N(f)\).
We split \(\tVV_f(X_N(f))\) into two different contributions: the diagonal and the off-diagonal part.
More precisely, recalling the definition~\eqref{def:proxy} of \(X_N(f)\) and expanding the covariances, we have that
\[
\begin{split}
  \tVV_f(X_N(f)) & = \bigg( \sumtwo{n_1,n_2=\rho N}{|n_1-n_2| \leq 2 M_N}^{\frac12 N} +  \sumtwo{n_1,n_2=\rho N}{|n_1-n_2|>2 M_N}^{\frac12 N} \bigg)\sum_{x_1,x_2 \in \Z^2} \tCC_f(X_{n_1,x_1}(f),X_{n_2,x_2}(f))  \\
  & =: \quad \mathrm{Diag}_{f} \quad + \quad \mathrm{Off\text{-}Diag}_{f} \,.
\end{split}
\]
We then have two lemmas dealing with the two terms.

\begin{lemma}[Diagonal contribution]
  \label{lem:diagonal-Var}
  Define \(\overline{\mathrm{Diag}}_{f} \coloneqq \mathrm{Diag}_{f} - \Var(X_N(f))\).
  Then there is a constant \(C = C_{\hat{\beta}}>0\) such that 
  \[
   \sup_{f\in \cM_1(\sqrt{\rho N})}\overline{\mathrm{Diag}}_{f} \leq  C^{k_N} \Big(\frac{\rho N}{M_N}\Big)^{-1/2}\, .
  \]
\end{lemma}

Combined with the lower bound~\eqref{eq:lower-tEE} and Claim~\ref{claim:lower-bound} (from which we get that \(\tEE_f[X_N(f)]^2 \geq 1\) for \(N\) sufficiently large), Lemma~\ref{lem:diagonal-Var} shows that
\begin{equation}
  \label{key-3}
  \sup_{f\in \mathcal{M}_1(\sqrt{\rho N})}  \frac{\overline{\mathrm{Diag}}_{f}}{\tEE_f[X_N(f)]^2} \leq  C^{k_N} \rho^{-1/2} \Big(\frac{N}{M_N}\Big)^{-1/2}  = N^{o(1)} N^{-(1-\gamma)/2} \,.
\end{equation}
For the last identity, we have used that \(\rho = N^{o(1)}\) and \(C^{k_N} = N^{o(1)}\) since \(k_N = (\log\log N)^2 \ll \log N\), and also the definition~\eqref{def:MN} of \(M_N=N^{\gamma}\).
Since \(\gamma <1\), the term~\eqref{key-3} goes to~\(0\).
Note that the diagonal term \(\mathrm{Diag}_f\) is actually equal to \(\overline{\mathrm{Diag}}_{f}+\Var(X_N(f))\), so the diagonal contribution to the second term in~\eqref{eq:change-of-measure-2} also contains a term \(\Var(X_N(f))/\tEE[X_N(f)]^2\), but this is treated as in~\eqref{key-1}.

We now turn to the off-diagonal term (its proof is the most technical part of the article). 
Again, let \(T_N\) be defined in Claim~\ref{claim:ratio} and let us introduce
\begin{equation}
  \label{eq:epsilon-N}
  \gep_N = \gep_N(\rho) \coloneqq  \frac{\sum_{ \sqrt{\rho N}\leq |z| \leq T_N} h(z)\mathcal{V}_{M_N,\leq k_N}^{\beta_N}(z)}{\sum_{|z| \leq T_N} h(z)\mathcal{V}_{M_N,\leq k_N}^{\beta_N}(z)}  \,,
\end{equation}
for which Claim~\ref{claim:ratio} shows that \(\lim_{N\to\infty}\gep_N =0\), since \(\sqrt{\rho N} = N^{\frac12+o(1)}\).

\begin{lemma}[Off-diagonal contribution]
  \label{lem:off-diagonal-Var}
  There is a constant \(C>0\) such that 
  \[
   \sup_{f\in \cM_1(\sqrt{\rho N})}  \mathrm{Off\text{-}Diag}_{f} \leq \Big( \gep_N+ \log\Big(\frac{2}{\rho}\Big)^{-1} \Big)\, \bigg( \log \Big(\frac{2}{\rho}\Big) \,\beta_N^2 \sum_{z\in \Zeven^2, |z|\leq T_N} h(z)\mathcal{V}_{M_N,\leq k_N}^{\beta_N}(z) \bigg)^2 \,.
  \]
\end{lemma}

Therefore, combined with the lower bound~\eqref{eq:lower-tEE}, this lemma shows that
\begin{equation}
  \label{key-2}
  \sup_{f\in \mathcal{M}_1(\sqrt{\rho N})} \frac{\mathrm{Off\text{-}Diag}_{f}}{\tEE_f[X_N(f)]^2} \leq C \Big( \gep_N + \log \Big(\frac{2}{\rho}\Big)^{-1} \Big)  \,.
\end{equation}
This term also goes to \(0\) since \(\rho \downarrow 0\) and \(\lim_{N\to\infty}\gep_N =0\).

\smallskip
All together, combining the estimates~\eqref{key-1}-\eqref{key-2}-\eqref{key-3} yields that the right-hand side of~\eqref{eq:change-of-measure-2} goes to \(0\) uniformly in \(f\in\mathcal{M}_1(\sqrt{N})\).
This concludes the proof of Theorem~\ref{thm:key-theorem}.
\qed

\section{Proof of the variance(s) estimates}
\label{sec:proof-variance}

In this section, we prove the bounds on \(\tEE_f[X_N(f)]\), \(\Var(X_N(f))\) from Lemma~\ref{lem:tildeEX} and then we turn to the estimates on \(\tVV_f(X_N(f))\) from Lemmas~\ref{lem:diagonal-Var} and~\ref{lem:off-diagonal-Var}.

\subsection{Proof of Lemma~\ref{lem:tildeEX}}

Before we start the proofs, let us introduce some notation and make a few observations.
For \(A\subseteq \N\times \Z^2\) of the form \(\{(n_1,x_1), \ldots (n_k,x_k)\}\) with \(n_1<\cdots <n_k\), we denote by \(\hat{A} = A-(n_1,0)\) the translated set (with the first point removed), that is
\[
\hat{A} \coloneqq \{ (n_2-n_1,x_2), \ldots (n_k-n_1,x_k) \} \,.
\]
We then have \(q^{(f)}(A) = q^{(f)}_{n_1}(x_1) q^{(x_1)}(\hat{A})\), \textit{i.e.}\ we decompose \(q^{(f)}(A)\) into a first transition kernel \(q_{n_1}^{(f)}(x_1) = \sum_{x\in \Z^2} f(x) q_{n_1}(x_1-x)\) and then some kernel \(q^{(x_1)}(\hat A)\) ``internal to \(A\)''.

Recall also that, by time-independence, \(\EE[\xi(A)\xi(B)]=0\) except if \(A,B\) share the same time-indices.
If \(A,B\) share the same time-indices (in particular \(|A|=|B|\)) and are non-empty, we can write, by time-independence and translation invariance
\begin{equation}
  \label{eq:correl-decomp}
  \EE[\xi(A) \xi(B)] = \EE\big[ \xi_{n_1,x_1} \xi_{n_1,y_1}\big] \, \EE\big[\xi(\hat{A}) \xi(\hat{B})\big] = \beta_N^2 h_N(x_1-y_1)\, \EE\big[\xi(\hat{A}) \xi(\hat{B})\big] \,,
\end{equation}
where \((n_1,x_1)=\mathrm{start}(A)\) and \((n_1,y_1) = \mathrm{start}(B)\) and we recall that \(h_N(z)\coloneqq \beta_N^{-2} (\e^{\beta_N h(z)}-1)\).

With the above observations, we notice that \(\EE[X_{n,x}(f)X_{n',x'}(f)]=0\) if \(n\neq n'\).
When, \(n=n'\), using \eqref{eq:correl-decomp}, we have that 
\begin{equation*}
  \EE[X_{n,x}(f)X_{n,y}(f)] = \sum_{A \in \mathcal{I}_{n,x},B\in \mathcal{I}_{n,y}} q_n^{(f)}(x) q_n^{(f)}(y)\, \beta_N^2 h_N(x-y)\, q^{(x)}(\hat{A}) q^{(y)}(\hat{B}) \EE\big[ \xi(\hat A) \xi(\hat{B}) \big] \,.
\end{equation*}
Recalling the definition~\eqref{def:V} of \(\mathcal{V}_{m,\leq k}^{\beta}\) and that of \(\mathcal{I}_{n,x}\), by translation invariance we obtain
\[
\sum_{A \in \mathcal{I}_{n,x},B\in \mathcal{I}_{n,y}}  q^{(x)}(\hat{A}) q^{(y)}(\hat{B}) \EE\big[ \xi(\hat A) \xi(\hat{B}) \big] = \mathcal{V}_{M_N,\leq k_N}^{\beta_N} (x-y) \,.
\]
We thus end up with
\begin{equation}
  \label{eq:correl-Xnx}
  \EE[X_{n,x}(f)X_{n,y}(f)] = q_n^{(f)}(x) q_n^{(f)}(y) \, \beta_N^2\, h_N(x-y) \, \mathcal{V}_{M_N,\leq k_N}^{\beta_N} (x-y) \,.
\end{equation}
Note again here that, because of parity issues the above is non-zero only if \(x-y \in \Zeven^2\) (recall that~\(f\) has support in \(\Zeven^2\)).

\begin{remark}
  In the i.i.d.\ setting, this simplifies to \(\EE[X_{n,x}(f)X_{n,y}(f)] = 0\) if \(x\neq y\) and to \(\EE[X_{n,x}(f)^2] = q_n^{(f)}(x)^2 \beta_N^{2} \mathcal{V}_{M_N,\leq k_N}^{\beta_N}(0)\) if \(x=y\).
\end{remark}

\begin{proof}[Proof of Lemma~\ref{lem:tildeEX}]
Using the chaos decomposition~\eqref{def:chaos-expansion} of \(W_N^{\beta}(f)\), we get that 
\begin{equation}
  \label{eq:tiltedE-1}
  \begin{split}
    \tEE_f[X_{n,x}(f)] = \EE[X_{n,x}(f) W_N^{\beta_N}(f)]& = \sum_{A\in \mathcal{I}_{n,x}} \sum_{y\in\Z^2} \sum_{ B \in \mathcal{I}_{n,y}} q^{(f)}(A) q^{(f)}(B) \EE[\xi(A) \xi(B)] \\
    & = \sum_{y \in \Z^2}  \EE[X_{n,x}(f) X_{n,y}(f)] \,.
  \end{split}
\end{equation}
By linearity of the expectation and recalling the definition~\eqref{def:proxy} of \(X_{N}(f)\), we thus get that 
\begin{equation}
  \label{eq:variance}
  \tEE_f[X_{N}(f)] = \sum_{n=\rho N}^{\frac12 N} \sum_{x,y\in \Z^2} \EE[X_{n,x}(f)X_{n,y}(f)] = \Var(X_N(f)) \,,
\end{equation}
the last identity simply coming the fact that \(\EE[X_{n,x}(f)X_{n',x'}(f)]=0\) when \(n\neq n'\).

Now, using~\eqref{eq:correl-Xnx}, we get that  
\begin{equation}
  \label{eq:XnxXny}
  \sum_{x,y\in \Z^2} \EE[X_{n,x}(f)X_{n,y}(f)] =  \sum_{ x, y\in \Z^2} q^{(f)}_n(x) q^{(f)}_n(y) \,\beta_N^2 h_N(x-y) \mathcal{V}_{M_N,\leq k_N}^{\beta_N} (x-y) \,,
\end{equation}
and we now obtain upper and lower bounds on this sum.

\begin{remark}
  Again, the proof is almost trivial in the i.i.d.\ setting. Indeed, the sum is equal to \(\beta_N^2 \mathcal{V}_{M_N,\leq k_N}^{\beta_N} (0)\) times \(\sum_{x\in \Z^2} q^{(f)}_n(x)^2 = q^{(f\ast \tilde{f})}_n(0)\) with \(\tilde f(x) = f(-x)\), by the Chapman--Kolmogorov property.
  Then, \(q^{(f\ast \tilde{f})}_n(0)\) is easily estimated by the local limit theorem.
\end{remark}

\smallskip
\textbullet\
For the upper bound, we first split the sum according to wether \(|x-y|\leq T_N\) or \(|x-y|>T_N\).
In the first case, by the local limit theorem, we can bound \(q^{(f)}_n(x)\leq \frac{C}{n}\) uniformly for \(x\in \Z^2\).
Hence, summing over \(x\in \Z^2\) with \(|x-y|\leq T_N\) (and \(x-y\in \Zeven^2\)) and then summing \(q^{(f)}_n(y)\) over \(y\), we get that
\begin{equation*}
  \sum_{x,y\in \Z^2, |x-y|\leq T_N} \EE[X_{n,x}(f)X_{n,y}(f)] \leq   \frac{C}{n}  \, \Big(\beta_N^2  \sum_{z\in \Zeven^2, |z|\leq T_N} h(z) \mathcal{V}_{M_N,\leq k_N}^{\beta_N}(z) \Big) \,,
\end{equation*}
having also used that \(h_N(z)\leq C h(z)\).

On the other hand, note that the sum \(|x-y|> T_N\) is empty in the summable case~\eqref{finite-sum}, because then \(T_N=+\infty\).
We thus focus on the non-summable case~\eqref{def:h}, where \(T_N =\sqrt{N \log N}\).
We simply bound \(\mathcal{V}_{M_N,\leq k_N}^{\beta_N}(x-y) \leq C^{k_N}\) thanks to~\eqref{claim:sup-V} and also \(h_N(x-y) \leq C (\log N)^{a-1}/N\) for \(|x-y|> T_N\) by the asymptotic~\eqref{def:h}.
Summing over \(x,y\), we thus get from~\eqref{eq:XnxXny} that 
\begin{equation*}
  \sum_{x,y\in \Z^2, |x-y|> T_N} \EE[X_{n,x}(f)X_{n,y}(f)] \leq C \beta_N^2 \frac{C^{k_N} (\log N)^{a-1}}{N} \bP_f^{\otimes 2} \big( |S_n^{(1)} -S_n^{(2)}| > T_N \big) \,,
\end{equation*}
where \(\bP_f^{\otimes 2}\) denotes the law of two independent random walks with law \(\bP_f \coloneqq \sum_{x\in \Z^2} f(x) \bP_x\).
Now, since \(T_N =\sqrt{N \log N}\) and since the starting points of \(S^{(1)}\) and \(S^{(2)}\) are at distance at most \(2\sqrt{\rho N} \leq 2\sqrt{N}\) (recall that \(f\in \mathcal{M}_1(\sqrt{\rho N})\)), standard large deviations yield that
\[
\bP_f^{\otimes 2} \big( |S_n^{(1)} -S_n^{(2)}| > T_N \big)\leq \e^{-c \log N}  = N^{-c} \,,
\]
uniformly for \(n\leq N\).
We thus end up with 
\begin{equation*}
  \begin{split}
    \sum_{x,y\in \Z^2, |x-y|> T_N} \EE[X_{n,x}(f)X_{n,y}(f)] & \leq \frac{C}{N} \beta_N^2 C^{k_N} (\log N)^{a-1} N^{-c}  \\
    & \leq \frac{C}{n} \beta_N^2 \sum_{z\in \Zeven^2, |z|\leq T_N} h(z) \mathcal{V}_{M_N,\leq k_N}^{\beta_N}(z)\,,
  \end{split}
\end{equation*}
where the last inequality holds for \(N\) large enough, since \(C^{k_N} (\log N)^{a-1} N^{-c} \to 0\), recalling that \(k_N \ll \log N\); we have also used that \(\sum_{|z|\leq T_N} h(z) \mathcal{V}_{M_N,\leq k_N}^{\beta_N}(z) \geq 1\) and that \(n\leq N\).

All together, either in the summable or non-summable case, we have obtained that
\begin{equation}
  \label{eq:upper-sumXnxy}
  \sum_{x,y\in \Z^2} \EE[X_{n,x}(f)X_{n,y}(f)] \leq   \frac{C}{n}  \, \Big(\beta_N^2  \sum_{z\in \Zeven^2, |z|\leq T_N} h(z) \mathcal{V}_{M_N,\leq k_N}^{\beta_N}(z) \Big) \,,
\end{equation}

\smallskip
\textbullet\
For the lower bound, we restrict the sum over \(x,y\) in~\eqref{eq:XnxXny} to \(|x|,|x-y| \leq \sqrt{\rho N}\) so in particular \(|y|\leq 2\sqrt{n}\).
By the local limit theorem, we have that  \(q_n^{(f)}(x)\), \(q_n^{(f)}(y)\geq \frac{c}{n}\) uniformly for such \(x,y\) with \(|x|_1,|y|_1\) having the same parity as \(n\) (in particular \(x-y\in \Zeven^2\)), recalling that \(f\in \mathcal{M}_1(\sqrt{\rho N})\) and \(n \geq \rho N\).
All together, we get that
\begin{equation}
  \label{eq:lower-sumXnxy}
  \begin{split}
    \sum_{x,y\in \Z^2} \EE[X_{n,x}(f)X_{n,y}(f)] & \geq \sum_{|x|\leq \sqrt{n}}  \frac{c^2}{n^2}  \, \Big(\beta_N^2  \sum_{z\in \Zeven^2, |z|\leq \sqrt{\rho N}} h_N(z) \mathcal{V}_{M_N,\leq k_N}^{\beta_N}(z) \Big) \\
    & \geq \frac{c'}{n}  \, \Big(\beta_N^2  \sum_{z\in \Zeven^2, |z|\leq T_N} h(z) \mathcal{V}_{M_N,\leq k_N}^{\beta_N}(z) \Big) \,,
  \end{split}
\end{equation}
where we have also used Claim~\ref{claim:ratio} for the second inequality (recall we choose \(\rho=N^{o(1)}\)); we have also used that \(h_N(z)\geq h(z)\).

\smallskip
Now, plugging~\eqref{eq:upper-sumXnxy} or~\eqref{eq:lower-sumXnxy} back in \eqref{eq:variance}, the sum \(\sum_{n=\rho N}^{\frac12 N} \frac{1}{n}\) gives a factor \(\log(\frac2\rho)\), which concludes the proof of \eqref{eq:upper-variance},~\eqref{eq:lower-tEE}.
\end{proof}

\subsection{Diagonal contribution to the sized-biased variance: Proof of Lemma~\ref{lem:diagonal-Var}}

The first step is to show that for \(j \geq 0\), one can compare the sum of the covariances (near the diagonal with \(n_1,n_2 \in \llb j M_N +1, (j+3) M_N \rrb\)) with the third moment of some chaos expansion.

Let us denote 
\begin{equation}
  \label{eq:chaos-Wj}
  \overline{W}_{ (a,b], \leq k}^{\beta_N} (f)\coloneqq \sum_{A \subseteq \llb a+1, b \rrb \times \Z^2,\,  1\leq |A|\leq k} q^{(f)}(A) \xi(A) \,.
\end{equation}
Then, for any \(j\geq 0\), let us prove that
\begin{equation}
  \label{eq:third-moment}
  \begin{split}
    \sum_{n_1,n_2 = jM_N+1}^{(j+3)M_N} \sum_{x_1,x_2\in \Z^2}  \Big(\tEE_f\big[X_{n_1,x_1}(f) X_{n_2,x_2}(f) \big]&- \EE\big[X_{n_1,x_1}(f) X_{n_2,x_2}(f) \big]\Big)\\
    &  \leq  \EE\Big[ \overline{W}_{(j M_N,(j+4) M_N], \leq 2 k_N+2}^{\beta_N} (f)^3  \Big] \,.
  \end{split}
\end{equation}

\begin{proof}[Proof of~\eqref{eq:third-moment}]
  First, let us expand \(\tEE_f[X_{n_1,x_1}(f)X_{n_2,x_2}(f)] = \EE[X_{n_1,x_1}(f)X_{n_2,x_2}(f)W_N^{\beta}(f)]\) via the chaos expansion, see~\eqref{def:chaos-expansion}.
  Using that \(\EE[\xi(A_1)\xi(A_2)\xi(B)] =0\) if time indices of \(B\) are not included in the time-indices of \(A_1\cup A_2\), we obtain that 
  \[
  \begin{split}
      \tEE_f \big[X_{n_1,x_1}(f)& X_{n_2,x_2}(f) \big] =   \EE[X_{n_1,x_1}(f) X_{n_2,x_2}(f)]   \\
      & + \sumtwo{A_1\in \mathcal{I}_{n_1,x_1}}{A_2 \in \mathcal{I}_{n_2,x_2}} \sumtwo{B \subseteq \llb j M_N+1, (j+4) M_N \rrb \times \Z^2}{1\leq |B|\leq |A_1|+|A_2| \leq 2k_N+2}  q^{(f)}(A_1)q^{(f)}(A_2) q^{(f)}(B) \EE\big[ \xi(A_1)\xi(A_2) \xi(B)\big] \,,
  \end{split}
  \]
  where the first term accounts for the contribution of \(B =\emptyset\).

  Now, let us stress that all terms in the last sum are non-negative, since \(\EE[\xi_{n,x} \xi_{n,y} \xi_{n,z}]\geq 0\), either by a direct calculation\footnote{Is is equal to \(\e^{\beta_N^2 (h(x-y) + h(y-z) +h(z-x))} - \e^{\beta_N^2 h(x-y)} - \e^{\beta_N^2 h(y-z)}- \e^{\beta_N^2 h(z-x)} +2\), which can be checked to be non-negative.} or by using the Harris--FKG inequality since by Assumption~\ref{main-assumption} the \(\xi_{n,x}\) are non-decreasing functions of i.i.d.\ random variables.
  We now sum over \(n_1,n_2\in \llb j M_N+1, (j+4) M_N \rrb\) and $x_1,x_2\in\Z^2$.
  We can then enlarge the sum to all sets \(A_1,A_2 \subseteq \llb j M_N+1, (j+4) M_N \rrb \times \Z^2 \) keeping only the restriction \(|A_1|,|A_2|,|B|\leq 2k_N+2\) to get an upper bound.
  In the end, we obtain that 
  \begin{equation*}
    \begin{split}
     \sum_{n_1,n_2 = j M_N+1}^{(j+4)M_N} & \sum_{x_1,x_2\in \Z^2}   \Big(\tEE_f\big[X_{n_1,x_1}(f) X_{n_2,x_2}(f) \big]- \EE\big[X_{n_1,x_1}(f) X_{n_2,x_2}(f) \big]\Big) \\
      &\qquad  \leq  \sumtwo{A_1,A_2,B \subseteq \llb j M_N+1, (j+4) M_N \rrb \times \Z^2}{1\leq |A_1|,|A_2|,|B|\leq 2 k_N+2}  q^{(f)}(A_1)q^{(f)}(A_2) q^{(f)}(B) \EE\big[ \xi(A_1)\xi(A_2) \xi(B)\big]\,.
     \end{split}
  \end{equation*}
  This proves the claim, since the right-hand side is exactly the desired third moment.
\end{proof}

\begin{proof}[Proof of Lemma~\ref{lem:diagonal-Var}]
Simply bounding \(\tCC_f(X_{n_1,x_1}(f),X_{n_2,x_2}(f) \big)\leq \tEE_f[X_{n_1,x_1}(f)X_{n_2,x_2}(f)]\), we have that 
\[
\begin{split}
  \mathrm{Diag}_{f} & \leq  \sumtwo{n_1,n_2=\rho N}{|n_1-n_2| \leq 2 M_N}^{\frac12 N}\sum_{x_1,x_2\in \Z^2}\EE\big[X_{n_1,x_1}(f) X_{n_2,x_2}(f) \big]   \\
  &\qquad + \sumtwo{n_1,n_2=\rho N}{|n_1-n_2| \leq 2 M_N}^{\frac12 N} \sum_{x_1,x_2\in \Z^2} \Big(\tEE_f\big[X_{n_1,x_1}(f),X_{n_2,x_2}(f) \big] -\EE\big[X_{n_1,x_1}(f) X_{n_2,x_2}(f) \big] \Big) \,.
\end{split}
\]
Since \(\EE[X_{n_1,x_1}(f) X_{n_2,x_2}(f)]=0\) for \(n_1\neq n_2\), the first sum is in fact exactly \(\Var(X_N(f))\), see~\eqref{eq:variance}.
For the second sum, regrouping terms into blocks \(n_1,n_2 \in \llb j M_N +1, (j+3) M_N \rrb\) (possibly overcounting terms, note that \(\tEE_f[X_{n_1,x_1}(f),X_{n_2,x_2}(f)] -\EE[X_{n_1,x_1}(f) X_{n_2,x_2}(f)]\) is non-negative as noticed above) and using \eqref{eq:third-moment}, we get that 
\begin{equation}
  \label{eq:bound-Diag-1}
  \overline{\mathrm{Diag}}_{f}\leq \sum_{j= \rho N/M_N}^{N/M_N} \EE\Big[ \overline{W}_{(j M_N,(j+4) M_N], \leq 2 k_N+2}^{\beta_N} (f)^3  \Big]  \,.
\end{equation}

It thus remains to control the third moment of \(\overline{W}_{(j M_N,(j+4) M_N], \leq 2 k_N+2}^{\beta_N} (f)\).
For this, we use the hypercontractivity of Gaussian spaces, see~\cite[Thm.~5.1]{Jan97}. 
More precisely, consider the so-called Mehler transform \(\mathfrak{M}_r: f(\omega) \mapsto  \EE_{\tilde\omega}[f(r \omega + \sqrt{1-r^2} \tilde \omega)]\) for \(r\in (0,1)\), where \(\tilde \omega\) is an independent copy of \(\omega\).
Recalling that \(\xi^{(\beta)}_{n,x} = \e^{\beta \omega_{n,x} -\frac12 \beta^2}-1\), a straightforward calculation shows that \(\mathfrak{M}_r \xi^{(\beta)}_{n,x} = \xi^{(r\beta)}_{n,x}\): in particular, we have that
\[
\overline{W}_{(j M_N,(j+4) M_N], \leq 2 k_N+2}^{\beta_N} (f) =
\mathfrak{M}_{r} \overline{W}_{(j M_N,(j+4) M_N], \leq 2 k_N+2}^{\beta_N/r} (f) \,.
\]
Then, \cite[Thm.~5.8]{Jan97} tells that, for \(r^2\leq \frac{p-1}{q-1}\), \(\mathfrak{M}_r\) is a linear map from \(L^p(\PP)\) to \(L^q(\PP)\) and has norm~\(1\).
Applying this with \(p=2\), \(q=3\) and \(r=1/\sqrt{2}\), we thus obtain that 
\begin{equation}
  \label{eq:hypercontractivity}
  \EE\Big[ \Big|\overline{W}_{(j M_N,(j+4) M_N], \leq 2 k_N+2}^{\beta_N} (f) \Big|^3  \Big] \leq \EE\Big[ \overline{W}_{(j M_N,(j+4) M_N], \leq 2 k_N+2}^{\sqrt{2}\beta_N} (f)^2  \Big]^{3/2} \,,
\end{equation}
so that we have reduced to a second moment estimate.

Now, using the chaos expansion~\eqref{eq:chaos-Wj} and the fact that \(\EE[\xi(A) \xi(B)] =0\) if \(A,B\) do not have the same starting point, similarly to~\eqref{eq:correl-Xnx} we get that
\begin{equation*}
  \begin{split}
    \EE\Big[ \overline{W}_{(j M_N,(j+4) M_N], \leq 2 k_N+2}^{\sqrt{2} \beta_N}& (f)^2  \Big] \\
    & = \sum_{n = j M_N+1}^{(j+4)M_N} \sum_{x,y\in \Z^2} q^{(f)}_n(y) q^{(f)}_n(x) 2 \beta_N^2 h_N(x-y)  \mathcal{V}_{4 M_N,\leq 2k_N+2}^{\sqrt{2}\beta_N}(x-y)\,.
  \end{split}
\end{equation*}
(Here we have replaced \(\beta_N\) by \(\sqrt{2} \beta_N\) in the definition of \(h_N\).)

Then, with the exact same proof as in~\eqref{eq:upper-sumXnxy}, we get that (uniformly for \(f\in \mathcal{M}_1(\sqrt{\rho N})\))
\[
\sum_{x,y\in \Z^2} q^{(f)}_n(y) q^{(f)}_n(x) \beta_N^2 h_N(x-y)  \mathcal{V}_{4 M_N,\leq 2k_N+2}^{\sqrt{2} \beta_N}(x-y) \leq \frac{C}{n} \beta_N^2 \sum_{|z|\leq T_N} h(z) \mathcal{V}_{4 M_N,\leq 2k_N+2}^{\sqrt{2}\beta_N}(z) \,.
\]
Summing over \(n \in \{j M_N+1, \ldots, (j+4)M_N\}\), this gives that
\[
 \EE\Big[ \overline{W}_{(j M_N,(j+4) M_N], \leq 2 k_N+2}^{\beta_N} (f)^2  \Big] \leq \frac{C'}{j} \beta_N^2 \sum_{|z|\leq T_N} h(z)\mathcal{V}_{4 M_N,\leq 2k_N+2}^{\sqrt{2}\beta_N}(z)  \leq \frac{C'}{j} C^{2k_N+2}\,.
\]
For the last inequality, we have used that \(\mathcal{V}_{4 M_N,\leq 2k_N+2}^{\sqrt{2}\beta_N}(z) \leq C^{2k_n+2}\) by Claim~\ref{claim:sup-V}, together with the fact \(\beta_N^2 \sum_{|z|\leq T_N} h(z)\) remains bounded (in fact it goes to \(0\)) --- this is obvious in the summable case~\eqref{finite-sum} and follows from the fact that \(\sum_{|z|\leq T_N} h(z) \sim c (\log N)^{a+1}\) thanks to~\eqref{sum:h}, together with the scaling of \(\beta_N\) from Assumption~\ref{hyp:scaling}.

Going back to~\eqref{eq:hypercontractivity} and plugging this in~\eqref{eq:bound-Diag-1}, we get that 
\[
  \overline{\mathrm{Diag}}_{f} \leq (2C' \, C^{2 k_N+2})^{3/2} \bigg(\sum_{j=\rho N/M_N}^{N/M_N} j^{-3/2} \bigg)  \leq C''\, C^{3 k_N}  \Big(\frac{\rho N}{M}\Big)^{-1/2}\,.
\] 
This concludes the proof of Lemma~\ref{lem:diagonal-Var}, since all constants are uniform in \(f\in \mathcal{M}_1(\sqrt{\rho N})\). 
\end{proof}

\subsection{Off-Diagonal contribution to the sized-biased variance: Proof of Lemma~\ref{lem:off-diagonal-Var}}

Let us consider \(n_1,n_2\) with  \(\rho N\leq n_1 < n_2 \leq \frac12 N\) and \(n_2-n_1 >2 M_N\).
We write
\begin{equation*}
  \begin{split}
    \tCC_f[X_{n_1,x_1}&(f),X_{n_2,x_2}(f)] \\
    & = \EE[ X_{n_1,x_1}(f) X_{n_2,x_2}(f) W_N^{\beta_N}(f) ] -  \EE[X_{n_1,x_1}(f) W_N^{\beta_N}(f) ] \EE[X_{n_2,x_2}(f)W_N^{\beta_N}(f) ] \,.
  \end{split}
\end{equation*}
Then, expanding the chaos and using time-independence (we refer to Figure~\ref{fig:covariances} for an illustration), we get that this is equal to 
\begin{equation}
  \label{eq:covariances-chaos}
  \begin{split}
  \sumtwo{A_1\in \mathcal{I}_{n_1,x_1}}{A_2 \in \mathcal{I}_{n_2,x_2}} \sum_{y_1,y_2\in \Z^2} \sumtwo{B_1 \in \mathcal{I}_{n_1,y_1}}{B_2\in \mathcal{I}_{n_2,y_2}} q^{(f)}(A_1)q^{(f)}(A_2) & \big( q^{(f)}(B_1\cup B_2) - q^{(f)}(B_1) q^{(f)}(B_2) \big)  \\[-10pt]
  & \qquad \qquad \qquad   \EE\big[\xi(A_1) \xi(B_1)\big] \,\EE\big[\xi(A_2) \xi(B_2)\big] \,.
  \end{split}
\end{equation}

\begin{figure}
  \centering
  \begin{tikzpicture}[xscale=1.2, yscale=1, >=stealth]
    \draw[thick, ->] (0, 0) -- (9, 0);
    \draw[thick, ->] (0, 0) -- (0,3);
    \draw[thick] (3, 0.1) -- (3, -0.1) node[below, yshift=-0.1cm] {$n_1$};
    \draw[thick] (4.4, 0.1) -- (4.4, -0.1) node[below, yshift=-0.1cm] {\small $n_1+M_N$};
    \draw[thick] (7.5, 0.1) -- (7.5, -0.1) node[below, yshift=-0.1cm] {$n_2$};
    \draw[thick] (8.7, 0.1) -- (8.7, -0.1) node[below, yshift=-0.1cm] {\small $n_2+M_N$};
    \draw[thick, blue] (0, 0.9) to[out=30, in=180] (3, 2.8);
    \draw[thick] (0, 0.6) to[out=30, in=180] (3, 2.2);
    \draw[thick] (0, 0.2) to[out=0, in=180] (7.5, 1.5);
    \draw[thick, blue] (4.2, 2.8) to[out=0, in=180] (7.5, 2.0);
    \draw[thick] (3, 2.2) -- (3.3, 1.8) -- (3.6, 2.5) -- (3.9, 2.8) -- (4.2, 2.4);
    \draw[thick, blue] (3, 2.8) -- (3.3, 2.6) -- (3.6, 2.2) -- (3.9, 2.5) -- (4.2, 2.8);
    \foreach \p in {(3, 2.2), (3.3, 1.8), (3.6, 2.5), (3.9, 2.8), (4.2, 2.4)}
        \filldraw \p circle (2pt); 
    \foreach \p in {(3, 2.8), (3.3, 2.6), (3.6, 2.2), (3.9, 2.5), (4.2, 2.8)}
        \draw[thick, blue, fill=white] \p circle (2pt);
    \node[above] at (3.3, 2.6) {\blue $B_1$};
    \node[above right, xshift=-0.2cm] at (4.2, 2.8) {\small\blue $(\ell_1,v_1)$};
    \node[below] at (3.3, 1.8) {$A_1$};
    \draw[thick] (7.5, 1.5) -- (7.8, 1.4) -- (8.1, 1.9) -- (8.4, 1.4);
    \draw[thick, blue] (7.5, 2.0) -- (7.8, 1.9) -- (8.1, 1.4) -- (8.4, 2.0);
    \foreach \p in {(7.5, 1.5), (7.8, 1.4), (8.1, 1.9), (8.4, 1.4)}
        \filldraw \p circle (2pt);
        
    \foreach \p in {(7.5, 2.0), (7.8, 1.9), (8.1, 1.4), (8.4, 2.0)}
        \draw[thick, blue, fill=white] \p circle (2pt);
    \node[above] at (8.4, 2.0) {\blue $B_2$};
    \node[above left, xshift=0.2cm] at (7.5, 2.0) {\small\blue $(n_2,y_2)$};
    \node[below left, xshift=0.2cm] at (7.5, 1.5) {\small $(n_2,x_2)$};
    \node[below] at (8.4, 1.4) {$A_2$};
\end{tikzpicture}
\begin{tikzpicture}[xscale=1.2, yscale=1, >=stealth]
    \draw[thick, ->] (0, 0) -- (9, 0);
    \draw[thick, ->] (0, 0) -- (0,3);
    \draw[thick] (3, 0.1) -- (3, -0.1) node[below, yshift=-0.1cm] {$n_1$};
    \draw[thick] (4.4, 0.1) -- (4.4, -0.1) node[below, yshift=-0.1cm] {\small $n_1+M_N$};
    \draw[thick] (7.5, 0.1) -- (7.5, -0.1) node[below, yshift=-0.1cm] {$n_2$};
    \draw[thick] (8.7, 0.1) -- (8.7, -0.1) node[below, yshift=-0.1cm] {\small $n_2+M_N$};
    \draw[thick, blue] (0, 0.9) to[out=30, in=180] (3, 2.8);
    \draw[thick] (0, 0.6) to[out=30, in=180] (3, 2.2);
    \draw[thick] (0, 0.2) to[out=0, in=180] (7.5, 1.5);
    \draw[thick, blue] (0, 0.1) to[out=0, in=180] (7.5, 2.0);
    \draw[thick] (3, 2.2) -- (3.3, 1.8) -- (3.6, 2.5) -- (3.9, 2.8) -- (4.2, 2.4);
    \draw[thick, blue] (3, 2.8) -- (3.3, 2.6) -- (3.6, 2.2) -- (3.9, 2.5) -- (4.2, 2.8);
    \foreach \p in {(3, 2.2), (3.3, 1.8), (3.6, 2.5), (3.9, 2.8), (4.2, 2.4)}
        \filldraw \p circle (2pt); 
    \foreach \p in {(3, 2.8), (3.3, 2.6), (3.6, 2.2), (3.9, 2.5), (4.2, 2.8)}
        \draw[thick, blue, fill=white] \p circle (2pt);
    \node[above] at (3.3, 2.6) {\blue $B_1$};
    \node[above right, xshift=-0.2cm] at (4.2, 2.8) {\small\blue $(\ell_1,v_1)$};
    \node[below left, xshift=0.2cm] at (7.5, 1.5) {\small $(n_2,x_2)$};
    \node[below] at (3.3, 1.8) {$A_1$};
    \draw[thick] (7.5, 1.5) -- (7.8, 1.4) -- (8.1, 1.9) -- (8.4, 1.4);
    \draw[thick, blue] (7.5, 2.0) -- (7.8, 1.9) -- (8.1, 1.4) -- (8.4, 2.0);
    \foreach \p in {(7.5, 1.5), (7.8, 1.4), (8.1, 1.9), (8.4, 1.4)}
        \filldraw \p circle (2pt);
        
    \foreach \p in {(7.5, 2.0), (7.8, 1.9), (8.1, 1.4), (8.4, 2.0)}
        \draw[thick, blue, fill=white] \p circle (2pt);
    \node[above] at (8.4, 2.0) {\blue $B_2$};
    \node[above left, xshift=0.2cm] at (7.5, 2.0) {\small\blue $(n_2,y_2)$};
    \node[below] at (8.4, 1.4) {$A_2$};
\end{tikzpicture}
  \caption{
  Illustration of the expansion~\eqref{eq:covariances-chaos} of \(\tCC_f[X_{n_1,x_1}(f),X_{n_2,x_2}(f)]\). 
  The picture on top represents the decomposition of \(\EE[ X_{n_1,x_1}(f) X_{n_2,x_2}(f) W_N^{\beta_N}(f) ]\): since \(n_2>n_1+M_N\), it can be split into a first contribution between \(n_1\) and \(n_1+M_N\) (with subsets \(A_1,B_1\)) and a second contribution between \(n_2\) and \(n_2+M_N\) (with subsets \(A_2,B_2\)), the terms \(\EE[\xi(A_1) \xi(A_2) \xi(B_1\cup B_2)]\) factorizing to \(\EE[\xi(A_1)\xi(B_1)] \EE[ \xi(A_2) \xi(B_2)]\), by time-independence.
  The picture on the bottom represents the decomposition of \(\EE[ X_{n_1,x_1}(f) W_N^{\beta_N}(f) ] \EE[ X_{n_2,x_2}(f) W_N^{\beta_N}(f)]\), which is the product of a first contribution between \(n_1\) and \(n_1+M_N\) (with subsets \(A_1,B_1\)) and a second contribution between \(n_2\) and \(n_2+M_N\) (with subsets \(A_2,B_2\)).
  Subtracting the two terms give~\eqref{eq:covariances-chaos}.
  }
  \label{fig:covariances}
\end{figure}
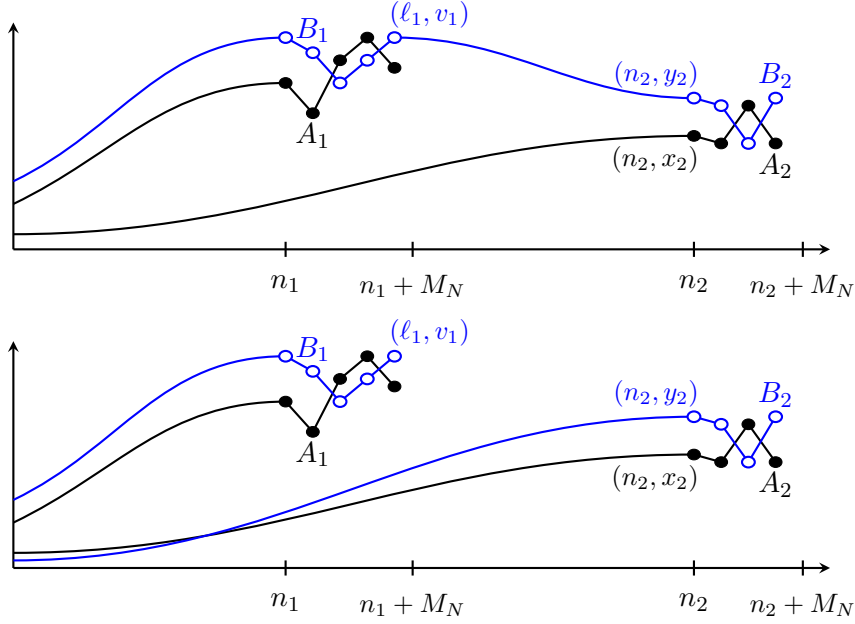

Recall that if \(A\) is a non-empty set with starting point \((n_1,x_1)\) we denote by \(\hat A = A -(n_1,0)\) (with the first point being removed), so that in particular \(q^{(f)}(A) = q^{(f)}_{n_1}(x_1) q^{(x_1)}(\hat A)\).
With this notation, we may write 
\[
q^{(f)}(B_1\cup B_2) - q^{(f)}(B_1) q^{(f)}(B_2)
  = q^{(f)}(B_1) \big( q_{n_2-\ell_1}(y_2-v_1) - q^{(f)}_{n_2}(y_2)\big) q^{(y_2)}(\hat{B}_2) \,, 
\]
where we have denoted by \((\ell_1,v_1)\) the last point of \(B_1\) and by \((n_2,y_2)\) the first point of \(B_2\) (one may refer to Figure~\ref{fig:covariances} for an illustration).
Note also that, recalling~\eqref{eq:correl-decomp}, we have that \(\EE[\xi(A_2) \xi(B_2)]= \beta_N^2 h_N(x_2-y_2) \EE[\xi(\hat A_2) \xi(\hat B_2)]\).
Overall, summing over \(\hat{A}_2,\hat{B}_2\), analogously to~\eqref{eq:correl-Xnx} we get that~\eqref{eq:covariances-chaos} is equal to
\begin{equation*}
  \begin{split}
    \sum_{A_1\in \mathcal{I}_{n_1,x_1}} \sum_{y_1\in \Z^2} & \sum_{B_1 \in \mathcal{I}_{n_1,y_1}} 
     q^{(f)}(A_1) q^{(f)}(B_1)\EE[\xi(A_1) \xi(B_1)] \\
    & \sum_{y_2\in \Z^2} q^{(f)}_{n_2}(x_2) \big( q_{n_2-\ell_1}(y_2-v_1) - q^{(f)}_{n_2}(y_2)\big) \beta_N^2 h_N(x_2-y_2) \cV_{M_N,\leq k_N}^{\beta_N}(x_2-y_2) \,.
  \end{split}
\end{equation*}
Summing over \(x_1,x_2\) and using a change of variable\footnote{Note that \(x_2-y_2\in \Zeven^2\) and \(x_1-y_1 \in \Zeven^2\) because of parity issues, but we keep this restriction implicit in the sums, to lighten notation.} \(z=x_2-y_2\), we thus have that
\begin{equation}
  \label{eq:covariance-1}
  \begin{split}
    \sum_{x_1,x_2 \in \Z^2}   \tCC_f[X_{n_1,x_1}(f),X_{n_2,x_2}(f)] 
    = \sum_{x_1,y_1 \in \Z^2}& \sumtwo{A_1\in \mathcal{I}_{n_1,x_1}}{ B_1 \in \mathcal{I}_{n_1,y_1}}  q^{(f)}(A_1) q^{(f)}(B_1)\EE[\xi(A_1) \xi(B_1)] \\
    & \quad  \times \beta_N^2 \sum_{z\in \Zeven^2}  h_N(z) \cV_{M_N,\leq k_N}^{\beta_N}(z) \mathbf{S}_{n_2}(z)\,,
  \end{split}
\end{equation}
where \(\mathbf{S}_{n_2}(z) = \mathbf{S}_{n_2}^{(\ell_1,v_1)}(z)  \coloneqq  \sum_{y_2\in \Z^2} q^{(f)}_{n_2}(y_2+z) ( q_{n_2-\ell_1}(y_2-v_1) - q^{(f)}_{n_2}(y_2))\).
Note that by the Chapman--Kolmogorov property (and the symmetry of the random walk), we have that 
\[
\mathbf{S}_{n_2}(z)= q^{(f)}_{2n_2-\ell_1}(z-v_1) - q^{(f\ast \tilde{f})}_{2n_2}(z) \,,
\]
with the notation \(\tilde{f}(x) =f(-x)\).
We now split the last sum in~\eqref{eq:covariance-1} over~\(z\) into \(|z|\leq \sqrt{\rho N}\), \(\sqrt{\rho N} < |z| \leq T_N\) and \(|z|>T_n\).

\begin{remark}
  Notice that in the i.i.d.\ case we do not need to treat the different cases since the sum is reduced to \(z=0\).
  Then, we can bound \(\mathbf{S}_{n_2}(0) \leq q^{(f)}_{2n_2-\ell_1}(0) - q^{(f\ast \tilde{f})}_{2n_2}(0) \leq C\,\frac{n_1}{(n_2)^2} \), as done in \cite[Eq.~(6.4)]{BCT25}.
\end{remark}

\smallskip
\noindent
\textit{(i) If \(|z|\leq \sqrt{\rho N}\). }
Then, applying the local limit theorem \cite[Thm.~1.2.1]{Law96}, we have on the one hand
\[
q^{(f)}_{2n_2-\ell_1}(z-v_1) = \sum_{x\in \Z^2} f(x) \frac{2}{\pi (2n_2-\ell_1)} \e^{- \frac{|z-v_1-x|^2}{2n_2-\ell_1}}  + O\Big( \frac{1}{(2n_2-\ell_1)^2}\Big) \leq \frac{2}{\pi (2n_2-\ell_1)}  +  O\Big( \frac{1}{(n_2)^2}\Big)\,,
\]
simply bounding the exponentials by \(1\) and using that \(\sum_{x\in \Z^2}f(x)=1\); we have also used that \(2n_2-\ell_1\geq n_2\) for the second term.

On the other hand, we have that 
\[
q^{(f\ast\tilde{f})}_{2n_2}(z) = \sum_{x\in \Z^2} f\ast\tilde{f}(x) \frac{1}{\pi n_2} \e^{- \frac{|z-x|^2}{2n_2}}  + O\Big( \frac{1}{(n_2)^2}\Big) \geq \frac{1 -9\rho}{\pi n_2}  + O\Big( \frac{1}{(n_2)^2}\Big) \,.
\]
For the second inequality, we have used that \(f\ast\tilde f \in \mathcal{M}_1(2\sqrt{\rho N})\) so that \(|z-x|\leq 3\sqrt{\rho N}\) in the sum, and thus \(\exp(- \frac{|z-x|^2}{2n_2}) \geq 1-  \frac{9\rho N}{n_2} \geq 1-9\rho\).
All together, for \(|z|\leq \sqrt{\rho N}\), we get the bound
\begin{equation}
  \label{eq:z<rho}
  \mathbf{S}_{n_2}(z) \leq \frac{2}{\pi} \Big(\frac{1}{2n_2-\ell_1} - \frac{1}{2n_2} \Big) + \frac{9\rho}{\pi n_2} + \frac{C}{(n_2)^2} \leq C\, \frac{n_1}{(n_2)^2} + C\, \frac{\rho}{n_2} \,,
\end{equation}
where we have used that \(\frac{1}{2n_2-\ell_1} - \frac{1}{2n_2} = \frac{\ell_1}{ 2n_2 (2n_2-\ell_1) } \leq \frac{n_1}{(n_2)^2}\) since \(\ell_1 \leq 2n_1\) and \(2n_2-\ell_1 \geq n_2\).
All together, we get that 
\begin{equation}
  \label{sum:z<rho-N}
  \beta_N^2 \sum_{|z|\leq \sqrt{\rho N}}  h_N(z) \cV_{M_N,\leq k_N}^{\beta_N}(z) \mathbf{S}_{n_2}(z) \leq C \Big(\frac{n_1}{(n_2)^2} + \frac{\rho}{n_1} \Big) \,\beta_N^2 \sum_{ |z|\leq T_N}  h_N(z) \cV_{M_N,\leq k_N}^{\beta_N}(z) \,,
\end{equation}
where we have also enlarged the sum as a last step, to get an upper bound.

\smallskip
\noindent
\textit{(ii) If \(\sqrt{\rho N}<|z| \leq T_N\). }
We can then bound \(q^{(f)}_{2n_2-\ell_1}(z-v_1) \leq \frac{C}{2n_2-\ell_1}\) by the local limit theorem, so that 
\begin{equation}
  \label{eq:z>rho}
  \mathbf{S}_{n_2}(z)  \leq \frac{C}{n_2} \,,
\end{equation}
since \(n_2\geq \ell_1\).
Recalling also the definition~\eqref{eq:epsilon-N} of \(\gep_N\), we thus get that 
\begin{equation}
  \label{sum:z<TN}
  \beta_N^2 \sum_{\sqrt{\rho N} < |z|\leq T_N}  h_N(z) \cV_{M_N,\leq k_N}^{\beta_N}(z) \mathbf{S}_{n_2}(z) \leq \frac{C}{n_2}\,\gep_N \, \beta_N^2 \sum_{ |z|\leq T_N}  h_N(z) \cV_{M_N,\leq k_N}^{\beta_N}(z) \,.
\end{equation}

\smallskip
\noindent
\textit{(iii) If \(|z| > T_N\). }
Note that this case is empty in the summable case~\eqref{finite-sum} where we have set \(T_N=+\infty\).
In the non-summable case~\eqref{def:h}, we bound \(\mathcal{V}_{M_N,\leq k_N}(z) \leq C^{k_N}\) by Claim~\ref{claim:sup-V} and also \(h(z) \leq C (\log N)^{a-1}/N\) by~\eqref{def:h}, recalling that \(T_N=\sqrt{N \log N}\).
Bounding also \(\mathbf{S}_{n_2}(z) \leq q^{(f)}_{2n_2-\ell_1}(z-v_1)\), we get that 
\begin{equation*}
  \beta_N^2\sum_{|z|> T_N}  h_N(z) \cV_{M_N,\leq k_N}^{\beta_N}(z) \mathbf{S}_{n_2}(z) \leq C \beta_N^2 \, \frac{(\log N)^{a-1}}{N}\, C^{k_N} \, \bP_f\big( |S_{2n_2-\ell_1} -v_1| >T_N \big)\,.
\end{equation*}
Now we consider two cases: either \(|v_1| \leq \frac12 T_N\), in which case 
\[
  \bP_f\big( |S_{2n_2-\ell_1} -v_1| >T_N \big) \leq \bP_f\big( |S_{2n_2-\ell_1}| > \tfrac12 T_N \big) \leq \e^{-c \log N} = N^{-c} \,,
\]
using also that \(|S_0|\leq \sqrt{\rho N}\) since \(f\in \mathcal{M}_1(\sqrt{\rho N})\) and that \(2n_2-\ell_1 \leq 2N\).
On the other hand, if \(|v_1| > \frac12 T_N\) we simply bound the probability by \(1\).
All together, using that \(C^{k_N} = N^{o(1)}\) because \(k_N \ll \log N\) and that \(\beta_N^2 (\log N)^{a-1} \to 0\) by \eqref{def:beta-N}, we end up with
\begin{equation}
  \label{sum:z>TN}
  \beta_N^2  \sum_{|z|> T_N}  h_N(z) \cV_{M_N,\leq k_N}^{\beta_N}(z) \mathbf{S}_{n_2}(z)  \leq \frac{C}{n_2} N^{-c/2} + \frac{C^{k_N}}{N} \ind_{\{|v_1| >\frac12 T_N\}}.
\end{equation}
having also used that \(\frac{1}{N}\leq \frac{1}{n_2}\).

\smallskip
\noindent
\textit{Conclusion of the proof. }
Notice that the bounds~\eqref{sum:z<rho-N}, \eqref{sum:z<TN} do not depend on \(A_1,B_1\) (or \((\ell_1,v_1)\)) in \eqref{eq:covariance-1} but that \eqref{sum:z>TN} does (at least for its second term).
Notice also that, similarly as in~\eqref{eq:tiltedE-1}
\begin{equation*}
  \begin{split}
    \sum_{x_1,y_1\in \Z^2}\sum_{A_1\in \mathcal{I}_{n_1,x_1}, B_1 \in \mathcal{I}_{n_1,y_1}} 
    q^{(f)}(A_1) q^{(f)}(B_1) & \EE[\xi(A_1) \xi(B_1)] = \sum_{x_1, y_1\in \Z^2}  \EE[X_{n_1,x_1} X_{n_1,y_1}] \\
    & \leq \frac{C}{n_1} \beta_N^2\, \sum_{|z| \leq T_N} h(z) \, \mathcal{V}_{M_N,\leq k_N}^{\beta_N} (z) \,,
  \end{split}
\end{equation*}
applying~\eqref{eq:upper-sumXnxy} for the last inequality.

All together, using the bounds~\eqref{sum:z<rho-N}, \eqref{sum:z<TN}, \eqref{sum:z>TN} back in~\eqref{eq:covariance-1}, this gives that there is a constant \(C>0\) such that, for any \(\rho N\leq n_1 < n_2 \leq \frac12 N\) with \(n_2-n_1 >2 M_N\), 
\[
\sum_{x_1,x_2 \in \Z^2}   \tCC_f[X_{n_1,x_1}(f),X_{n_2,x_2}(f)]  \leq  \mathbf{I}_{n_1,n_2} + \mathbf{II}_{n_1,n_2} \,,
\]
with 
\begin{equation*}
   \mathbf{I}_{n_1,n_2} \coloneqq C \bigg(\beta_N^2 \sum_{|z|\leq T_N} h(z) \mathcal{V}_{M_N,\leq k_N}^{\beta_N}(z)\bigg)^2  \Big[ \frac{1}{n_1 n_2} \Big(\rho + \gep_N + N^{-c/2} \Big) +  \frac{1}{(n_2)^2}  \Big] \,,
\end{equation*}
and, recalling that we denoted by \((\ell_1,v_1)\) the last point of~\(B_1\),
\[
\mathbf{II}_{n_1,n_2} \coloneqq \frac{C^{k_N}}{N} \sum_{x_1,y_1\in \Z^2} \sum_{A_1\in \mathcal{I}_{n_1,x_1}, B_1 \in \mathcal{I}_{n_1,y_1}}  q^{(f)}(A_1) q^{(f)}(B_1)  \EE[\xi(A_1) \xi(B_1)] \ind_{\{|v_1| >\frac12 T_N\}} \,.
\]

Then, one can easily check that
\[
\sum_{ \rho N \leq n_1 < n_2 \leq \frac12 N} \frac{1}{n_1 n_2} \leq c \log \Big(\frac{2}{\rho}\Big)^2 
\quad \text{ and } \quad
\sum_{ \rho N \leq n_1 < n_2 \leq \frac12 N} \frac{1}{(n_2)^2} \leq c \log \Big(\frac{2}{\rho}\Big) \,,
\]
so that summing \(\mathbf{I}_{n_1,n_2}\) over \(\rho N \leq n_1<n_2 \leq N\), we end up with 
\[
\sum_{ \rho N \leq n_1 < n_2 \leq \frac12 N} \mathbf{I}_{n_1,n_2} \leq C' \bigg( \log \Big(\frac{2}{\rho}\Big) \beta_N^2 \sum_{|z|\leq T_N} h(z) \mathcal{V}_{M_N,\leq k_N}^{\beta_N}(z)\bigg)^2  \Big[\rho + \gep_N +N^{-c/2} + \log\Big(\frac{2}{\rho}\Big)^{-1} \Big]\,.
\]

Turning to \(\mathbf{II}_{n_1,n_2}\), we show below that there are constants \(C,c>0\) such that
\begin{equation}
  \label{bound-sum>TN}
  \sum_{n_1=\rho N}^{\frac12 N}\sum_{x_1,y_1\in \Z^2} \sum_{A_1\in \mathcal{I}_{n_1,x_1}, B_1 \in \mathcal{I}_{n_1,y_1}}  q^{(f)}(A_1) q^{(f)}(B_1)  \EE[\xi(A_1) \xi(B_1)] \ind_{\{|v_1| >\frac12 T_N\}} \leq C\, N^{-c} \,.
\end{equation}
The proof is postponed to Section~\ref{sec:estim-technical} below (it is inspired by the proof of Claim~\ref{claim:sup-V}).
We thus get that
\[
\sum_{ \rho N \leq n_1 < n_2 \leq \frac12 N} \mathbf{II}_{n_1,n_2} \leq C\, C^{k_N}  N^{-c} \leq C' \, N^{-c/2} \bigg( \log \Big(\frac{2}{\rho}\Big) \beta_N^2 \sum_{|z|\leq T_N} h(z) \mathcal{V}_{M_N,\leq k_N}^{\beta_N}(z)\bigg)^2 \,,
\]
where for the second inequality we have used that \(C^{k_N} \beta_N^{-4} N^{-c/2}\to 0\) and also that we have \(\log (\frac2\rho) \sum_{|z|\leq T_N} h(z) \mathcal{V}_{M_N,\leq k_N}^{\beta_N}(z)\geq 1\).

Finally, using that \(\rho, N^{-c/2} \leq C \log(\frac2\rho)\) (recall \(\rho \downarrow0\) with \(\rho=N^{o(1)}\)) and since the constants are uniform in \(f\in \mathcal{M}_1(\sqrt{\rho N})\), this concludes the proof of Lemma~\ref{lem:off-diagonal-Var}.
\qed

\section{Second moment estimates: proof of Claims~\ref{claim:sup-V},~\ref{claim:ratio} and~\ref{claim:lower-bound}}
\label{sec:estim-V}

In this section, we prove the claims on the truncated covariances from Section~\ref{sec:conclusion-proof}, and~\eqref{bound-sum>TN}.

\subsection{Proof of Claim~\ref{claim:sup-V}}

First of all, notice that \(\mathcal{V}_{m,\leq k}^{\beta_N} (x) \geq 1\), simply by restricting the sum~\eqref{def:V} to \(A=B=\emptyset\).

For the upper bound, the idea is to reduce to a sub-critical estimate. 
First, we show that for any \(\tilde \beta_N \leq \beta_N\), we have that
\begin{equation}
  \label{eq:compare-beta-tilde-beta}
  \mathcal{V}_{m,\leq k}^{\beta_N} (x) \leq \Big(\frac{\beta_N^2 \e^{\beta_N^2}}{\tilde \beta_N^2} \Big)^{k} \mathcal{V}_{m,\leq k}^{\tilde \beta_N} (x) \,.
\end{equation}
This is simply due to~\eqref{eq:cov-xi}: using that \(\e^{\beta_N^2 h(z)} -1 \leq \beta_N^2 \e^{\beta_N^2} h(z)\) and \(\e^{\tilde \beta_N^2 h(z)} -1 \geq \tilde \beta_N h(z)\), we get that if \(|A|=|B|=\ell \),
\begin{equation}
  \label{eq:compare-covariances}
  \begin{split}
  & \EE[ \xi^{(\beta_N)}(A) \xi^{(\beta_N)}(B)] = \prod_{i=1}^\ell \Big(\e^{\beta_N^2 h(x_i-y_i)}-1\Big) \ind_{\{n_i=m_i\}} \\
  & \qquad \leq \Big(\frac{\beta_N^2 \e^{\beta_N^2}}{\tilde \beta_N^2}\Big)^\ell \prod_{i=1}^k \Big(\e^{\tilde \beta_N^2 h(x_i-y_i)}-1\Big) \ind_{\{n_i=m_i\}} = \Big(\frac{\beta_N^2 \e^{\beta_N^2}}{\tilde \beta_N^2}\Big)^\ell \EE[ \xi^{(\tilde \beta_N)}(A) \xi^{(\tilde \beta_N)}(B)]\,. 
  \end{split}
\end{equation}
Note that we have kept here the dependence on \(\beta_N,\tilde\beta_N\) of the variables \(\xi\) in order to make the change from \(\beta_N\) to \(\tilde\beta_N\) more explicit (recall that \(\xi_{n,x}^{(\beta)} \coloneqq \e^{\beta \omega_{n,x} -\frac12 \beta^2}-1\)).
Plugged into the definition~\eqref{def:V} of \(\mathcal{V}_{m,\leq k}^{\beta_N}\), since the sum is restricted to \(|A|=|B|\leq k\), this gives~\eqref{eq:compare-beta-tilde-beta}.

Now, we choose \(\tilde\beta_N\) as in Assumption~\ref{hyp:scaling} but with \(\hat \beta = \frac12 \hat\beta_c\): using the monotonicity of \(\mathcal{V}_{m,\leq k}^{\tilde\beta_N}\) in \(m,k\) (which is clear from~\eqref{def:V}), we get that 
\[
\sup_{x\in \Z^2}  \mathcal{V}_{m,\leq k}^{\tilde \beta_N} (x) \leq \sup_{x\in \Z^2} \EE\big[W_N^{\tilde\beta_N}(x)W_N^{\tilde\beta_N}(0)\big] <+\infty \,,
\]
using~\eqref{eq:sup-covariances} for the last part, since \(\tilde\beta_N\) is in the sub-critical regime.
Now, since we have that \(\lim_{N\to\infty}\beta_N^2 \e^{\beta_N^2}/\tilde \beta_N^2 = 2 \hat\beta /\hat\beta_c\), we obtain from \eqref{eq:compare-beta-tilde-beta} that 
\[
\sup_{x\in \Z^2} \mathcal{V}_{m,\leq k}^{\beta_N} (x) \leq C \big( C' \hat\beta/\hat\beta_c\big)^{k} \,,
\]
which concludes the proof of Claim~\ref{claim:ratio}.
\qed

\subsection{Proof of \texorpdfstring{\eqref{bound-sum>TN}}{}}
\label{sec:estim-technical}

First of all, notice that the left-hand side of~\eqref{bound-sum>TN} is bounded by 
\begin{equation}
  \label{eq:reduction-1}
  \sum_{n_1=\rho N}^{\frac12N} \sum_{x_1,y_1\in \Z^2}\sum_{A_1\in \mathcal{I}_{n_1,x_1}} \sum_{B_1 \in \mathcal{I}_{n_1,y_1} \setminus \hat{\mathcal{I}}_{n_1,y_1}} q^{(f)}(A_1) q^{(f)}(B_1)  \EE[\xi(A_1) \xi(B_1)] \,,
\end{equation}
where, in analogy with the definition \(\mathcal{I}_{n,x}\) we have set 
\[
\hat{\mathcal{I}}_{n,x} = \Big\{ A \subset \llb n,n+M_N \rrb \times \llb -\tfrac12 T_N, \tfrac12 T_N \rrb^2 \,,\mathrm{start}(A) = (n,x) \,, 1\leq |A| \leq 1+k_N  \}\,,
\]
\textit{i.e.}\ \(\hat{\mathcal{I}}_{n,x}\) contains subsets constrained to stay within spatial distance \(\frac12 T_N\) of the origin.

Now, we use the same idea as in Claim~\ref{claim:sup-V} to reduce to subcritical estimates.
We use~\eqref{eq:compare-covariances} with \(\tilde \beta_N\) scaled as in Assumption~\ref{hyp:scaling} but with \(\hat{\beta} = \frac12 \hat\beta_c\), so that~\eqref{eq:reduction-1} is bounded by
\[
\begin{split}
  \Big(\frac{\beta_N^2 \e^{\beta_N^2}}{\tilde \beta_N^2}\Big)^{k_N}  \sum_{n_1=\rho N}^{\frac12N} \sum_{x_1,y_1\in \Z^2} \sum_{A_1\in \mathcal{I}_{n_1,x_1}} \sum_{B_1 \in \mathcal{I}_{n_1,y_1} \setminus \hat{\mathcal{I}}_{n_1,y_1}} q^{(f)}(A_1) q^{(f)}(B_1)  \EE[\xi^{(\tilde \beta_N)}(A_1) \xi^{(\tilde \beta_N)}(B_1)] & \\
  \leq \big( C' \hat\beta/\hat\beta_c \big)^{k_N} \sum_{A_1\subset \llb 1,N \rrb \times \Z^2}  \sumtwo{B_1 \subset \llb 1,N \rrb \times \Z^2}{\exists (m,y)\in B_1, |y|>\frac12 T_N} q^{(f)}(A_1) q^{(f)}(B_1)  \EE[\xi^{(\tilde \beta_N)}(A_1) \xi^{(\tilde \beta_N)}(B_1)] &
\end{split}
\]
where for the second inequality we have simply removed the restrictions on the sizes of \(A_1,B_1\), keeping only the restriction on \(B_1\).
One can now observe that this last sum over \(A_1,B_1\) corresponds to the chaos decomposition of
\[
\EE\big[ W_N^{\tilde\beta_N}(f)^2 \big] - \EE \big[W_N^{\tilde\beta_N}(f) \widehat{W}_N^{\tilde \beta_N}(f)\big] \,,
\]
where 
\[
\widehat{W}_N^{\tilde \beta_N}(f) \coloneqq \bE_f\Big[ \e^{\sum_{n=1}^N (\tilde \beta_N\omega(n,S_n) - \frac12 \tilde \beta_N^2)} \ind_{\{\max_{1\leq n \leq N}|S_n| \leq \frac12 T_N\}}\Big] \,.
\]
Then, similarly to~\eqref{eq:covariances}, we get that  
\[
\EE\big[ W_N^{\tilde\beta_N}(f)^2 \big] - \EE \big[W_N^{\tilde\beta_N}(f) \widehat{W}_N^{\tilde \beta_N}(f)\big] = \bE^{\otimes 2}_f \Big[ \e^{\tilde\beta_N^2 \mathcal{L}_N(S,S')} \ind_{\{\max_{1\leq n \leq N}|S_n| > \frac12 T_N \}} \Big] 
\]
with \(\mathcal{L}_N(S,S') = \sum_{n=1}^N h(S_n-S_n')\). 
Applying the Cauchy--Schwarz inequality, this is bounded by 
\[
\bE^{\otimes 2}_f \Big[ \e^{2 \tilde\beta_N^2 \mathcal{L}_N(S,S')} \Big]^{1/2} \bP_f\Big(\max_{1\leq n \leq N}|S_n| > \frac12 T_N  \Big) \leq  \EE\big[ W_N^{\sqrt{2} \tilde\beta_N}(f)^2 \big]^{1/2} \e^{-c \log N} \,,
\]
where we have used a standard large deviation bound to estimate \(\bP_f(\max_{1\leq n \leq N}|S_n| > \frac12 T_N )\), recalling that \(T_N =\sqrt{N \log N}\) and that \(f \in \mathcal{M}_1(\sqrt{\rho N})\).
Now, since \(\sqrt{2} \tilde \beta_N\) satisfies the scaling of Assumption~\ref{hyp:scaling} with \(\hat\beta = \frac{1}{\sqrt{2}} \hat{\beta}_c < \hat\beta_c \), we get that the first term \(\EE[ W_N^{\sqrt{2} \tilde\beta_N}(f)^2]\) is bounded by a constant, thanks to~\eqref{eq:sup-covariances}. 

All together, we have proven that \(\eqref{eq:reduction-1} \leq C (C'\hat\beta/\hat\beta_c)^{k_N} N^{-c}\).
Since \(C^{k_N} = N^{o(1)}\) (recall that \(k_N\ll \log N\)), this concludes the proof of~\eqref{bound-sum>TN}, renaming the constants if necessary.
\qed

\subsection{Proof of Claim~\ref{claim:ratio}}

The proof will rely on the following identity, which simply derives from decomposing~\eqref{def:V} according to the first point in \(A,B\) (provided that they are non-empty) and recalling \eqref{eq:cov-xi}:
\[
\mathcal{V}_{m,\leq k}^{\beta_N}(z) = 1 + \sum_{n=1}^m \sum_{x,y\in \Z^2} q_n(x-z)q_n(y) \beta_N^2 h_N(x-y) \mathcal{V}_{m-n,\leq k-1}^{\beta_N} (x-y) \,.
\]
Using the monotonicity of \(\mathcal{V}_{m,\leq k}^{\tilde\beta_N}\) in \(m,k\), we get that 
\begin{equation}
  \label{eq:decompose-V}
  \mathcal{V}_{m,\leq k}^{\beta_N}(z) \leq  1 + \sum_{n=1}^m \sum_{x,y\in \Z^2} q_n(x-z)q_n(y) \beta_N^2 h_N(x-y) \mathcal{V}_{m,\leq k}^{\beta_N}(x-y) \,.
\end{equation}

\subsubsection*{(i) The summable case}

First note that in the summable case we have \(\sum_{z\in \Z^2} h(z) \mathcal{V}_{m,\leq k}^{\beta_N}(z) <+\infty\) for any \(N\in \N\); recall also that we have set \(T_N=+\infty\) here.

Thus, using~\eqref{eq:decompose-V} and simply bounding \(q_n(x-z) \leq \frac{c}{n}\) by the local CLT, we get that 
\begin{align*}
  \sum_{|z|> t_N} h(z)\mathcal{V}_{M_N,\leq k_N}^{\beta_N}(z) \leq \sum_{|z|> t_N} h(z) + \sum_{|z|> t_N} h(z) \beta_N^2 \sum_{n=1}^{M_N} \frac{c}{n} \sum_{x,y\in \Z^2} q_n(y) h(x-y)\mathcal{V}_{M_N,\leq k_N}^{\beta_N}(x-y) \,.
\end{align*}
Summing first over \(x\in \Z^2\) then over \(y\) and finally over \(n \leq M_N\) with \(M_N\leq N\), we get that 
\[
\sum_{|z|> t_N} h(z)\mathcal{V}_{M_N,\leq k_N}^{\beta_N}(z) \leq \bigg(\sum_{|z|> t_N} h(z)\bigg) \bigg(1+ c \beta_N^2 \log N \times \sum_{z'\in \Z^2} h(z')\mathcal{V}_{M_N,\leq k_N}^{\beta_N}(z') \bigg) \,.
\]
Now, from the scaling of Assumption~\ref{hyp:scaling}, we have that \(\beta_N^2 \log N \leq C\). 
Using also that \(\sum_{z\in \Z^2} h(z)\mathcal{V}_{M_N,\leq k_N}^{\beta_N}(z) \geq 1\), we get that 
\[
\frac{\sum_{|z|> t_N} h(z)\mathcal{V}_{M_N,\leq k_N}^{\beta_N}(z) }{\sum_{z\in \Z^2} h(z)\mathcal{V}_{M_N,\leq k_N}^{\beta_N}(z) } \leq C' \sum_{|z|> t_N} h(z)  \xrightarrow[\;N\to\infty\;]{} 0\,,
\]
since \(h\) is summable.
This concludes the proof of Claim~\ref{claim:ratio} in the summable~\eqref{finite-sum} case; in fact the claim holds for any sequence \(t_N\) with \(t_N\to\infty\).
\qed

\subsubsection*{(ii) The non-summable case}

In the case~\eqref{def:h}, we cannot use the same idea as above since \(\sum_{z\in \Z^2} h(z) \mathcal{V}_{M_N,\leq k_N}^{\beta_N}(z) =+\infty\) for any \(N\in\N\) and we need to work with truncated sums.
Recall that we have set \(T_N = \sqrt{N \log N}\).

First, let us consider the sum in \eqref{eq:decompose-V} when \(|x-y| > T_N\).
In that case, we may bound \(h_N(x-y) \leq C (\log N)^{a-1} /N\) by~\eqref{def:h} and we also use Claim~\ref{claim:sup-V} to get that \(\mathcal{V}_{M_N,\leq k_N}^{\beta_N}(z) \leq C^{k_N}\).
Thus, we obtain that 
\[
\sum_{n=1}^{M_N} \sum_{x,y\in \Z^2, |x-y|> T_N} q_n(x-z)q_n(y) \beta_N^2 h_N(x-y) \mathcal{V}_{M_N,\leq k_N}^{\beta_N}(x-y) \leq C \sum_{n=1}^{M_N} C^{k_N} \frac{\beta_N^2 (\log N)^{a-1}}{N} 
\]
where we also have used that the sum over \(x,y\) of \(q_n(x-z)q_n(y)\) is equal to~\(1\).
Now, recalling that \(C^{k_N} = N^{o(1)}\) because \(k_N \ll \log N\), the above is bounded by \(N^{o(1)} M_N/N\). Since \(M_N=N^{\gamma}\) with \(\gamma<1\), recall~\eqref{def:MN}, this term goes to \(0\).
Therefore, for \(N\) large enough, we get that 
\begin{align*}
  \mathcal{V}_{M_N,\leq k_N}^{\beta_N}(z) & \leq  2 + \sum_{n=1}^{M_N} \sum_{x,y\in \Z^2, |x-y|\leq T_N} q_n(x-z)q_n(y) \beta_N^2 h_N(x-y) \mathcal{V}_{M_N,\leq k_N}^{\beta_N}(x-y)  \\
  & \leq 2 + \beta_N^2 \sum_{n=1}^{M_N} \frac{c}{n}  \times \sum_{|z'|\leq T_N} h_N(z') \mathcal{V}_{M_N,\leq k_N}^{\beta_N}(z') \,,
\end{align*}
where we have used as above the local CLT to bound \(q_n(x-z)\leq \frac{c}{n}\) then summed over \(x\) and finally over \(y\in \Z^2\).
All together, using also that \(\mathcal{V}_{M_N,\leq k_N}^{\beta_N}(z) \geq 1\), we get that 
\[
\frac{\sum_{t_N\leq |z|\leq T_N} h_N(z) \mathcal{V}_{M_N,\leq k_N}^{\beta_N}(z)}{\sum_{|z|\leq T_N} h_N(z) \mathcal{V}_{M_N,\leq k_N}^{\beta_N}(z)} \leq \sum_{t_N\leq |z|\leq T_N} h_N(z) \,\bigg( \frac{2}{\sum_{|z|\leq T_N} h_N(z)} + C \beta_N^2 \log N \bigg) \,.
\]

Now, we have that \(\sum_{|z|\leq T} h(z) \sim c_a (\log T)^{a+1}\), see~\eqref{sum:h}, showing in particular that 
\[
\sum_{t_N\leq |z|\leq T_N} h_N(z) = o\bigg( \sum_{|z|\leq T_N} h_N(z)\bigg) = o\big( (\log N)^{a+1} \big) \,,
\]
provided that \(t_N = (T_N)^{1+o(1)}\); the second identity follows simply recalling that \(T_N = \sqrt{N \log N}\).
Since \(\beta_N^2 \log N = O((\log N)^{-(a+1)})\) by the scaling of Assumption~\ref{hyp:scaling}, we conclude that 
\[
\lim_{N\to\infty} \frac{\sum_{t_N\leq |z|\leq T_N} h_N(z) \mathcal{V}_{M_N,\leq k_N}^{\beta_N}(z)}{\sum_{|z|\leq T_N} h_N(z) \mathcal{V}_{M_N,\leq k_N}^{\beta_N}(z)} = 0\,,
\]
which finishes the proof of Claim~\ref{claim:ratio}.
\qed

\subsection{Proof of Claim \ref{claim:lower-bound}}

Our goal is to prove that, letting \((\beta_N)_{N\geq 1}\) be scaled as in Assumption~\ref{hyp:scaling} with \(\hat \beta>\hat\beta_c\), then for any \(\gep\in (0,1)\) there is a constant \(c_{\gep}>0\) such that
\begin{equation}
	\label{eq:mainclaim}
	\mathcal{V}_{N,\leq k_N}^{\beta_N}(0) \ge c_{\gep} \bigg((1-\varepsilon) \Big( \frac{\hat \beta}{\hat{\beta}_c}\Big)^2 \,\bigg)^{k_N}
\end{equation}
for any sequence \((k_N)_{N\geq 0}\) with \(k_N \to \infty\) with \(\log k_N \ll \log  N\) and for \(N\) sufficiently large.

Applying \eqref{eq:mainclaim} with \(N=M_N\) and \(\hat\beta=\hat\beta'\) and \(\gep \in (0,1) \) fixed such that \((1-\gep) (\hat\beta'/\hat\beta_c)^2>1\), we obtain that
\begin{equation*}
		(\log N)^{-b}\mathcal{V}_{M_N,\leq k_N}^{\beta_N}(0) \ge c_{\gep} (\log N)^{-b}\, \bigg((1-\varepsilon) \Big( \frac{\hat \beta}{\hat{\beta}_c}\Big)^2 \,\bigg)^{k_N} = c_{\gep}\, \e^{ c_{\hat\beta} k_N - b \log \log N}\,.
\end{equation*}
with \(c_{\hat\beta} \coloneqq \log  ((1-\gep)(\hat\beta'/\hat\beta_c)^2)  >0\).
Recalling that \(k_N= (\log\log N)^2 \gg \log \log N\), this concludes the proof of Claim~\ref{claim:lower-bound}.

Since the summable case is a by-product of the proof of Lemma~\ref{lem:sec-mom-summable} (see Remark~\ref{rem:lower-bound}, the i.i.d.\ case being even simpler), we focus on proving~\eqref{eq:mainclaim} only in the non-summable case.
We rely on the proof of \cite[Lemma~2.7]{CCD25}.
It gives that for any \(M>1\), if \(N\) is large enough \(\mathcal{V}^{\beta_N}_{N,\leq k_N}(0)\) is bounded from below by
\begin{equation*}
	1+\sum_{k=1}^{k_N} 
	\Big(\frac{(1-\varepsilon)\beta_N^2(\log N)^{a+2}}{2^{a+1} M}\Big)^k 
	\sum_{0\leq l_1,\ldots,l_k\leq \lfloor M\rfloor-1}
	\prod_{i=1}^{k} \left(1-\frac{\log k_N}{\log N}-\frac{l_i\vee l_{i-1}}{M}-\frac4M\right)
		\left(\frac{l_i}{M}\right)^a\,,
\end{equation*}
with \(l_0=0\).
Letting \(\delta_N \coloneqq \frac{\log k_N}{\log N}\) and factorizing the product by \(1-\delta_N\), setting \(M' = (1-\delta_N)M\) we can rewrite this as 
\begin{equation}
  \label{eq:V(0)}
  \mathcal{V}^{\beta_N}_{N,\leq k_N}(0)
	\ge
	1+\sum_{k=1}^{k_N} 
	\Big(\frac{\beta_N^2(1-\varepsilon) (1-\delta_N)^{a+2} (\log N)^{a+2}}{2^{a+1}}\Big)^k I_k^{(M')}\,,
\end{equation}
where we have set (writing \(M\) instead of \(M'\) for simplicity)
\begin{equation*}
	I_k^{(M)}\coloneqq
	\frac1{M^k}
	\sum_{0\leq l_1,\ldots,l_k\leq \lfloor M\rfloor-1}
	\prod_{i=1}^{k} \left(1-\frac{l_i\vee l_{i-1}}{M}-\frac{4}{M}\right)
		\left(\frac{l_i}{M}\right)^a \quad \text{ with }l_0=0 \,.
\end{equation*}

We stress that for any fixed \(k\), by a Riemann sum approximation we have
\begin{equation}
  \label{eq:Riemann-sum-0}
  \lim_{M\to\infty} I_k^{(M)} 
  = \alpha_k \coloneqq \int_{s_1,\ldots,s_{k}\in[0,1]} \prod_{i=1}^k \bigl(1-s_i\vee s_{i-1}\bigr) s_i^{\,a} \dd s_1 \cdots \dd s_k\quad \text{ with }s_0=0 \,.
\end{equation}
Now, the idea is that \cite[Thm.~6.1]{CCD25} shows that the series \(\sum_{k\ge0} z^{2k}\alpha_k\) is equal to \(\tilde J_{\alpha}( \frac{2}{a+2}z)^{-1}\); in particular the radius of convergence for \(\alpha_k\) is equal to \((\frac{a+2}{2})^2 z_a^2\).
The difficultly here is that we need to estimate \(I_k^{(M)}\) with a growing \(k\to\infty\) (ideally for \(k=k_N\)).

To circumvent the difficulty of controlling the Riemann sums approximations, we use some ``super-multiplicativity'' property, by introducing a truncated version of \(I_k^{(M)}\). 
Throughout the proof, we adopt the convention that, for every \(\zeta\in[0,1]\), the notation \(\zeta M\) stands for \(\lfloor\zeta M\rfloor\). 
For \(l_0\le \zeta M\), define,
\begin{align*}
	\widetilde I_k^{(M)}(l_0;M-1,\zeta M)\coloneqq
	\frac{1}{M^k}
	\sum_{\substack{0\le l_1,\ldots,l_{k-1}\le M-1\\0\le l_k\le \zeta M	}}
	\prod_{i=1}^k\left(1-\frac{l_i\vee l_{i-1}}{M}-\frac{4}{M}\right)
	\left(\frac{l_i}{M}\right)^a 
\end{align*}
and
\begin{align*}
\widetilde I_k^{(M)}(M-1,\zeta M)\coloneqq \inf_{l_0 \le \zeta M} \widetilde I_k^{(M)}(l_0;M-1,\zeta M)\,.
\end{align*}
This way, restricting the sum to \(l_{jk} \leq \zeta M\) we clearly have that \(\widetilde I_{(j+1)k}^{(M)}(l_0;M-1,\zeta M) \geq \tilde I_{jk}^{(M)}(l_0;M-1,\zeta M)  \widetilde I_k^{(M)}(M-1,\zeta M)\).
Thus, by induction we get that 
\begin{equation}\label{eq:Ik-concatenation}
	I_{jk}^{(M)} \geq \widetilde I_{jk}^{(M)}(0;M-1,\zeta M)
	\ge \left(\widetilde I_k^{(M)}(M-1,\zeta M)\right)^j,
	\qquad j,k\in\mathbb \N\,.
\end{equation}

Moreover, by a Riemann sum approximation, letting
\begin{align*}
	\alpha_k(s_0;\xi,\zeta) \coloneqq
	\int_{s_1,\ldots,s_{k-1}\in[0,\xi], s_k\in[0,\zeta]}
	\prod_{i=1}^k
	\bigl(1-s_i\vee s_{i-1}\bigr)s_i^{\,a}
 \dd s_1 \cdots \dd s_k\,,
\end{align*}
we obtain
\begin{equation}
	\label{eq:riemannsumapprox}
	\lim_{M\to \infty} \widetilde I_k^{(M)}(M-1,\zeta M) = \inf_{s_0 \le \zeta} \alpha_k(s_0;1,\zeta)  = \alpha_k (\zeta;1,\zeta) \,,
\end{equation}
noting that \(s_0\mapsto \alpha_k(s_0;1,\zeta)\) is non-increasing for the last inequality.

Finally, we prove below the following estimate, which tells that \(\alpha_k(\zeta;1,\zeta)\) has the same radius of convergence as \(\alpha_k\) defined in~\eqref{eq:Riemann-sum-0}:
\begin{equation}
	\label{eq:limsupalpha}
	\limsup_{k\to\infty}\alpha_k(\zeta;1,\zeta)^{1/k}
	=\left(\frac{2}{a+2}\right)^2\frac1{z_a^2}\,.
\end{equation}

Let us postpone the proof of~\eqref{eq:limsupalpha} until the end of the argument and conclude the proof from here.
By \eqref{eq:limsupalpha}, for any $\varepsilon \in (0,1)$ there exists $k_0=k_0(\varepsilon)\in\N$ such that 
\begin{equation}
	\label{eq:k0}
	\alpha_{k_0}(\zeta;1,\zeta) \ge \bigg( (1-\varepsilon)\Big(\frac{2}{a+2}\Big)^2\frac1{z_a^2} \bigg)^{k_0}\,;
\end{equation}
Moreover, this \(k_0\) being fixed, by \eqref{eq:riemannsumapprox} there exists
$M_0=M_0(\varepsilon)$ such that, for every $M\ge M_0$,
\begin{equation}
	\label{eq:M0}
	\widetilde I_{k_0}^{(M)}(M-1,\zeta M) \ge (1-\varepsilon)^{k_0} \alpha_{k_0}(\zeta;1,\zeta) \ge \bigg( (1-\varepsilon)^2\Big(\frac{2}{a+2}\Big)^2\frac1{z_a^2} \bigg)^{k_0}\,.
\end{equation}

The proof of \eqref{eq:mainclaim} now follows. 
Recall the definition \eqref{def:beta-N} of $\beta_N$ and restrict the sum in \eqref{eq:V(0)} to simply \(k =j_N k_0\) with \(j_N = \lfloor k_N /k_0 \rfloor\). 
By applying \eqref{eq:Ik-concatenation} and \eqref{eq:M0}, for $M\ge M_0$ we obtain 
\begin{equation*}
	\mathcal{V}^{\beta_N}_{N,\leq k_N}(0) \geq 
	\Big( (1-\varepsilon) (1-\delta_N) \Big(\frac{a+2}{2} \Big)^2 \hat\beta^2 \Big)^{j_N k_0}
	I_{j_N k_0}^{(M)} \ge  \bigg( (1-\varepsilon)^4 \Big(\frac{\hat \beta}{z_a}\Big)^2 \bigg)^{j_N k_0} \,,
\end{equation*} 
where we have also chosen \(N\) large enough so that \(1-\delta_N  \geq 1-\gep\).
Since \(j_N k_0 \geq k_N - k_0\), we obtain \eqref{eq:mainclaim}, after renaming the parameter~$\varepsilon$; with \(c_{\gep} \coloneqq \big( (1-\varepsilon)^4 (\hat \beta/z_a)^2 \big)^{-k_0}\).
\qed

\begin{proof}[Proof of~\eqref{eq:limsupalpha}]
  Let us state here \cite[Thm.~6.1]{CCD25}, which gives that for any \(s_0<1\), 
  \[
  \sum_{k=0}^{\infty} z^{2k} \alpha_{k}(s_0;1,1) = \frac{\tilde J \big(\frac{2}{a+2} (s_0)^{\frac{a+2}{2}} z\big)}{\tilde J \big(\frac{2}{a+2} z\big)} \,.
  \]
  In particular, for any fixed \(s_0<1\), the radius of convergence of the series with terms \((\alpha_{k}(s_0;1,1))_{k\geq 0}\) is \((\frac{a+2}{2})^2 z_a^2\).
  Thus, by the Cauchy--Hadamard theorem it follows that 
  \begin{equation}
  	\label{eq:cauchyhad}
  	\limsup_{k\to\infty} \alpha_{k}(s_0;1,1)^{1/k}=\left(\frac{2}{a+2}\right)^2\frac1{z_a^2}\,,\qquad s_0 \in [0,1)\,.
  \end{equation}
  In particular, the upper bound in~\eqref{eq:limsupalpha} follows simply by observing that \(\alpha_k(\zeta;1,\zeta)\leq \alpha_{k}(\zeta;1,1)\).

  For the lower bound, let us fix \(\xi \in (\zeta,1)\).
  Since \(\xi \mapsto \alpha_k (s_0;\xi,\zeta)\) is non-increasing, we get that \(\alpha_k (\zeta;1,\zeta) \ge \alpha_k (\zeta;\xi,\zeta)\).
  Additionally, bounding \((1-s_k\vee s_{k-1}) \geq 1-\xi\) for $s_k \le \zeta$ and $s_{k-1}\le \xi$ in the integral \(\alpha_k (\zeta;\xi,\zeta)\), we get that 
  \begin{equation*}
    \begin{split}
      \alpha_k (\zeta;\xi,\zeta) 
      & \geq  \int_{\substack{s_1,\ldots,s_{k-1}\in[0,\xi]}}
      \prod_{i=1}^{k-1} \bigl(1-s_i\vee s_{i-1}\bigr)s_i^{\,a} \Big( (1-\xi) \int_{0}^{\zeta} (s_k)^a \dd s_k\Big)
      \dd s_1 \cdots \dd s_{k-1} \\
      & =  \frac{(1-\xi)\, \zeta^{a+1}}{a+1} \, \alpha_{k-1}(\zeta;\xi,\xi)\,.
    \end{split}
  \end{equation*}
  Moreover, by the change of variables $t_i=s_i/ \xi $, $i=1,\ldots,k-1$, the integral can be bounded  by
  \begin{equation*}
  	\begin{split}
  		\alpha_{k-1}(\zeta;\xi,\xi) &=\xi^{(k-1)(a+1)} \int_{[0,1]^{k-1}} \prod_{i=1}^{k-1} \big( 1 - \xi(t_i \vee t_{i-1}) \big) t_i^a \dd t_1 \cdots \dd t_{k-1} \quad \text{ with }t_0 = \zeta/\xi<1\\
  		& = \xi^{(k-1)(a+1)} \alpha_{k-1}(\zeta/\xi; 1,1)\,.
  	\end{split}
  \end{equation*}
  All together, we have that \(\alpha_k (\zeta;\xi,\zeta) \geq  \frac{(1-\xi)\, \zeta^{a+1}}{a+1}  \xi^{(k-1)(a+1)} \alpha_{k-1}(\zeta/\xi; 1,1) \), so that applying \eqref{eq:cauchyhad} with \(s_0=\zeta/\xi<1\), we obtain  
  \begin{equation*}
  		\limsup_{k\to\infty}\alpha_k(\zeta;1,\zeta)^{1/k}
  		\ge
  		\limsup_{k\to\infty} \big(\xi^{a+1}\big)^{(k-1)/k}  \alpha_{k-1}(\zeta/\xi;1,1)^{1/k}
  		=\xi^{a+1}
  		\left(\frac{2}{a+2}\right)^2\frac1{z_a^2}\,.
  \end{equation*}
  Since \(\xi \in (\zeta,1)\) can be taken arbitrary close to \(1\), this yields the lower bound in \eqref{eq:limsupalpha}.
\end{proof}

\begin{appendix}

\section{Second moment estimates}

\subsection{The case of summable correlations}
\label{app:sec-mom-summable}

In this section, we prove the following second moment lemma in the summable case~\eqref{finite-sum};it follows the same lines as in the i.i.d.\ setting.

\begin{lemma}
  \label{lem:sec-mom-summable}
  Suppose that Assumption~\ref{main-assumption} holds with summable correlations~\eqref{finite-sum}, and let $(\beta_N)_{N \ge 0}$ be as in~\eqref{def:beta-N-summable} from Assumption~\ref{hyp:scaling}. 
  Then,
  \begin{equation}
  	\label{eq:limit-second-moment-summable}
  	\lim_{N\to\infty} \EE\Big[ (W_N^{\beta_N})^2 \Big] =
  	\begin{cases}
  		\frac{1}{1-\hat{\beta}^2} & \quad \text{ if } \hat \beta<1 \,,\\
  		+\infty & \quad \text{ if } \hat \beta \geq 1\,.\\
  	\end{cases}
  \end{equation}
  Additionally, if \(\hat \beta<1\), then
  \begin{equation}
    \label{eq:covariances-summable}
    \sup_{N\geq 0}\sup_{x\in \Z^2} \EE\big[ W_N^{\beta_N}(x) W_N^{\beta_N}(0)\big] <+\infty \,.
  \end{equation}
\end{lemma}

\begin{proof}
  First of all, using the chaos expansion~\eqref{def:chaos-expansion}, decomposing over the points of \(A,B\) and using~\eqref{eq:cov-xi}, we get that \(\EE[W_N^{\beta_N}(x)W_N^{\beta_N}(y)]\) is equal to
  \[
   \sum_{k=0}^{\infty} (\beta_N^2)^k \sum_{1\le n_1<\cdots<n_k\le N}
		\sumtwo{x_1,\ldots,x_k\in\Z^2}{y_1, \ldots,y_k \in \Z^2} \prod_{i=1}^k q_{n_i-n_{i-1}}(x_i-x_{i-1})q_{n_i-n_{i-1}}(y_i-y_{i-1}) h_N(x_i-y_i) \,,
  \]
  with by convention \(n_0=0\), \(x_0=x\), \(y_0=y\).
  With a change of variable \(z_i= x_i-y_i \in \Zeven^2\) and using that \(\sum_{x_i\in \Z^2} q_{n_i-n_{i-1}}(x_i-x_{i-1})q_{n_i-n_{i-1}}(y_i-y_{i-1}) = q_{2(n_i-n_{i-1})}(z_i-z_{i-1})\), we get that 
  \begin{equation}
    \label{eq:formula-covariances}
    \EE[W_N^{\beta_N}(x)W_N^{\beta_N}(y)] = \sum_{k=0}^{\infty} (\beta_N^2)^k \sum_{1\le n_1<\cdots<n_k\le N}
		\sum_{z_1,\ldots,z_k\in\Zeven^2}\prod_{i=1}^k q_{2(n_i-n_{i-1})}(z_i-z_{i-1}) h_N(z_i)
  \end{equation}
  with by convention \(n_0=0\), \(z_0=x-y\).

  \smallskip
  \textbullet\
  As far as the upper bound is concerned, we only deal with the case \(\hat{\beta}<1\). 
  We simply bound \(q_{2n}(z) \leq q_{2n}(0)\), which is valid for any \(z\in \Zeven^2\).
  Using also that \(h_N(z)\leq \e^{\beta_N^2} h(z)\), summing over \(z_1,\ldots, z_k\in \Z^2\) gives that 
  \[
  \EE[W_N^{\beta_N}(x)W_N^{\beta_N}(y)] \leq \sum_{k=0}^{\infty} \Big(\beta_N^2 \e^{\beta_N^2} \Sigma_h \Big)^k \sum_{1\le n_1<\cdots<n_k\le N}
		q_{2(n_i-n_{i-1})}(0) \leq \sum_{k=0}^{\infty} \Big(\beta_N^2 \e^{\beta_N^2} \Sigma_h R_N \Big)^k\,,
  \]
  where we have simply extended the sum to \(1\leq n_i-n_{i-1}\leq N\) and introduced \(R_N \coloneqq \sum_{\ell=1}^N q_{2\ell}(0)\).
  Now, using the scaling~\eqref{def:beta-N-summable} together with the fact that \(R_N \sim \frac{1}{\pi} \log N\) we get that
  \[
  \lim_{N\to\infty} \beta_N^2 \e^{\beta_N^2} \Sigma_h R_N = \hat\beta^2 \,.
  \]
  All together, we have that 
  \[
  \limsup_{N\to\infty} \sup_{x,y\in \Z^2} \EE[W_N^{\beta_N}(x)W_N^{\beta_N}(y)]  \leq  \frac{1}{1-\hat\beta^2} \,.
  \]
  This shows the upper bound in~\eqref{eq:limit-second-moment-summable} and also~\eqref{eq:covariances-summable}.

  \smallskip
  \textbullet\ We now turn to the lower bound.
  Let \((k_N)_{N\geq 0}\) be a sequence that grows to \(\infty\), with \(\log k_N \ll \log N\).
  We start from~\eqref{eq:formula-covariances} (with \(x=y\)), and we restrict the sum to \(k\leq k_N\) and \(1\leq \ell_i \coloneqq n_i-n_{i-1} \leq N/k_N\): using also that \(h_N(z)\geq h(z)\), we get that 
  \begin{equation}
    \label{eq:restrict-lower-bound}
    \begin{split}
    \EE \big[(W_N^{\beta_N})^2\big] 
    &\geq \sum_{k=0}^{k_N} (\beta_N^2)^k
		\sum_{1\le \ell_1,\ldots, \ell_k\le N/k_N}
		\sum_{z_1,\ldots, z_k \in \Zeven^2}
		\prod_{i=1}^{k}
		q_{2 \ell_i}(z_i-z_{i-1})\, h(z_i)  \\
    & \geq \sum_{k=0}^{k_N} (\beta_N^2)^k
		\sum_{z_1,\ldots, z_k \in \Zeven^2}
		\prod_{i=1}^{k}
		R_{N/k_N}(z_i-z_{i-1})\, h(z_i)  \,,
  \end{split}
  \end{equation}
  where we have introduced \(R_{m}(z) \coloneqq \sum_{\ell=1}^m q_{2\ell}(z)\) for any \(z\in \Z^2\), similarly as in \cite{CCD25}.

  Now, by the local CLT \cite[Thm.~1.2.1]{Law96}, we have that \(q_{2\ell}(z) \geq \frac{1 - c(\log N)^{-1}}{\pi \ell}\)  if \(|z|\leq \log N\) and \(\ell\geq (\log N)^3\).
  We thus have that for \(|z|\leq \log N\),
  \begin{equation}
    \label{eq:RN-z-lower}
    R_{N/k_N}(z) \geq \sum_{\ell= (\log N)^3}^{N/k_N} q_{2\ell}(z)  \geq (1- c (\log N)^{-1}) \sum_{\ell= (\log N)^3}^{N/k_N} \frac{1}{\pi \ell}  \geq (1-\gep_N) \frac{\log N}{\pi}\,,
  \end{equation}
  for some \(\gep_N\downarrow 0\), using also that \(\log k_N\ll \log N\) for the last inequality.
  Therefore, restricting the sum in~\eqref{eq:restrict-lower-bound} to \(|z_i|\leq \log N\) we get that 
  \begin{equation}
    \label{eq:lower-second-moment}
    \EE \big[(W_N^{\beta_N})^2\big] \geq \sum_{k=0}^{k_N} \Big( \beta_N^2 (1-\gep_N) \frac{\log N}{\pi} \sum_{z\in \Zeven^2,|z|\leq \log N} h(z)\Big)^k \,.
  \end{equation}
  To conclude the proof, notice that \(\lim_{N\to\infty} \beta_N^2 \frac{\log N}{\pi} \sum_{|z|\leq \log N} h(z) = \hat\beta^2\) thanks to the scaling~\eqref{def:beta-N-summable}.
  We thus get \(\liminf_{N\to\infty}	\EE[(W_N^{\beta_N})^2] \geq \frac{1}{1-\hat \beta^2}\) if \(\hat \beta<1\) and \(\liminf_{N\to\infty}	\EE[(W_N^{\beta_N})^2] =+\infty\) if \(\hat \beta\geq 1\), which concludes the proof of the lower bound in~\eqref{eq:limit-second-moment-summable}.
\end{proof}

\begin{remark}
  \label{rem:lower-bound}
  We stress that the proof of Claim~\ref{claim:lower-bound} in the summable case follows directly from~\eqref{eq:lower-second-moment} above.
  Indeed, the same lower bound as in~\eqref{eq:lower-second-moment} holds for \(\mathcal{V}_{N,\leq k_N}^{\beta_N}(0)\); the formula~\eqref{eq:formula-covariances} holds for \(\mathcal{V}_{N,\leq k_N}^{\beta_N}\) with the sum simply truncated at \(k\leq k_N\). 
  Since for any \(\gep>0\) fixed, we have \(\beta_N^2 \frac{\log N}{\pi} \sum_{|z|\leq \log N} h(z) \geq (1-\gep) \hat\beta^2\) for \(N\) large enough, the lower bound \eqref{eq:lower-second-moment} gives \(\mathcal{V}_{N,\leq k_N}^{\beta_N}(0) \geq ((1-\gep) \hat\beta^2)^{k_N}\), which concludes the proof of~\eqref{eq:mainclaim} and thus of Claim~\ref{claim:lower-bound}.
  The i.i.d.\ case is even simpler.
\end{remark}

\begin{remark}
  \label{rem:convergence}
  Let us observe that, in view of~\eqref{eq:covariances}, Lemma~\ref{lem:sec-mom-summable} shows a convergence of the Laplace transform of \(\mathcal{L}_N(S,S') \coloneqq \sum_{n=1}^N h(S_{n}-S_n') \stackrel{(d)}{=} \sum_{n=1}^N h(S_{2n})\).
  Indeed, letting \(C_h \coloneqq \Sigma_h/\pi = \mathfrak{C}_h^{-2}\) as in~\eqref{eq:scaling-L}, Lemma~\ref{lem:sec-mom-summable} shows that for any \(\hat\beta>0\)
  \[
  \lim_{N\to\infty} \bE\Big[ \e^{ \hat\beta^2 \frac{1}{C_h \log N} \sum_{n=1}^N h(S_{2n})} \Big] = \bE[\e^{ \hat\beta^2 Y}] \,,
  \]
  where \(Y\sim \mathrm{Exp}(1)\). 
  One thus deduces that \(\frac{1}{C_h \log N} \sum_{n=1}^N h(S_{2n}) \xrightarrow[N\to\infty]{(d)} Y\sim \mathrm{Exp}(1)\).
\end{remark}

\subsection{Proof of \texorpdfstring{\eqref{eq:scaling-L}}{}}
\label{sec:scaling-L}

We start from the formula
\begin{equation}
  \label{def:RN-x}
  \bE^{\otimes 2}[\mathcal{L}_N(S,S')] = \sum_{n=1}^N \bE[h(S_{2n})]  = \sum_{x\in \Zeven^2} h(x) R_N(x) \,,
\end{equation}
with \(R_N(x) \coloneqq \sum_{n=1}^N q_{2n}(x)\) the Green function as above.
Let us stress that the local limit theorem \cite[Thm.~1.2.1]{Law96} shows that \(q_{2n}(x) =  \frac{1}{\pi\ell} \e^{-|x|^2/\ell} + O(n^{-2})\).
This shows that 
\begin{equation}
  \label{eq:asymp-RN-x}
  R_N(x) = \sum_{\ell=1}^{N} \frac{1}{\pi \ell} \e^{- |x|^2/\ell} +O(1) = \frac{1}{\pi} \log \Big( \frac{N}{|x|^2}\Big) + O(1)\,.
\end{equation}
The second identity comes from dividing the sum according to whether \(n\leq |x|^2\) or \(n>|x|^2\), both terms being estimated as done in~\cite[Thm.~1.6.2]{Law96}.

\smallskip
\textbullet\ In the summable case~\eqref{finite-sum}, using also that \(R_N(x)\leq R_N(0) \le C \log N\) for all \(N\geq 1\), we have that for any fixed \(T>1\),
\[
\sum_{x\in \Zeven^2} h(x) R_N(x) = \sum_{x\in \Zeven^2,|x|\leq T} h(x) R_N(x) + \sum_{x\in \Zeven^2, |x|>T} h(x) O(\log N) \,.
\]
Thus, taking \(T_N= \log N\) so that \(\log T_N \ll \log N\), we get that \(R_N(x) = (1+o(1)) \frac{1}{\pi}\log N\) uniformly for \(|x|\leq T_N\), by~\eqref{eq:asymp-RN-x}.
Thus, we obtain
\[
\sum_{x\in \Zeven^2} h(x) R_N(x) = (1+o(1))\frac{1}{\pi}\log N \sum_{x\in \Zeven^2,|x|\leq T_N} h(x) + o(\log N) = (1+o(1))\frac{\Sigma_h}{\pi}\log N \,.
\]

\textbullet\ In the non-summable case~\eqref{def:h}, notice that we have (see e.g.\ \cite[App.~A, Eq.~(5)]{CCD25})
\begin{equation}
  \label{sum:h}
  \sum_{x\in \Zeven^2,|x|\leq T} h(x) \sim \frac{\pi}{a+1} (\log T)^{a+1} \qquad \text{ as } T\to\infty \,.
\end{equation}
Thus, applying~\eqref{eq:asymp-RN-x}, we get that
\begin{equation*}
  \begin{split}
  \sum_{|x|\leq \sqrt{N}} h(x) R_N(x) &  = \frac1\pi \sum_{|x|\leq \sqrt{N}}  h(x) \log N - \frac{2}{\pi} \sum_{|x|\leq \sqrt{N}}  h(x) \log |x| +  O\Big( (\log N)^{a+1} \Big) \\
  & = (1+o(1))\frac{2^{-(a+1)}}{a+1} (\log N)^{a+2} - (1+o(1))2 \cdot \frac{ 2^{-(a+2)}}{a+2}(\log N)^{a+2}  \,,
  \end{split}
\end{equation*}
where we have used~\eqref{sum:h} for the first sum and~\eqref{sum:h} with \(\tilde h(x) = h(x) \log |x|\) (\textit{i.e.}\ with \(a+1\) instead of \(a\) in~\eqref{def:h}) for the second sum.

Now, one can easily see that the remaining sum is negligible.
We simply bound \(h(x) \leq C (\log N)^a/N\) for \(|x|>\sqrt{N}\) so that 
\[
 \sum_{|x| > \sqrt{N}} h(x) R_N(x) \leq C \frac{(\log N)^a}{N} \sum_{x\in \Zeven^2} R_N(x) = C (\log N)^a \,,
\]
where we have simply used that \(\sum_{x\in \Zeven^2} R_N(x) = \sum_{n=1}^N \sum_{x\in \Zeven^2}q_{2n}(x)=N\).
All together, this proves~\eqref{eq:scaling-L} with the constant \(C_a= \frac{2^{-(a+1)}}{a+1} - \frac{ 2^{-(a+1)}}{a+2}\), as announced.
\qed

\section{Computation of \texorpdfstring{\(\nabla \log W_N^{\beta}\)}{PDFstring}}
\label{app-gradient}

We write here the proof of the formula~\eqref{eq:nabla-f}, \textit{i.e.}\ we compute the norm of the gradient $\nabla f_N(\hat \omega)$, with \(f_N(\hat \omega) = \log W_N^{\beta}\) and \(\omega\) is related to \(\hat \omega\) by~\eqref{omega-hat-omega}. 
Recall that
	\begin{equation*}
			f_N(\hat \omega) = \log \bE \Big[\, \e^{ \sum_{m=1}^N (\beta\omega(m,S_m) -\frac12\beta^2) }\,\Big]
      = \log \bE \Big[\, \e^{\beta \sum_{m=1}^N \big\lbrace \sum_{z\in\Z^2}h_0(S_m-z) \hat\omega(m,z)\big\rbrace-N\frac{\beta^2}{2}}\,\Big]\,,
	\end{equation*}
  using the identity~\eqref{omega-hat-omega} for the second identity.

	For $(n,y)\in \N \times \Z^2$ we first compute the derivative with respect to $\hat \omega(n,y)$:
	\begin{equation*}
		\frac{\partial f_N(\hat \omega)}{\partial \hat \omega(n,y)}=\frac{1}{W_N^\beta} \bE \Big[ \, \beta h_0(S_n-y) \, \e^{\beta\, \sum_{m=1}^N \big\lbrace \sum_{z\in\Z^2}h_0(S_m-z) \hat\omega(m,z)\big\rbrace-N\frac{\beta^2}{2}} \,\Big]\,,
	\end{equation*}
  with \(\sum_{z\in\Z^2}h_0(S_m-z) \hat\omega(m,z) = \omega(m,S_m)\).
	We therefore have
  \begin{equation*}
    \Big(\frac{\partial f_N(\hat \omega)}{\partial \hat \omega(n,y)}\Big)^2 = \frac{1}{(W_N^\beta)^2} \bE^{\otimes 2} \Big[ \, \beta^2\, h_0(S_n-y) h_0(S'_n-y)\, \e^{\sum_{m=1}^N (\beta\{ \omega(m,S_m)+\omega(m,S'_m) \}-  \beta^2)} \,\Big] \,,
  \end{equation*}
  where $S$ and $S'$ are two independent copies of the random walk.
  Now, using that \(h=h_0\ast h_0\), we end up with
  \begin{equation*}
      |\nabla f_N(\hat \omega)|^2 =\sum_{n=1}^N\sum_{y\in \Z^2} \Big(\frac{\partial f_N(\hat \omega)}{\partial \hat \omega(n,y)}\Big)^2 = \frac{\beta^2}{(W_N^\beta)^2} \sum_{n=1}^N \bE^{\otimes 2} \Big[h(S_n-S'_n)\, \e^{\sum_{m=1}^N (\beta\{ \omega(m,S_m)+\omega(m,S'_m) \}-  \beta^2)}\Big].
  \end{equation*}
  This proves~\eqref{eq:nabla-f}, recalling that $\mathcal{L}_N(S,S') \coloneqq\sum_{n=1}^N h(S_n-S_n')$.

\subsection*{Acknowledgements}
The authors are deeply grateful to Nicolas Bouchot, Clément Cosco, Anna Donadini, Gaspard Gomez and Maël Laoufi for helpful discussions, comments and advice, and for the exchanges that motivated this work. 

Q.B.\ acknowledges the support of Institut Universitaire de France and ANR Local (ANR-22-CE40-0012-02). 
F.C.\ acknowledges the support from the European Union’s Horizon 2020 research and innovation programme under the Marie Skłodowska-Curie grant agreement No.\ 101034255 and from INdAM/GNAMPA.

\end{appendix}

\bibliographystyle{alpha}
\bibliography{biblio.bib}

\begin{thebibliography}{CRW26}

\bibitem[AKQ14]{AKQ14a}
Tom Alberts, Konstantin Khanin, and Jeremy Quastel.
\newblock {The intermediate disorder regime for directed polymers in dimension $1+1$}.
\newblock {\em Ann. Probab.}, 42(3):1212--1256, 2014.

\bibitem[AY15]{AY15}
Kenneth Alexander and G{\"o}khan Yıldırım.
\newblock Directed polymers in a random environment with a defect line.
\newblock {\em Electron. J. Probab.}, 20:20 pp., 2015.

\bibitem[BCC26]{BCC26}
Nicolas Bouchot, Cl{\'e}ment Cosco, and Francesca Cottini.
\newblock Fluctuations of log-partition function with spatially correlated disorder.
\newblock {\em in preparation}, 2026+.

\bibitem[BCT25]{BCT25}
Quentin Berger, Francesco Caravenna, and Nicola Turchi.
\newblock Strong disorder for {S}tochastic {H}eat {F}low and {2D} directed polymers.
\newblock {\em preprint arXiv:2508.02478}, 2025.

\bibitem[BL17]{BL17}
Quentin Berger and Hubert Lacoin.
\newblock The high-temperature behavior for the directed polymer in dimension $1+2$.
\newblock {\em Annales de l'Institut Henri Poincar\'e, Probabilit\'es et Statistiques}, 53(1):430--450, 2017.

\bibitem[BN26]{BN26}
Quentin Berger and Shuta Nakajima.
\newblock Sharp behavior of the free energy for the two-dimensional directed polymer model.
\newblock {\em preprint arXiv:2605.30707}, 2026.

\bibitem[Bol89]{Bol89}
Erwin Bolthausen.
\newblock A note on the diffusion of directed polymers in a random environment.
\newblock {\em Communications in Mathematical Physics}, 123(4):529--534, 1989.

\bibitem[BTZ26]{BCZ26}
Quentin Berger, Nicola Turchi, and Nikos Zygouras.
\newblock Fractional moments of the {S}tochastic {H}eat {F}low and {2D} directed polymers.
\newblock {\em arXiv:2608.13359}, 2026.

\bibitem[CC22]{CC22}
Francesco Caravenna and Francesca Cottini.
\newblock Gaussian limits for subcritical chaos.
\newblock {\em Electronic Journal of Probability}, 27:1--35, 2022.

\bibitem[CCD25]{CCD25}
Cl{\'e}ment Cosco, Francesca Cottini, and Anna Donadini.
\newblock A central limit theorem for two-dimensional directed polymers with critical spatial correlation.
\newblock {\em arXiv:2509.16694}, 2025.

\bibitem[CD25]{CD25}
Cl{\'e}ment Cosco and Anna Donadini.
\newblock On the central limit theorem for the log-partition function of 2d directed polymers.
\newblock {\em Electronic Communications in Probability}, 30:1--11, 2025.

\bibitem[CG]{CG26}
Cl{\'e}ment Cosco and Gaspard Gomez.
\newblock work in progress, personal communication.

\bibitem[CG23]{CG23}
Yingxia Chen and Fuqing Gao.
\newblock Scaling limits of directed polymers in spatial-correlated environment.
\newblock {\em Electronic Journal of Probability}, 28:1--57, 2023.

\bibitem[CH02]{CH02}
Philippe Carmona and Yueyun Hu.
\newblock On the partition function of a directed polymer in a {G}aussian random environment.
\newblock {\em Probab. Theory Related Fields}, 124(3):431--457, 2002.

\bibitem[CH06]{CH06}
Philippe Carmona and Yueyuen Hu.
\newblock Strong disorder implies strong localization for directed polymers in a random environment.
\newblock {\em ALEA, Lat. Am. J. Probab. Math. Stat.}, 2:217--229, 2006.

\bibitem[Com17]{Com17}
Francis Comets.
\newblock {\em Directed Polymers in Random Environments}, volume 2175 of {\em Ecole d'Et{\'e} de probabilit{\'e}s de {S}aint-{F}lour}.
\newblock Springer International Publishing, 2017.

\bibitem[CR26]{CR26}
Junjie Cao and Guanglin Rang.
\newblock Scaling limit of $1+1$ dimensional directed polymer with power-law tail and spatial correlated noise.
\newblock {\em preprint arXiv:2607.09246}, 2026.

\bibitem[CRW26]{CRW26}
Junjie Cao, Guanglin Rang, and Jianglun Wu.
\newblock Disorder thresholds and free energy of brownian directed polymers with product and radial spatial correlations.
\newblock {\em arXiv:2608.29501}, 2026.

\bibitem[CSY03]{CSY03}
Francis Comets, Tokuzo Shiga, and Nobuo Yoshida.
\newblock Directed polymers in a random environment: strong disorder and path localization.
\newblock {\em Bernoulli}, 9(4):705--723, 2003.

\bibitem[CSZ17a]{CSZ17a}
Francesco Caravenna, Rongfeng Sun, and Nikos Zygouras.
\newblock Polynomial chaos and scaling limits of disordered systems.
\newblock {\em Journal of the European Mathematical Society}, 19:1--65, 2017.

\bibitem[CSZ17b]{CSZ17b}
Francesco Caravenna, Rongfeng Sun, and Nikos Zygouras.
\newblock Universality in marginally relevant disordered systems.
\newblock {\em The Annals of Applied Probability}, 27(5):3050--3112, 2017.

\bibitem[CSZ23]{CSZ23}
Francesco Caravenna, Rongfeng Sun, and Nikos Zygouras.
\newblock The critical 2d {S}tochastic {H}eat {F}low.
\newblock {\em Invent. Math.}, 233(1):325--460, 2023.

\bibitem[CSZ25]{CSZ25}
Francesco Caravenna, Rongfeng Sun, and Nikos Zygouras.
\newblock From disordered systems to the {C}ritical {2D} {S}tochastic {H}eat {F}low.
\newblock {\em preprint arXiv:2511.08479}, 2025.

\bibitem[CSZ26]{CSZ24-rev}
F~Caravenna, R~Sun, and N~Zygouras.
\newblock The critical 2d {S}tochastic {H}eat {F}low and related models.
\newblock {\em {CIME} Lecture Notes}, 2026+.

\bibitem[CTT17]{CTT17}
Francesco Caravenna, Fabio~Lucio Toninelli, and Niccol{\`o} Torri.
\newblock {Universality for the pinning model in the weak coupling regime}.
\newblock {\em Ann. Probab.}, 45(4):2154--2209, 2017.

\bibitem[CV06]{CV06}
Francis Comets and Vincent Vargas.
\newblock Majorizing multiplicative cascades for directed polymers in random media.
\newblock {\em ALEA, Lat. Am. J. Probab. Math. Stat.}, 2:267--277, 2006.

\bibitem[DHL25]{DHL25}
Alexander Dunlap, Martin Hairer, and Xue-Mei Li.
\newblock A critical stochastic heat equation with long-range noise.
\newblock {\em arXiv:2509.23790}, 2025.

\bibitem[HH85]{HH85}
D.~A. Huse and C.~L. Henley.
\newblock Pinning and roughening of domain walls in ising systems due to random impurities.
\newblock {\em Physics Review Letters}, 54:2708--2711, 1985.

\bibitem[Jan97]{Jan97}
Svante Janson.
\newblock {\em Gaussian Hilbert Spaces}.
\newblock Cambridge Tracts in Mathematics. Cambridge University Press, 1997.

\bibitem[JL26]{JL26}
Stefan Junk and Hubert Lacoin.
\newblock On the free energy of directed polymers in dimension \(d\geq 3\).
\newblock {\em in preparation}, 2026.

\bibitem[KT24]{KT24}
Sefika Kuzgun and Ran Tao.
\newblock Time-dependent averages of a critical long-range stochastic heat equation.
\newblock {\em arXiv:2411.09058}, 2024.

\bibitem[Lac10]{Lac10a}
H.~Lacoin.
\newblock New bounds for the free energy of directed polymer in dimension $1+1$ and $1+2$.
\newblock {\em Commun. Math. Phys.}, 294:471--503, 2010.

\bibitem[Lac11]{Lac11}
Hubert Lacoin.
\newblock Influence of spatial correlation for directed polymers.
\newblock {\em Ann. Probab.}, 39(1):139--175, 2011.

\bibitem[Lac25]{Lac25}
Hubert Lacoin.
\newblock The localization transition for the directed polymer in a random environment is smooth.
\newblock {\em arXiv:2505.13382}, 2025.

\bibitem[Law96]{Law96}
G.~F. Lawler.
\newblock {\em Intersection of random walks}.
\newblock Probability and its Applications. Springer Birkh{\"a}user, Boston, 1996.

\bibitem[Led01]{Led05}
M.~Ledoux.
\newblock {\em The concentration of measure phenomenon}, volume~89.
\newblock American Mathematical Society, 2001.

\bibitem[MT04]{MT04}
Carl Mueller and Roger Tribe.
\newblock A singular parabolic anderson model.
\newblock {\em Electron. J. Probab.}, 9:98--144, 2004.

\bibitem[Nak14]{Nak14}
Makoto Nakashima.
\newblock A remark on the bound for the free energy of directed polymers in random environment in 1+ 2 dimension.
\newblock {\em Journal of Mathematical Physics}, 55(9), 2014.

\bibitem[Nak19]{Nak19}
Makoto Nakashima.
\newblock Free energy of directed polymers in random environment in $1+ 1$-dimension at high temperature.
\newblock {\em Electron. J. Probab.}, 24:1--43, 2019.

\bibitem[Ran20]{Rang20}
Guanglin Rang.
\newblock From directed polymers in spatial-correlated environment to stochastic heat equations driven by fractional noise in 1+ 1 dimensions.
\newblock {\em Stoch. Process. Appl.}, 130(6):3408--3444, 2020.

\bibitem[Tsa24]{Tsai24}
Li-Cheng Tsai.
\newblock {S}tochastic {H}eat {F}low by moments.
\newblock {\em preprint arXiv:2410.14657}, 2024.

\bibitem[Zyg24]{Zyg24}
Nikos Zygouras.
\newblock Directed polymers in a random environment: a review of the phase transitions.
\newblock {\em Stochastic Processes and their Applications}, 177:104431, 2024.

\end{thebibliography}

\end{document}